\documentclass[a4paper,12pt,intlimits,oneside]{amsart}
\usepackage{amsmath}
\usepackage{amsthm}
\usepackage{latexsym}
\usepackage{amssymb}
\usepackage{xcolor}
\usepackage{graphicx}
\usepackage[colorlinks=true]{hyperref}
\usepackage{tikz}
\numberwithin{figure}{section}
\def\R{{\mathbb R}}
\def\C{{\mathbb C}}

\def\T{{\mathbb T}}
\def\Z{{\mathbb Z}}

\def\N{{\mathbb N}}

\def\e{\varepsilon}

\def\build#1_#2^#3{\mathrel{
\mathop{\kern 0pt#1}\limits_{#2}^{#3}}}
\def\td_#1,#2{\mathrel{\mathop{\build\longrightarrow_{#1\rightarrow #2}^{}}}}
\DeclareFontFamily{U}{MnSymbolC}{}
\DeclareSymbolFont{MnSyC}{U}{MnSymbolC}{m}{n}
\DeclareFontShape{U}{MnSymbolC}{m}{n}{
    <-6>  MnSymbolC5
   <6-7>  MnSymbolC6
   <7-8>  MnSymbolC7
   <8-9>  MnSymbolC8
   <9-10> MnSymbolC9
  <10-12> MnSymbolC10
  <12->   MnSymbolC12}{}
\DeclareMathSymbol{\intprod}{\mathbin}{MnSyC}{'270}
\newtheorem{theorem}{Theorem}
\newtheorem{corollary}{Corollary}
\newtheorem{proposition}{Proposition}
\newtheorem{lemma}{Lemma}
\newtheorem{remark}{Remark}
\newtheorem{definition}{Definition}
\begin{document}
\title[Spectrum of Lax operator of BO equation]{On the spectrum of the Lax operator of the Benjamin-Ono equation on the torus}
\author[P. G\'erard]{Patrick G\'erard}
\address{Laboratoire de Math\'ematiques d'Orsay, CNRS, Universit\'e Paris--Saclay, 91405 Orsay, France} \email{{\tt patrick.gerard@math.u-psud.fr}}
\author[T. Kappeler]{Thomas Kappeler}
\address{Institut f\"ur Mathematik, Universit\"at Z\"urich, Winterthurerstrasse 190, 8057 Zurich, Switzerland} 
\email{{\tt thomas.kappeler@math.uzh.ch}}
\author[P. Topalov]{Petar Topalov}
\address{Department of Mathematics, Northeastern University,
567 LA (Lake Hall), Boston, MA 0215, USA}
\email{{\tt p.topalov@northeatsern.edu}}

\subjclass[2010]{ 37K15 primary, 47B35 secondary}

%\date{September 29, 2021}

\begin{abstract}
We investigate the spectrum of the Lax operator $L_u$ of the Benjamin-Ono equation on the torus 
for complex valued potentials $u$ in the Sobolev space $H^{-s}(\T, \C)$, $0 \le s < 1/2$,
with small imaginary part and prove analytic properties of the moment map, defined in terms of
spectral data of $L_u$.
\end{abstract}

\keywords{Lax operator, Benjamin--Ono equation, moment map, 
spectral properties, complex valued potentials}

\thanks{
T.K. partially supported by the Swiss National Science Foundation.
P.T. partially supported by the Simons Foundation, Award \#526907.}

\maketitle

\tableofcontents

\medskip

\section{Introduction}\label{intro}
In this paper we consider the Lax operator $L_u$ of the Benjamin-Ono equation, introduced by Nakamura \cite{Nak},
on the torus $\T := \R / 2\pi \Z$.
It acts on the Hardy space 
$$
H_+:= \{ f \in L^2(\T, \C) \, : \,  \widehat f(n) := \frac{1}{2\pi} \int_0^{2\pi}  f(x)  e^{- i n x} = 0 \ \ \forall \, n < 0  \}
$$
and is given by $L_u = D - T_u$, where $D = -i \partial_x$ is the Fourier multiplier on $H_+$, 
$$
f = \sum_{n \ge 0} \widehat f(n) e^{i nx} \mapsto  D f (x)= \sum_{n \ge 0} n \widehat f(n) e^{ inx} \, , 
$$
$T_u$ is the Toeplitz operator on $H_+$ with potential $u$, 
$$
 f \mapsto \Pi(uf)\, ,
$$
and where $\Pi :  L^2(\T, \C) \to H_+$ denotes the Szeg\H{o} projector.
The aim of this paper is to analyze the spectrum of $L_u$ in the case $u$ is  a {\em complex valued} potential
in a neighborhood of the real Sobolev space $H^{-s}_{r}$ in $H^{-s}_{c}$ with $0 \le s < 1/2$,
where for any $\sigma \in \R$,
$$
H^{\sigma}_c  \equiv H^{\sigma}(\T, \C) :=  \{ f = \sum_{n \in \Z}  \widehat f(n)  e^{i n x} \, : \, \widehat f(n) \in \C \ \   \forall \, n \in \Z, \ \|f\|_{\sigma} < \infty \} 
$$
with
$$
\|f\|_{\sigma} := \Big( \sum_{n \in \Z} \langle n \rangle^{2\sigma}|\widehat f(n)|^{2}  \Big)^{1/2}\, , \qquad \langle n \rangle := \max\{1, |n|\} \, ,
$$
and
$$
H^{\sigma}_r  \equiv H^{\sigma}(\T, \R) :=  \{ f  \in H^{\sigma}_c \, : \, \widehat f(-n) = \overline{ \widehat f}(n) \,\, \forall \, n \ge 0  \}\, .
$$
Since for any $a \in \C$, $L_{u+ a} = L_u - a$, there is no loss of generality by restricting our attention to
 potentials $u$ of mean zero, i.e., potentials in  $H^{-s}_{c, 0}$ where for any $\sigma \in \R$
$$
H^{\sigma}_{c, 0} \equiv H^{\sigma}_0(\T, \C) := \{ u \in  H^{\sigma}_c  \, : \, \langle u | 1 \rangle = 0 \} \, .
$$ 
Here $\langle f | g \rangle$ denotes the standard $L^2-$inner product
$$
\langle f | g \rangle := \frac{1}{2\pi}\int_0^{2\pi} f(x) \overline g(x) dx\, , \qquad f, g \in L^2(\T, \C) ,
$$
extended by duality to a sesquilinear map  $H^{- \sigma}_c \times H^{\sigma}_c \to \C $.
To state our main results, we need to introduce some more notation.
For any $\sigma \in \R,$ let $H^{\sigma}_{r, 0} := H^{\sigma}_{c, 0} \cap H^{\sigma}_{r} $
and
$$
H_+^\sigma:= \{  f \in H^\sigma_c \, : \, \widehat f(n) = 0 \ \forall \, n < 0  \} \, .
$$ 
Furthermore, we denote
by $ \ell^{1, \sigma}(\N, \C)$ the weighted $\ell^1-$sequence space,
$$
 \ell^{1, \sigma}(\N, \C) := \{ z = (z_n)_{n \ge 1} \, : \,  z_n \in \C \ \forall \, n \in \N, \  \|z \|_{1, \sigma} < \infty \} \, ,
 $$
 with the norm
 $$
 \|z \|_{1, \sigma}:=  \sum_{n \ge 1} n^\sigma |z_n| \, .
$$
Our main results are the following ones.
\begin{theorem}\label{main theorem}
For any $0 \le s < 1/2,$ there exists a neighborhood $U^{-s}$ of $H^{-s}_{r,0}$ in $H^{-s}_{c,0}$
with the following properties:\\
(i) For any $u \in  U^{-s}$, the Lax operator $L_{u} = D - T_u$ defines an unbounded operator on $H^{-s}_+$
with domain $H^{1-s}_+$. It has compact resolvent and hence
its spectrum $\text{spec}(L_u)$ is discrete. It consists of
a sequence of simple eigenvalues $\lambda_n(u)$, $n \ge 0$, satisfying
$\lim_{n \to \infty}|\lambda_n - n | = 0$
and
$$
| \lambda_{n+1}(u) - \lambda_n(u)| > 1/2\, , \quad  |\Im \lambda_n(u)| < 1/4\, , \qquad \forall \,  n \ge 0\, .
$$
 Furthermore, $\lambda_n : U^{-s} \to \C$ is analytic for any $n \ge 0$. \\
(ii) For any $u \in  U^{-s}$, the gap lengths,
$$
\gamma_n(u) := \lambda_n(u) - \lambda_{n-1}(u) - 1, \qquad \forall \, n \ge 1\, ,
$$
satisfy $\Gamma(u) := (\gamma_n(u))_{n \ge 1} \in  \ell^{1, 1- 2s}(\N, \C)$
and the map $\Gamma : U^{-s} \to \ell^{1, 1- 2s}(\N, \C)$
is analytic.
\end{theorem}
\begin{remark}
(i) Since for any $u$ in $H^{-s}_{r, 0}$ with
$0 \le s < 1/2$,  $\gamma_n(u)$, $n \ge 1$, are action variables for the Benjamin-Ono equation
(cf. \cite{GK} $(s=0)$, \cite{GKT} $(0 < s < 1/2)$),
we refer to the map $\Gamma$ as the {\em moment map} of the Benjamin-Ono equation.\\
(ii) It has been shown in  \cite[Appendix C]{GK} that for any $u \in H^0_{r,0}$, $\gamma_n(u)$, $n \ge 1$,
are the lengths of the gaps of the spectrum of $L_u,$ when considered as an operator acting on the space of Hardy functions
on $\R$.
\end{remark}

For any $u \in U^{-s},$ denote by $\mathcal H_\lambda(u)$ the generating function
$$
\mathcal H_\lambda(u) := \langle (L_{u} - \lambda)^{-1} 1 | 1 \rangle \, , \qquad \lambda \in \C \setminus \text{spec}(L_{u})\, ,
$$
introduced for real valued potentials in \cite{GK} ($s=0$) and \cite{GKT} ($0 < s < 1/2$). 
(In contrast to  \cite{GK},  \cite{GKT}, it is convenient to define $\mathcal H_\lambda(u)$ with the opposite sign of $\lambda$ in this paper.)
For such potentials,
$\mathcal H_\lambda(u)$ can be thought of as a perturbation determinant of $ (L_{u}  - \lambda)^{-1}$ (cf. \cite[Appendix A]{GKT1}) 
and encodes $\text{spec}(L_{u})$ as well as the Benjamin-Ono hierarchy.

Theorem \ref{main theorem} allows to extend the product representation for $\mathcal H_\lambda(u)$, established for
potentials in $H^{-s}_{r,0}$ \big(cf. \cite[Section 3]{GK} (s=0) and \cite{GKT} ($0 < s < 1/2$)\big), to  $U^{-s}$.  
% cf. Section \ref{proof of main results near zero} and Section \ref{proof of main results} .
 \begin{corollary}\label{product representations}
  For any $u$ in $U^{-s}$, the following holds:\\
(i) The generating function $\mathcal H_\lambda(u)$ is a meromorphic function in $\lambda$ with poles contained in $\text{spec}(L_u)$
and admits the  product representation
$$
\mathcal H_\lambda(u) = \frac{1}{\lambda_0(u) - \lambda} \prod_{p \ge 1} \Big(1 - \frac{\gamma_p(u)}{\lambda_p(u) - \lambda} \Big) \, 
$$
where for any $\lambda \in \C \setminus \text{spec}(L_{u})$,  the infinite product is absolutely convergent.\\
(ii)   For any $n \ge 0,$
 \begin{equation}\label{trace formula for eigenvalues}
 \lambda_n(u) = n -  \sum_{k \ge n+1} \gamma_k(u)
 \end{equation}
 where the infinite sum is absolutely convergent.\\
 (iii)  If $s=0,$ one has
 $$
\frac{1}{2\pi} \int_0^{2\pi} u^2 dx = 2 \sum_{k \ge 1} k \gamma_k(u)
 $$
 where the infinite sum is absolutely convergent.
% (iv) For any $n \ge 1$, the function $\kappa_n(u)$  admits the  product representation
%$$
%\kappa_n(u) = \frac{1}{\lambda_n(u) - \lambda_0(u)}  \prod_{p \ne n} \Big(1 - \frac{\gamma_p(u)}{\lambda_p(u) - \lambda_n(u)} \Big)
%$$
%whereas for $n=0$ one has
%$$
%\kappa_0(u) =  \prod_{p \ge 1} \Big(1 - \frac{\gamma_p(u)}{\lambda_p(u) - \lambda_0(u)} \Big) \ .
%$$
%Both infinite products are absolutely convergent.\\
 \end{corollary}
 
 \medskip
 
 \noindent
  {\em Ausgangslage.} 
% {\em Starting situation.} 
 It follows from  \cite{GK} $(s=0)$ and \cite{GKT} $(0 < s < 1/2)$ that for any $u \in H^{-s}_{r, 0}$, the operator $L_{u}$
is an unbounded operator on $H_+^{-s}$ with domain $H^{1-s}_+$ and compact resolvent. Its spectrum $\text{spec}(L_{u})$ consists of a sequence 
of simple real eigenvalues $\lambda_0(u) < \lambda_1(u) < \cdots $, with the property that $\gamma_n(u)  \ge 0$ for any $n \ge 1$. 
Since  $\text{spec}(L_{u})$ is simple, it is possible to introduce 
a canonically normalized basis of eigenfunctions  $f_n(\cdot) \equiv f_n(\cdot, u)$, $n \ge 0$, corresponding to the eigenvalues $\lambda_n(u)$
(cf. \cite{GK} (s=0), \cite{GKT} ($0 < s < 1/2$)).
Furthermore it was established in the latter papers that for any $0 \le s < 1/2$, the image of the moment map
$$
\Gamma: H^{-s}_{r, 0} \to \ell^{1, 1-2s}(\N, \R), u \mapsto (\gamma_n(u))_{n \ge 1}
$$ 
is  the positive quadrant $\mathcal Q^{-s}_+ := \ell^{1, 1-2s}(\N, \R_{\ge 0})$ 
of $\ell^{1, 1-2s}(\N, \R)$ and hence a noncompact infinite-dimensional convex polytope. 
The convexity of $\mathcal Q^{-s}_+$ is an instance
of the convexity of the image of the moment map in an infinite-dimensional setting.
The theory of the convexity of the image of the moment map in the finite-dimensional compact case
originated in the papers of Atiyah \cite{At} and Guillemin and Sternberg \cite{GS}.

\smallbreak

Remarkably, the trace formula \eqref{trace formula for eigenvalues} allows to express the eigenvalues of $L_{u}$
as {\em affine} functions of the gap lengths.  Hence the positive quadrant $\mathcal Q^{-s}_+$ is a {\em moduli space} 
for the spectrum of the operator $L_{u}$ with $u \in H^{-s}_{r, 0}$. We point out that for the
Lax operator of the KdV equation (cf. \cite{KP}) as well as for the one of the defocusing NLS equation (cf. \cite{GK1}),
moduli spaces for their spectra are also given by the sequences of gap lengths, but that the eigenvalues are 
{\em transcendental} functions of the gap lengths.

\smallbreak

The moment map is related to the map (cf. \cite{GK}, \cite{GKT})
$$
F: H^{-s}_{r, 0} \to \ell^{1, 2-2s}(\N, \R), \, u \mapsto (F_n(u))_{n \ge 1}\, , \quad  F_n(u):= | \langle 1 |  f_n \rangle |^2\, ,
$$
by
$$
F_n(u) = \kappa_n(u) \gamma_n(u) \, , \qquad \forall \, n \ge 1\, ,
$$
where $\kappa_n(u) $ are scaling factors, admitting product expansions in terms of $\gamma_k(u), k \ge 1$.
The map $F$ is continuous and is referred to as the quasi-moment map of the Benjamin-Ono equation.

For potentials $u$ with $\Im u \ne 0$, the methods developed in \cite{GK}, \cite{GKT} do not allow 
to prove that the eigenvalues $\lambda_n$, $n \ge 0$, of $L_u$ extend analytically to a neighborhood  of $H^{-s}_{r,0}$ in  $H^{-s}_{c,0}$, which is independent of $n$.
Furthermore, since the eigenvalues and hence the gap lengths might be complex valued, the trace formula
$\lambda_0(u)= - \sum_{n \ge 0} \gamma_n(u)$ no longer implies that $(\gamma_n(u))_{n \ge 0}$ is an $\ell^1$ sequence.
Therefore new arguments are required to prove decay properties of $\gamma_n$ as $n \to \infty$.
\smallskip

\noindent 
{\em Ideas of the proofs.}  Our analysis of the operator $L_u$ is based on perturbation theory. 
Assume that $w$ is a real valued potential in $H^{-s}_{r,0}$ with $0 \le s < 1/2$. For $u$
in a (sufficiently small) neighborhood $U_w^{-s} \subset H^{-s}_{c,0}$ of $w$,  we view $L_u$ as a perturbation of $L_w$ and
 will use the full strength of the study of the Lax operator with real valued potentials, established in \cite{GK} and \cite{GKT}.
In the case $w=0$, the operator $L_{0} = D$ is particularly simple and hence we treat this case first. Recall that $D$ is a Fourier multiplier
with the property that $D$ is an unbounded operator on $H^{-s}_+$ with domain $H^{1-s}_+$. It has a compact resolvent and its spectrum $\text{spec}(D)$ consists 
of the eigenvalues $\lambda_n(0) = n$, $n \ge 0$, all of them being simple.
For any $n \ge 0$, the canonically normalized eigenfunction $ f_n(\cdot, 0)$, corresponding to $\lambda_n(0)$, is given by $e^{inx}$.
In Section \ref{Lax operator} we show that for any $u$ in a sufficiently small neighborhood $U_0^{-s}$
of $0$ in $H^{-s}_{c,0}$ and for any $\lambda \in \C,$ $L_u -\lambda$ can be defined 
as an unbounded linear operator on $H^{-s}_+$ with domain $\text{dom}(L_u - \lambda) = H^{1-s}_+.$ In addition, $L_u -  \lambda $ can be viewed as a bounded linear operator 
$L_u - \lambda : H^{1-s}_+ \to H^{-s}_+$. To analyze the set of $\lambda'$s, for which $L_u - \lambda$ is invertible, we write $L_u - \lambda$ as
$$
L_u - \lambda = (\text{Id} - T_u (D- \lambda)^{-1})(D - \lambda)\, , \qquad \, \forall \lambda \in \C \setminus \Z_{\ge 0}\, ,
$$ 
and estimate the size of the operator $T_u (D- \lambda)^{-1} : H^{-s}_+ \to H^{-s}_+$. For $u$ in a sufficiently small neighborhood $U_0^{-s}$
of $0$ in $H^{-s}_{c,0}$, we show that 
$L_u - \lambda$ is invertible for any $\lambda$ with $\text{dist}(\lambda,  \Z_{\ge 0}) \ge 1/4$ and that $L_{u}$ has compact resolvent with
spectrum $\text{spec} (L_{u})$ consisting of simple eigenvalues $\lambda_n(u) \in \C$, $n \ge 0$, satisfying $| \lambda_n(u) - n |  < 1/4$.
It follows that the eigenvalues $\lambda_n$, $n \ge 0$, and hence the gaps $\gamma_n = \lambda_n - \lambda_{n-1} - 1$, $n \ge 1$, are analytic functionals on $U_0^{-s}$.
%To show that for any $u \in U_0^{-s}$, $\Gamma(u)= (\gamma_n(u))_{n \ge 1}$ is in $\ell^{1, 1-2s}$
The key step of the proof of the claimed properties of the moment map on $U_0^{-s}$ is to show that the quasi-moment map $F$ admits an analytic extension,
$ F: U_0^{-s} \to \ell^{1, 2-2s}(\N, \C)$ -- see Section \ref{Quasi-moment map}. To this end, we express $F_n$, $n \ge 1$, in terms of the Riesz projector
$P_{n} : H^{-s}_+ \to H^{1-s}_+$, $n \ge 0$, given by the contour integral
$$
P_{n}(u) = - \frac{1}{2\pi i} \int_{|\lambda - n| = 1/3} (L_{u} - \lambda)^{-1} d \lambda\, .
$$
%They are analytic maps, defined on $U_0$, with values in the space of bounded linear operators on $H_+$.
The claimed estimates for $F$ are then obtained by a somewhat involved analysis, using the expansion of
the operator $(\text{Id} - T_u (D- \lambda)^{-1})^{-1}$, appearing in $(L_u - \lambda )^{-1} = (D - \lambda)^{-1} (\text{Id} - T_u (D- \lambda)^{-1})^{-1}$,
in its Neumann series.

\smallskip
In Section \ref{Lax operator.II} we analyze $L_u$ for $u$ in a (sufficiently small) neighborhood $U_w^{-s} \subset H^{-s}_{c,0}$ of an arbitrary real valued potential 
$w$ in $H^{-s}_{r,0}$ with $0 \le s < 1/2$.  Our novel approach is based on the use of {\em finite gap potentials}, which were studied in detail in \cite{GK}.  
A potential $w$ in $H^{0}_{r,0}$ is said to be a finite gap potential if the set $ \{ n \ge 1\, : \, \gamma_n(w) \ne 0 \}$
is finite. Finite gap potentials are contained in $\cap_{m \ge 0} H^{m}_{r,0}$. Let $\mathcal U_0 := \{ 0 \}$ and for any given $N \ge 1$, 
denote by $\mathcal U_N$ the set of real valued finite gap potentials $w$
with $\gamma_N(w) \ne 0$ and $\gamma_n(w) = 0$ for any $n > N$. For any $0 \le s < 1/2,$ $\cup_{N\ge 0} \, \mathcal U_N$ is dense in  $H^{-s}_{r,0}$.
A key property of finite gap potentials is that for any  $w \in \mathcal U_N$, the eigenfunctions $f_n$ of $L_w$ with $n \ge N$ are given by 
\begin{equation}\label{formula eigenfunctions} 
f_n(x, w) = g_\infty(x, w) e^{inx}\, ,  \ \ \forall \, n \ge N,  \qquad g_\infty(x,w):= e^{i \partial_x^{-1}w(x)}.
\end{equation}
We remark that $g_\infty(x, w)$ is a solution of $-i\partial_x g - w g = 0$ and refer to 
Appendix \ref{asymptotics eigenfunctions} for a discussion of the asymptotics of the eigenfunctions $f_n(\cdot, w)$
for arbitrary elements $w$ in $H^{-s}_{r,0}$, $0 \le s < 1/2$. 
By approximating $w \in H^{-s}_{r,0}$ by an appropriately chosen finite gap potential $w_0$ with $w_0 \in \mathcal U_N$ for some $N \ge 1$, the formulas \eqref{formula eigenfunctions} 
for the eigenfunctions $f_n(\cdot, w_0)$ with $n \ge N$ allow us to apply the methods  developed to analyze $L_u$ for
 $u \in H^{-s}_{c,0}$ near $0$ to obtain corresponding results for $L_u$ with $u \in U_w^{-s}$.

 \smallskip

\noindent
{\em Subsequent work.} Besides being of interest in its own right, our study of the spectrum of the Lax operator $L_u$ 
for complex valued potentials  is motivated by our ongoing research project concerning the Birkhoff map $\Phi$ of the Benjamin-Ono equation, 
constructed in \cite{GK} on $L^2_{r,0}$ and then extended in \cite{GKT} to $H^{-s}_{r,0}$ for any $0 < s < 1/2$.
Theorem \ref{main theorem} allows us to prove that $\Phi$ is not only a homeomorphism, 
but a (real analytic) diffeomorphism,  $\Phi: H^{-s}_{r,0} \to h^{1/2 -s}_+$ (cf. \cite{GKT2}, \cite{GKT3}). 
The latter result will be applied to further analyze regularity properties of the solution map
of the Benjamin-Ono equation  and will serve as the starting point for studying (Hamiltonian) perturbations of the Benjamin-Ono equation by KAM methods.
Due to the nonlocal nature of the operator $L_u$, to prove that the map $\Phi$ is a real analytic diffeomorphism is considerably more difficult than
to prove the corresponding result for the Birkhoff map of the KdV equation (cf. \cite{KP}) or of the defocusing NLS equation (cf. \cite{GK1}).
A key ingredient into the proof of the analyticity of $\Phi$ is that 
%$\Gamma$ maps  $U^{-s}$ into  $\ell^{1, 1-2s}_+(\N, \C)$ and that 
the moment map $\Gamma : U^{-s} \to \ell^{1, 1-2s}_+(\N, \C)$ is analytic.
 
 \smallskip
 
 \noindent
{\em Organisation.} Section \ref{Lax operator} -  Section \ref{proof of main results near zero} are concerned with the proofs of
the results of  Theorem \ref{main theorem} and Corollary \ref{product representations}
for complex valued potentials $u$ {\em near zero}.  In addition, in Section \ref{proof of main results near zero} we discuss
 properties of the scaling factors $\kappa_n$ and the
 normalizing factors $\mu_n$, introduced in \cite{GK}.
 In Section \ref{Lax operator.II} - Section \ref{proof of main results} we analyze the spectrum of the Lax operator $L_u$ for $u\in H^{-s}_{c,0}$ with $\Im u$ small  
 and prove Theorem \ref{main theorem} and Corollary \ref{product representations}.
 In Appendix \ref{asymptotics eigenfunctions} we prove asymptotic estimates for canonically normalized eigenfunctions $f_n(\cdot, u)$ of $L_u$
 for real valued potentials $u \in H^{-s}_{r,0}$ with $0 \le s < 1/2$. Finally, in Appendix \ref{normally analytic} we review the notion of normally analytic maps.
 
 \bigskip
 
 \noindent
{\em Notation.} By and large, we use the notations introduced in \cite{GK} and \cite{GKT}.
For the convenience of the reader, we list the most frequently used ones in this paragraph.

For any $s \in \mathbb R$, denote by $H^s_c$ the Sobolev space $H^s_c = H^s(\mathbb T, \mathbb C)$ where
$\mathbb T = \mathbb R / 2\pi \mathbb Z$. For $s=0$, we also write $L^2_c$ for $H^0_c$.
Furthermore, $H^s_{c, 0}$ denotes the subspace of $H^s_c$ of elements with average zero.
For $s=0$, we also write $L^2_{c, 0}$ for $H^0_{c, 0}$. 

For any $s \in \mathbb R$, denote by $H^s_r$ the Sobolev space $H^s_r = H^s(\mathbb T, \mathbb R)$.
For $s=0$, we also write $L^2_r$ for $H^0_r$. Similarly, $H^s_{r, 0}$ is the subspace of $H^s_r$
of elements in $H^s_r$ of average zero. For $s=0$, we also write $L^2_{r, 0}$ for $H^0_{r, 0}$. 

The open ball in $H^{-s}_{c, 0}$ of radius $r > 0$, centered at $0$, is denoted by $B^{-s}_{c, 0}(r) \equiv B^{-s}_{c,0}(0, r)$.

We denote by $\langle f | g \rangle := \frac{1}{2\pi} \int_{0}^{2\pi} f \overline g d x$ the $L^2-$inner product on $L^2_c$ and 
by $\langle f , g \rangle$ the nondegenerate bilinear form on $L^2_c \times L^2_c$, given by 
$\langle f , g \rangle := \frac{1}{2\pi} \int_{0}^{2\pi} f  g d x$, which is extended to $H^{-s}_c \times H^s_c$ for any $s \in \R$ by duality.

For any $f \in H^s_c$ and $k \in \Z$, denote by $\widehat f(k) = \langle f(x) |  e^{ikx} \rangle$ the $k$th Fourier coefficient of $f$.

For any $s \in \R,$ we denote by $H^s_+$ the subspace of the Sobolev space $H^s_c$,
consisting of elements $f$ in $H^s_c$ with $\widehat f(k) = 0$ for any $k < 0$. 
For $s=0$ we also write $H_+$ instead of $H^0_+$.

For any $m \in \Z$, let  $ \Z_{\ge m}:= \{ k \in \Z \ : \ k \ge m \}$ and set  $\N_0 := \Z_{\ge 0}$
as well as $\N := \Z_{\ge 1}$.

The standard $\ell^p$ sequence spaces are denoted as follows: $\ell^p_+ \equiv \ell^p(\N, \C)$ and more generally, $\ell^p_{\ge m} \equiv \ell^p(\Z_{\ge m}, \C)$ ($1 \le p \le \infty$).
The standard weighted sequence spaces, needed in this paper, are  $h^s_+ \equiv h^s(\N, \C)$ ($s \in \R$) and  $\ell^{1, s}_+ \equiv \ell^{1, s}(\N, \C)$ ($s \in \R$).
In case $s=0$, we also write $\ell^2_+$ for $h^0_+$.

%%%%%%%%%%%%%%%%%%%%%%%%%%%%%%%%%%%%%%%%%%%%%%%%%%%%%%%%%%%%%%%
%%%%%%%%%%%%%%%%%%%%%%%%%%%%%%%%%%%%%%%%%%%%%%%%%%%%%%%%%%%%%%%%

\section[Lax operator I]{Proof of Theorem \ref{main theorem}$(i)$ for $u$ small}\label{Lax operator}

In this section, we define the Lax operator $L_u$ for complex valued potentials $u \in H^{-s}_{c,0}$ with $0 \le s < 1/2$
and obtain some general results about its spectrum.
Our main goal is to study spectral properties of $L_u$ for $u$ {\em near the zero potential} and to prove Theorem \ref{main theorem}$(i)$
for such potentials. In Section \ref{Lax operator.II} 
we will consider the case of potentials $u$ with $\Im u$ small.
%In this case the analysis of the spectrum of $L_u$ turns out to be quite elementary and hence we present it in a separate section. 

We begin with some preliminary considerations. For any $0 \le s < 1/2$, define
$$
\sigma := (s + 1/2)/2 > 0\, .
$$
Note that $1/4 \le \sigma < 1/2$
\begin{lemma}\label{multiplication u f}
For any $0 \le s < 1/2$, there exists a constant $C_s \ge 1$ so that the following holds:\\
(i) For any $f  \in H^{1-\sigma}_c$, $g\in H^{s}_c$, the function $f g$ is in $H^{s}_c$ and
$$
\| f g \|_{s} \le C_s \|f\|_{1-\sigma} \|g \|_{s} \,, \qquad \forall g \in H^s_c, f \in H^{1-\sigma}_c  .
$$
(ii) For any $u \in H^{-s}_c$ and $f \in H^{1-\sigma}_c$, the product $uf$ defines an element in $H^{-s}_c$, determined by
$$
\langle g | u f \rangle : = \langle g \overline f | u \rangle \, , \qquad \forall g \in H^{s}_c  \, .
$$
Furthermore,
$$
\| u f \|_{-s} \le C_s  \| u \|_{-s} \|f \|_{1-\sigma}\, , \qquad \forall u \in H^{-s}_c, \ f \in H^{1-\sigma}_c .
$$
Since $H^{1-s}_c$ embeds into $H^{1-\sigma}_c$, the estimates in items (i) and (ii)  hold in particular
when $\sigma$ is replaced by $s$.
\end{lemma}
\begin{proof}
(i) Since $1-\sigma > 1/2$, $H^{1-\sigma}_c$ acts by multiplication on itself and on $L^2_c.$ 
By interpolation it then follows that it acts on $H^s_c$ and that there exists
$C_s \ge 1$ so that 
$$
\| f g \|_{s} \le C_s \|f\|_{1-\sigma} \|g \|_{s}\, , \qquad \forall g \in H^{s}_c, f \in H^{1-\sigma}_c \, .
$$
(ii) By item (i) it follows that for any $u \in H^{-s}_c$, $g \in H^s_c$, and $f \in H^{1-\sigma}_c$,
$g\overline f$ is in $H^s_c$ and hence that $u f$ defines a bounded linear
functional on $H^s_c$, determined by 
$$
\langle g | u f \rangle : = \langle g \overline f | u \rangle \, , \qquad \forall g \in H^{s}_c\, .
$$
Furthermore, it follows that $|\langle g | u f \rangle | \le \|u\|_{-s} \| g \overline f \|_{s}$
and hence by (i)
$$
\| uf \|_{-s} \le C_s \|u\|_{-s} \|f \|_{1-\sigma}\, .
$$
\end{proof}
Lemma \ref{multiplication u f} implies that for any $0 \le s < 1/2$, the Toeplitz operator $T_u$
with potential $u \in H^{-s}_c$ is a well defined bounded linear operator,
$$
T_u : H^{1-s}_+ \to H^{-s}_+, \ f \mapsto  \Pi(uf) 
$$
and so is $L_u$,
$$
L_u : H^{1-s}_+ \to H^{-s}_+, \ f \mapsto  Df - T_uf 
$$
where $D = -i \partial_x$ 
%denotes the Fourier multiplier
%$$
%D:  H^{1-s}_+ \to H^{-s}_+ : \, \sum_{k \ge 0} \widehat f(k) e^{ikx}  \mapsto \sum_{k \ge 0} k \widehat f(k) e^{ikx}
%$$
and where $\Pi$ is the Szeg\H{o} projector, 
$$
 \Pi : H^{-s}_c \to H^{-s}_+ : \ f \mapsto \sum_{k \ge 0} \widehat f(k) e^{ikx} \,  .
$$
Alternatively, we view $L_u$ as an unbounded operator $L_u : H^{-s}_+ \to H^{-s}_+$ with domain $H^{1-s}_+$.

 In a first step, we determine half planes in $\C$ consisting of complex numbers $\lambda \in \C$ 
 with the property that the operator $L_{u} - \lambda : H^{1-s}_+ \to H^{-s}_+$ is invertible
 for any $u$ in a given bounded subset of $H^{-s}_{c,0}$ with $0 \le s < 1/2$.
  Note that the spectrum of the operator $D: H_+ \to H_+$ with domain $H^1_+$ 
is discrete, consisting of the sequence of simple eigenvalues $\lambda_n = n$, $n \ge 0$.
It follows that for any $\lambda \in \C \setminus \N_0$,
$$
D - \lambda : H^{1-s}_+ \to H^{-s}_+
$$
is a bounded invertible operator with bounded inverse, $(D - \lambda)^{-1} : H^{-s}_+ \to H^{1-s}_+$
and we can write
$$
L_u - \lambda = \big( {\rm Id} - T_u (D - \lambda)^{-1} \big) (D - \lambda).
$$
Recall that $\N_0$ denotes the set of nonnegative integers.
To show that $L_{u} - \lambda : H^{1-s}_+ \to H^{-s}_+$ is invertible, it thus suffices to show that
the operator norm of $T_u (D - \lambda)^{-1} :  H^{-s}_+ \to  H^{-s}_+$ is smaller than $1$
where the latter norm is defined with respect to a conveniently chosen norm on $H^{-s}_+$ which might depend on $\lambda$. 
 For any $0 \le s < 1/2$ and $M> 0$, define
 \begin{equation}\label{definition K_M}
 K_M := (2C_s M)^{2/(\frac{1}{2} - s)} + 1
 \end{equation}
 where $C_s$ is the constant given in Lemma \ref{multiplication u f}.
 \begin{lemma}\label{resolvent set 1}
 Let $M > 0$ and $0 \le s < 1/2$. Then the following holds:\\
 (i) For any $\lambda \in \C$ in the half plane $\{ \Re \lambda \le - K_M \}$ and any $u \in H^{-s}_{c,0}$ with $\|u\|_{-s} \le M,$
 the operator $L_u - \lambda : H^{1-s}_+ \to H^{-s}_+$ is invertible.\\
 (ii) For any $u \in H^{-s}_{c,0}$ with $\|u\|_{-s} \le M$
 and any $\lambda \in \C$ in the half planes $\{\Im \lambda \ge  \Re \lambda + 2 K_M \}$, $\{\Im \lambda \le  - \Re \lambda - 2 K_M \}$,
 the operator $L_u - \lambda : H^{1-s}_+ \to H^{-s}_+$ is invertible.
 \end{lemma}
 \begin{proof}
(i) For any $\lambda \in \C \setminus \N_0$ and $f \in H^{-s}_+$, we have by Lemma \ref{multiplication u f}
 $$
 \| T_u (D - \lambda)^{-1} f  \|_{-s}= \| \Pi (u (D - \lambda)^{-1} f ) \|_{-s} \le C_s \|u\|_{-s} \| (D - \lambda)^{-1} f \|_{1- \sigma}.
 $$
 Note  that 
 $$
 \| (D - \lambda)^{-1} f \|_{1- \sigma}^2 = \sum_{k \ge 0} \frac{ \langle k \rangle^{2(1-\sigma)}}{| k - \lambda |^2} |\widehat f(k)|^2 .
 $$
 Since for any $\lambda$ with $\Re \lambda \le -K_M$ , one has 
 $$
 |k - \lambda | \ge K_M \, , \qquad   |k - \lambda | \ge \langle k \rangle \, , \qquad \forall \, k \in \N_0 \, ,
 $$
 it follows that 
 $$
  \| (D - \lambda)^{-1} f \|_{1- \sigma}^2 \le \frac{1}{K_M^{2(\sigma - s)}} \sum_{k \ge 0}\frac{1} {\langle k \rangle^{2s}} |\widehat f(k)|^2 
  \le  \frac{1}{K_M^{2(\sigma - s)}} \| f\|_{-s}^2 \, .
 $$
 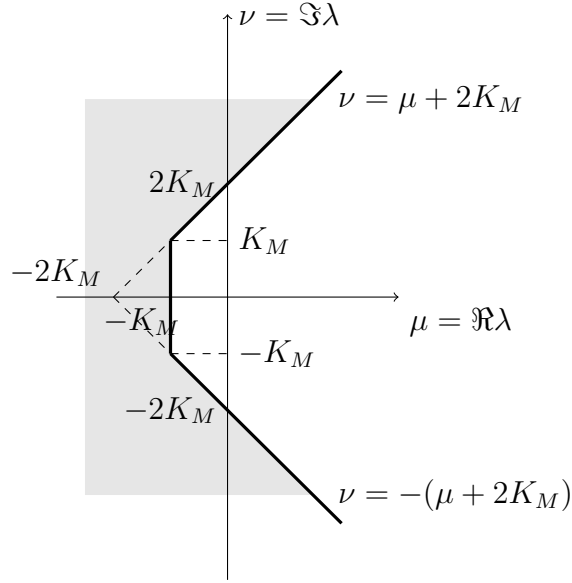
\begin{figure}
\begin{tikzpicture}[scale=0.75]
\fill [color=gray!20] (1.5,3.5) -- (-2.5,3.5) -- (-2.5,-3.5) -- (1.5,-3.5) -- (-1,-1) -- (-1,1) -- cycle ;
\draw (-2,0) node[above left] {$-2K_M$} ;
\draw (-1.5,0) node[below] {$-K_M$} ;
\draw (0,2) node[left] {$2K_M$} ;
\draw (0,-2) node[left] {$-2K_M$} ;
\draw  [very thick]  (-1,-1) -- (-1,1) ;
\draw  [very thick]  (-1,1) -- (2,4)  ;
\draw  [very thick]  (-1,-1) -- (2,-4) ;
\draw  [->]  (-3,0) -- (3,0) ;
\draw  [->]   (0,-5) -- (0,5) ;
\draw  [dashed]  (-1,1) -- (0,1) node[right]  {$K_M$}  ;
\draw  [dashed]   (-1,-1) -- (0,-1) node[right]  {$-K_M$}  ;
\draw [dashed] (-2,0) -- (-1,1) ;
\draw [dashed] (-2,0) -- (-1,-1) ;
\draw (3,0) node[below right] {$\mu=\Re \lambda$} ;
\draw (0,5) node[right] {$\nu=\Im \lambda$} ;
\draw (1.75,3.5) node[right] {$\nu =\mu +2K_M$} ;
\draw (1.75,-3.5) node[right] {$\nu =-(\mu +2K_M)$} ;
\end{tikzpicture}
\caption{ For $\lambda$ in the grey area, $L_u -\lambda$ is invertible.}
\end{figure}
Combining the estimates obtained so far and using that $\sigma - s = (1/2 - s)/2$
one infers from the definition of $K_M$ that for any $u \in H^{-s}_{c,0}$ with $\|u\|_{-s} \le M$
 $$
  \| T_u (D - \lambda)^{-1} f  \|_{-s} \le \frac{C_s M}{K_M^{(\frac{1}{2} - s)/2}} \|f\|_{-s}  \le \frac{1}{2} \| f \|_{-s}\, ,
  $$
  implying that $L_u - \lambda$ is invertible for any $\lambda \in \C$ with $\Re \lambda \le - K_M$.\\
  (ii) In view of item (i) it suffices to show that  $L_u - \lambda$ is invertible for any $\lambda \in \C$
  with $\mu:= \Re \lambda \ge - K_M$ and $\pm \nu := \Im \lambda \ge \mu + 2 K_M$.
  Since for such $\lambda$, 
  $$
   | k -  \lambda |^2 = (k - \mu)^2 + \nu^2 \ge (k - \mu )^2 + (\mu + 2 K_M  )^2, 
   $$
   arguing as in item (i), 
  it suffices to verify that for any $k \in \N_0$ and $\mu \ge - K_M,$
  $$
  (k - \mu )^2 + (\mu + 2 K_M  )^2 \ge  \langle k \rangle^2/2  , \qquad (k - \mu )^2 + (\mu + 2 K_M  )^2 \ge  K_M^2 \, .
  $$
  First note that the latter inequality clearly holds for  such $k$ and $\mu$ and that the first one holds for $k=0$.
  Finally, for any $k \in \N,$ one has
  $$
   (k - \mu )^2 + (\mu + 2 K_M  )^2  -  k^2/2 = \frac{1}{2}\big( k - 2\mu \big)^2  + 4K_M (\mu + K_M) \ge 0 \, .
  $$
  It then follows that for any $u \in H^{-s}_{c,0}$ with $\|u\|_{-s} \le M$ and $\lambda$ as above
 $$
  \| T_u (D - \lambda)^{-1} f  \|_{-s} \le \frac{\sqrt{2}C_s M}{K_M^{(\frac{1}{2} - s)/2}} \|f\|_{-s}  \le \frac{\sqrt{2}}{2} \| f \|_{-s}\, ,
  $$
  implying that $L_u - \lambda$ is invertible for any $\lambda \in \C$ in either of the two half planes.
 \end{proof}

In the remainder of this section we study the operator $L_u$ for $u \in H^{-s}_{c,0}$ near zero.
First we need to introduce some more notation.
For any $m \in \Z$ and $s \in \R,$ define the sequence space
$$
h^s_{\ge m} \equiv  h^s(\Z_{\ge m}, \C ) :=  \{ z=(z_k)_{k \ge m} \subset \C \ : \ \|z \|_s < \infty \}  , 
$$
$$
\|z\|_s := \big( \sum_{k=m}^\infty \langle k \rangle^{2s} |z_k|^2  \big)^{1/2} , 
$$
and for any $n \in \Z$, the $n-$shift operator $\mathcal S_n : h^s_{\ge 0} \to h^s_{\ge -n}$, given by
\begin{equation}\label{n-shift}
z=(z_k)_{k \ge 0} \mapsto \big( (\mathcal S_n z)_{k'} \big)_{k' \ge -n} \ , \qquad (\mathcal S_n z)_{k'}:= z_{k' + n} \, .
\end{equation}
Note that
\begin{equation}\label{def norm with shift}
\| \mathcal S_n z \|_{s} =  \big( \sum_{k'=-n}^\infty \langle k' \rangle^{2s} |(\mathcal S_n z)_{k'}|^2  \big)^{1/2} 
=  \big( \sum_{k=0}^\infty \langle k - n \rangle^{2s} |z_k|^2  \big)^{1/2} \, .
\end{equation}
Denote by $H^{s; n}_+$ the $\C-$vector space $H^s_+$ endowed with the inner product, induced by the norm
\begin{equation}\label{shifted norm}
\| f \|_{s;n} := \| \mathcal S_n \widehat f \|_s\, , \qquad  \widehat f := (\widehat f(k))_{k \ge 0} \, .
\end{equation}
\begin{lemma}\label{equivalence of norms}
For any $s \in \R_{\ge 0}$, $n \in \N_0$,  and $f \in H^{-s}_+$, the following estimates hold:
\begin{equation}\label{equivalence of shifted norms}
\| f \|_{-s;n} \le 2^s \langle n \rangle^s \| f \|_{-s} \,  , \qquad 
\| f \|_{-s} \le 2^s \langle n \rangle^s \| f \|_{-s;n} \, .
\end{equation}
\end{lemma}
\begin{proof}
To verify the first estimate of \eqref{equivalence of shifted norms}, note that 
$\langle k \rangle \le 2 \langle k-n \rangle \langle n \rangle$
for any $k, n \ge 0$ and hence 
$\frac{1}{\langle k-n \rangle } \le 2 \frac{\langle n \rangle}{\langle k \rangle}$,
implying  that 
$$
\frac{1}{\langle k-n \rangle^{2s}} \le 2^{2s} \frac{\langle n \rangle^{2s}}{\langle k \rangle^{2s}}\, ,  \qquad
\forall s \ge 0\, , \ k,n \ge 0\ .
$$
Hence in view of \eqref{def norm with shift} and \eqref{shifted norm} one has for any $f \in H^{-s}_+$
$$
\| f\|_{-s; n} 
%= \Big( \sum_{k=0}^{\infty} \frac{1}{\langle k-n \rangle ^{2s}} |\widehat f(k) |^2 \Big)^{1/2}
\le 2^s \langle n \rangle^s \Big( \sum_{k=0}^{\infty} \frac{1}{\langle k \rangle ^{2s}} |\widehat f(k) |^2 \Big)^{1/2}
= 2^s \langle n \rangle^s \|f \|_{-s} \, .
$$
To establish the second estimate of \eqref{equivalence of shifted norms}, we argue similarly. For any $k, n \ge 0$,
one has $\langle k - n \rangle \le  2 \langle k \rangle \langle n \rangle$ or 
$$
\frac{1}{\langle k \rangle^{2s} } \le 2^{2s} \frac{\langle n \rangle^{2s} }{\langle k - n \rangle^{2s}} \, , \qquad
\forall s \ge 0\, , \ k,n \ge 0\ ,
$$
and hence for any $f \in H^{-s}_+$ with $s \ge 0$,
$$
\|f \|_{-s} \le 2^s \langle n \rangle^s \Big( \sum_{k=0}^{\infty} \frac{1}{\langle k - n \rangle^{2s}} |\widehat f(k) |^2 \Big)^{1/2}
= 2^s \langle n \rangle^s \| f \|_{-s; n}\, .
$$
\end{proof}

For $0 < \rho \le 1/3$ and $n \in \N_0$, denote by $D_n(\rho)$ the open disk in $\C$
$$
D_n(\rho) := \{ \lambda \in \C \ : \ |\lambda - n | < \rho \} 
$$
and by $\partial D_n(\rho)$ its boundary,
$$
\partial D_n(\rho) = \{ \lambda = n + \rho e^{i\theta} \ : \ \theta \in \R / 2\pi \Z \} \, .
$$
Note that the distance between different discs is at least $1/3$. 
Furthermore, for any $n \ge 1$, we introduce the closed vertical strips $Vert_n(\rho)$ of width one
 with the disc $D_n(\rho)$ removed,
$$
\text{Vert}_n(\rho) := \{ \lambda \in \C \ : \ | \lambda - n | \ge \rho; \  | \Re \lambda - n | \le 1/2  \},
$$
whereas for $n = 0,$ we define
$$
\text{Vert}_0(\rho) :=  \{  \lambda \in \C \ : \ | \lambda | \ge \rho;  \  \Re \lambda  \le 1/2  \}.
$$
We remark that the union of $\text{Vert}_n(\rho) \cup D_n(\rho)$, $n \ge 0$, covers the complex plane.
\begin{lemma}\label{estimate of lambda in Vert}
For any $k, n \in \N_0$, $0 < \rho \le 1/3$, and $\lambda \in \text{Vert}_n(\rho)$,
$$
| k-\lambda | \ge \rho \langle k-n \rangle .
$$
\end{lemma}
\begin{proof}
Let $\lambda \in \text{Vert}_n(\rho)$ with $n \ge 0$.
Then $| n -\lambda | \ge \rho = \rho \cdot \langle 0 \rangle$
whereas for $k \ne n,$ 
$$
| k -\lambda | \ge | k - n | - 1/2 \ge \frac{1}{2} \langle k- n \rangle .
$$
\end{proof}
For $u \in H^{-s}_{c,0}$ with $0 \le s < 1/2$, $\lambda \in \text{Vert}_n(\rho)$ with $n \ge 0$, and $0 < \rho \le 1/3,$
we write $L_u - \lambda$ in the form
\begin{equation}\label{formula for L_u - lambda}
L_u - \lambda = (D - \lambda) - T_u = \big( \text{Id} - T_u(D - \lambda)^{-1}\big)\big( D - \lambda \big)  .
\end{equation}
In order to show that $L_u - \lambda : H^{1-s}_+ \to H^{-s}_+$ is invertible for $u$ near 0, we prove the following
\begin{lemma}\label{estimate of T_u(D- lambda)^-1}
Let $0 \le s < 1/2$ and $0 < \rho \le 1/3.$ Then for any $u \in H^{-s}_{c,0}$ and $n \ge 0$,
$$
\| T_u(D - \lambda)^{-1} \|_{H^{-s; n}_+ \to H^{-s; n}_+} \le \frac{1}{\rho} C_s \|u \|_{-s} \, , \qquad \forall \lambda \in \text{Vert}_n(\rho) \, ,
$$
where $C_s > 0$ is the constant of Lemma \ref{multiplication u f}.
\end{lemma}
\begin{proof}
Let $f \in H^{-s}_+$ and $\lambda \in  \text{Vert}_n(\rho)$. Expressing $T_u( (D - \lambda)^{-1} f)$ as a Fourier series yields
$$
T_u( (D - \lambda)^{-1} f) = \sum_{m=0}^{\infty} \Big(\sum_{k=0}^{\infty}  \widehat u(m-k) \frac{\widehat f(k)}{k - \lambda} \Big) e^{imx} \, .
$$
Hence $\| T_u( (D - \lambda)^{-1} f) \|_{-s; n}^2$ can be estimated by
$$
 \sum_{m=0}^{\infty} \frac{1}{\langle m-n \rangle^{2s}} \Big(\sum_{k=0}^{\infty} | \widehat u(m-k) | \frac{|\widehat f(k)|}{|k - \lambda|} \Big)^2 \, .
$$
With $m':= m-n$, $k':= k-n$ and using that $\widehat f(k' + n ) = \mathcal S_n \widehat f(k')$, Lemma \ref{estimate of lambda in Vert} implies that
$$
\| T_u( (D - \lambda)^{-1} f) \|_{-s,;n}^2 \le 
\sum_{m'= - n}^{\infty} \frac{1}{\langle m' \rangle^{2s}} \Big(\sum_{k'=-n}^{\infty} | \widehat u(m'-k') | \frac{|\mathcal S_n \widehat f(k')|}{\rho \langle k' \rangle} \Big)^2 \,  .
$$
Applying Lemma \ref{multiplication u f} one concludes that
$$
\| T_u( (D - \lambda)^{-1} f) \|_{-s; n} \le \frac{1}{\rho} C_s \|u \|_{-s} \| \mathcal S_n \widehat f \|_{-s} = 
\frac{1}{\rho} C_s \|u \|_{-s} \| f \|_{-s; n} \, .
$$
\end{proof}
As an immediate consequence of Lemma \ref{estimate of T_u(D- lambda)^-1}, one obtains the following
\begin{corollary}\label{1/2 bound}
Let $0 \le s < 1/2$, $0< \rho \le 1/3,$ and let $C_s$ be the constant of Lemma \ref{multiplication u f}. 
Then for any $u \in H^{-s}_{c,0}$ with $\| u \|_{-s} \le  \rho/4 C_s$ 
the following holds:\\
(i) For any $n \ge 0$
$$
\| T_u(D - \lambda)^{-1} \|_{H^{-s; n}_+ \to H^{-s; n}_+} \le \frac{1}{4}\, , \qquad \forall \lambda \in \text{Vert}_n(\rho) .
$$
(ii) For any $\lambda \in \bigcup_{n \ge 0} \text{Vert}_n(\rho)$, the operator $L_u - \lambda : H^{1-s}_+ \to H^{-s}_+$
is invertible with inverse given by 
\begin{equation}\label{formula inverse}
(L_u - \lambda)^{-1} = (D - \lambda)^{-1} (\text{Id} -  T_u(D - \lambda)^{-1})^{-1}\, .
\end{equation}
\end{corollary}
\begin{proof}
Item (i) follows from Lemma \ref{estimate of T_u(D- lambda)^-1}. It implies that for any $\lambda \in \text{Vert}_n(\rho)$ with $n \ge 0$,
 the Neumann series $\sum_{k = 0}^\infty (T_u(D - \lambda)^{-1})^k$ absolutely converges to a bounded operator
on $H^{-s; n}_+$, hence in turn $\text{Id} -  T_u(D - \lambda)^{-1} : H^{-s; n}_+ \to H^{-s; n}_+$ is invertible
and by Lemma \ref{equivalence of norms}, so is $\text{Id} -  T_u(D - \lambda)^{-1} : H^{-s}_+ \to H^{-s}_+$.
By \eqref{formula for L_u - lambda} it then follows that for any $\lambda \in \bigcup_{n \ge 0}\text{Vert}_n(\rho),$
the operator $L_u - \lambda : H^{1-s}_+ \to H^{-s}_+$ is invertible with inverse given by \eqref{formula inverse}.
\end{proof}
To summarize our findings, denote by $B^{s}_{c,0}(r) \equiv B^{s}_{c,0}(0, r)$ the open ball  in $H^{s}_{c,0}$, $s \in \R$,
 of radius $r > 0$, centered at $0$, 
$$
B^{s}_{c,0}(r) := \{ u \in H^{s}_{c,0} \ : \ \|u\|_s  < r  \} .
$$
and by $\text{Vert}^0_n(\rho)$ the interior of $\text{Vert}_n(\rho)$. So e.g. for any $n \ge 1$
$$
\text{Vert}^0_n(\rho) =  \{ \lambda \in \C \ : \ | \lambda - n | > \rho; \  | \Re \lambda - n | < 1/2  \} \, .
$$
We have proved the following
\begin{proposition}\label{Lu for u small}
Let $0 \le s < 1/2$, $0 < \rho \le1/3$, and let $C_s$ be the constant of Lemma \ref{multiplication u f}.  
Then for any $u \in H^{-s}_{c,0}$ with $\| u \|_{-s} \le \rho /4C_s$, 
$L_{u}$ defines an unbounded operator on $H^{-s}_+$ with domain $H^{1-s}_+$.
It has compact resolvent, hence has discrete spectrum,  and its resolvent set contains $\bigcup_{n \ge 0}\text{Vert}_n(\rho)$.
Furthermore, the map
$$
B^{-s}_{c,0}(r_s) \times \bigcup_{n \ge 0}\text{Vert}^0_n(1/4) \to \mathcal B(H^{-s}_+, H^{1-s}_+), (u, \lambda) \mapsto (L_u - \lambda)^{-1}
$$
is analytic where $r_s$ is defined as $r_s := 1/16C_s$.
\end{proposition}
Proposition \ref{Lu for u small} allows to introduce the Riesz projectors.
\begin{corollary}\label{Riesz projector analytic I}
Let $0 \le s < 1/2$, $\rho = 1/4$, and let $C_s$ be the constant of Lemma \ref{multiplication u f}. 
Then for any $u \in H^{-s}_{c,0}$ with $\| u \|_{-s} \le 1/16C_s$ and $n \ge 0$, the Riesz projector
\begin{equation}\label{Riesz projector for u small}
P_n(u):=  - \frac{1}{2\pi i} \int_{\partial D_n(1/3)} (L_u - \lambda)^{-1} d \lambda\  \in \mathcal B(H^{-s}_+, H^{1-s}_+) 
\end{equation}
is well defined (with $\partial D_n(1/3)$ being counterclockwise oriented)  and
$$
B^{-s}_{c,0}(r_s)  \to \mathcal B(H^{-s}_+, H^{1-s}_+), \ u \mapsto P_{n}(u)
$$
is analytic where $r_s = 1/16C_s$.
\end{corollary}
We remark that for any $\lambda \in \cup_{n \ge 0} {\rm Vert}_n(\rho)$,
the inverse $(L_u - \lambda)^{-1} $ is given by the (absolutely convergent) series
$$
\sum_{k = 0}^\infty  (D - \lambda)^{-1} (T_u (D - \lambda)^{-1} )^{k} = (D - \lambda)^{-1}  + (D - \lambda)^{-1}T_u (D - \lambda)^{-1}  + \ldots \ .
$$
We now have all the ingredients to prove Theorem \ref{main theorem}(i) for $u \in H^{-s}_{c,0}$ near $0$, 
which in this case reads as follows.
\begin{theorem}\label{main result of section 2}
Let $u \in H^{-s}_{c,0}$, $0 \le s < 1/2$,  with $\| u \|_{-s} \le \rho / 4C_s$, $0 < \rho \le 1/4$, and let
$C_s$ be the constant of Lemma \ref{multiplication u f}.  
Then the spectrum of $L_{u}$ consists of a sequence of eigenvalues,
denoted by $\lambda_n(u)$, $n \ge 0$. For any $n \ge 0,$ $\lambda_n(u)$ has algebraic multiplicity one and is contained in $D_n(\rho)$. 
Furthermore, the eigenvalues
$$
B^{-s}_{c,0}(r_s)  \to \mathbb C \, , \ u \mapsto \lambda_n(u)\, , \qquad \forall n \ge 0 , 
$$
and hence the gap lengths
$$
B^{-s}_{c,0}(r_s)  \to \mathbb C \, , \ u \mapsto \gamma_n(u)= \lambda_n(u) - \lambda_{n-1}(u) - 1 \, , \qquad \forall n \ge 1 ,
$$
are analytic maps.
\end{theorem}
\begin{remark}
(i) Note that for any $n \ge 0$ and $0 < \rho \le 1/4$, $D_{n+1}(\rho) = D_n(\rho) + 1$. Since by Theorem \ref{main result of section 2} one has
 $\lambda_n(u) \in D_n(\rho)$ for any $u \in H^{-s}_{c,0}$ with $\|u\|_{-s} < \rho/4C_s,$ it follows that
$$
\lambda_n(u) + 1 \in D_{n+1}(\rho) \, , \qquad \forall n \ge 0 \, .
$$
(ii) Since for any $u \in H^{-s}_{c,0}$ with $\| u \|_{-s} \le 1/16C_s$ and  $n \ge 0$, 
the eigenvalue of $\lambda_n(u)$ of $L_{u}$ is contained in $D_n(1/4)$ one has
$$
\text{dist}(\lambda_n(u), \partial D_n(1/3)) \ge 1/12 \, , \qquad \forall n \ge 0\,  .
$$
\end{remark}

\smallbreak
\noindent
{\em Proof of Theorem \ref{main result of section 2}.}
Let $u \in B^{-s}_{c,0}(r_s)$.
By Proposition \ref{Lu for u small}, $L_{u}$ has compact resolvent and 
hence for any $n \ge 0$, the projection $P_{n}(u)$ has finite rank and hence is of trace class. 
The value of ${\rm{Tr}}P_{n}$  at $u$ is given by the rank of $P_{n}(u)$, hence is an integer. 
Since by Corollary \ref{Riesz projector analytic I}, for any $n \ge 0$, the map
$B^{-s}_{c,0}(r_s) \to \C, \, u  \mapsto {\rm{Tr}}P_{n}(u)$ is analytic (and hence continuous),  
$ {\rm{Tr}}P_{n}$ is constant. Noting that
 the spectrum of $D$ consists of the sequence of simple eigenvalues $\lambda_n=n$, $n \ge 0$,
it then follows that for any $n \ge 0$, $L_{u}$ has precisely one eigenvalue in $D_n(1/4)$
and that this eigenvalue is of algebraic multiplicity one. We denote it by $\lambda_n(u)$.
Going through the above arguments one sees that for any $u \in H^{-s}_{c,0}$ with $\| u \|_{-s} \le \rho/4C_s$,
one has $\lambda_n(u) \in D_n(\rho)$.\\
It remains to prove that for any $n \ge 0$, 
$B^{-s}_{c,0}(r_s)  \to \mathbb C \ , \  u \mapsto \lambda_n(u)$ is analytic. This follows by observing that
$\lambda_n(u)$ equals the trace  of the operator  $- \frac{1}{2\pi i} \int_{\partial D_n(1/3)}  \lambda (L_u - \lambda)^{-1} d \lambda$,
$$
\lambda_n(u) =  -  \text{Tr} \big( \frac{1}{2\pi i} \int_{\partial D_n(1/3)}  \lambda (L_u - \lambda)^{-1} d \lambda \big) \,  .
$$ 
Furthermore, $\lambda_n(u)$ is real valued for $u$ real valued (cf. \cite{GK}, \cite{GKT}).
\hfill $\square$

\medskip

We finish this section with discussing sufficient conditions as well as a necessary condition for a potential $u \in B^{-s}_{c,0}(r_s)$ with $0 \le s < 1/2$ 
to satisfy $\gamma_n(u) = 0$ for a given $n \ge 1$. We refer to \cite[Lemma 2.5]{GK} for corresponding results in the case $u$ is in $L^2_r$. 
First we make some preliminary considerations. Arguing as in \cite[Section 2]{GK}, one shows that
for any $u \in B^{-s}_{c,0}(r_s)$ with $0 \le s < 1/2$ one has (cf. \cite[(2.7)]{GK})
\begin{equation}\label{shift identity}
L_u \mathcal S = \mathcal SL_u + \mathcal S - \langle u \mathcal S \cdot | 1 \rangle 1
\end{equation}
where for any $\sigma \in \R$, $\mathcal S \equiv \mathcal S_1: H^{\sigma}_+ \to H^{\sigma}_+, \ f \mapsto e^{ix}f$ is the one shift operator on $H^{\sigma}_+$.
The following lemma is a version of \cite[Lemma 2.6]{GK}, adapted to the situation at hand.
\begin{lemma}\label{identity with p k }
For any $u \in B^{-s}_{c,0}(r_s)$ with $0 \le s < 1/2$ and any set $h_n$, $n \ge 0$, of eigenfunctions of $L_{u}$,
corresponding to the eigenvalues $\lambda_n \equiv \lambda_n(u)$, one has
$$
(\lambda_p - \lambda_k  -1 )P_p \mathcal Sh_k = - \langle u \mathcal S h_k | 1 \rangle P_p1 \, , \qquad \forall \, p, k \ge 0 .
$$
\end{lemma}
\begin{proof}
Let $u \in B^{-s}_{c,0}(r_s)$ with $0 \le s < 1/2$ and let $h_n$, $n \ge 0$, be a set of eigenfunctions of $L_{u}$,
corresponding to the eigenvalues $\lambda_n$. Note that $h_n \in H^{1-s}_+$ for any $n \ge 0$.
Recall that by \eqref{shift identity}
$$
(L_u - \lambda_k  -1 )\mathcal S h_k = - \langle u \mathcal S h_k | 1 \rangle 1 \,.
$$
We now apply the projection $P_p \equiv P_p(u)$, $p \ge 0$, to both sides of the latter identity.
Using that $L_u$ and $P_p$ commute and that $L_uP_p = \lambda_p P_p$ one concludes that
$$
(\lambda_p - \lambda_k  -1 )P_p \mathcal S h_k = - \langle u \mathcal S h_k | 1 \rangle P_p1 \,.
$$
\end{proof}
\begin{lemma}\label{characterization gamma_n = 0} 
Let $u \in B^{-s}_{c,0}(r_s)$ with $0 \le s < 1/2$. Then the following holds:\\
(i) For any set $h_k$, $k \ge 0$, of eigenfunctions of $L_{u}$,
corresponding to the eigenvalues $\lambda_k(u)$, one has $\langle h_0 | 1 \rangle \ne 0$ and for any $n\ge 1$,
the following statements are equivalent:
$$
(G1) \  \langle h_n | 1 \rangle = 0; \quad (G2) \ L_u \mathcal S h_{n-1} = \lambda_n \mathcal S h_{n-1}; \quad
(G3)\  \langle u \mathcal S h_{n-1} | 1 \rangle = 0.
$$
Furthermore, any of these three conditions implies  $(G4) \  \gamma_n(u) = 0$.\\
(ii) If $\gamma_n(u)= 0,$ then $\langle P_n(u) 1 | 1 \rangle = 0$.
\end{lemma}
\begin{proof}
Let $u \in B^{-s}_{c,0}(r_s)$ with $0 \le s < 1/2$ and let $h_k$, $k \ge 0$, be a set of eigenfunctions of $L_{u}$,
corresponding to the eigenvalues $\lambda_k \equiv \lambda_k(u)$.
(i) Assume that for some $n \ge 0$, $\langle h_n | 1 \rangle = 0$. Then $h_n = \mathcal S g_n$ where $g_n \in H^{1-s}_+$, $g_n \ne 0$,
and by \eqref{shift identity}
$$
\lambda_n \mathcal S g_n (= L_u \mathcal S g_n) = \mathcal S L_ug_n + \mathcal S g_n - \langle u \mathcal S g_n | 1 \rangle 1 \, .
$$
Applying $\mathcal S^* = T_{e^{-ix}}$ to both sides of the latter identity and using that $\mathcal S^* 1 = 0$ one concludes that
$\lambda_n g_n = (L_u +1)g_n$ or
$$
L_u g_n = (\lambda_n -1)g_n \, ,
$$ 
implying that $\lambda_n -1$ is an eigenvalue of $L_{u}$.
In the case $n=0$, this  contradicts that $\lambda_0 - 1 \in \rm{Vert}_0(1/4)$ is contained
in the resolvent set of $L_{u}$.
Hence we conclude that  $\langle h_0 | 1 \rangle \ne 0$. 
If $n \ge 1$, it follows that $\lambda_n - 1$ is in $D_{n-1}(1/4)$ and hence by Theorem \ref{main result of section 2}
one has  $\lambda_n - 1 = \lambda_{n-1}$, implying that $\mathcal S h_{n-1}$ and $h_n$ are collinear.
This shows that $(G1)$ implies $(G2)$.
Applying \eqref{shift identity} to $h_{n-1}$, one gets
\begin{equation}
\begin{aligned}
L_u \mathcal S h_{n-1} & = \mathcal S L_uh_{n-1} + \mathcal S h_{n-1} - \langle u \mathcal S h_{n-1} | 1 \rangle 1 \nonumber \\
& = (\lambda_{n-1} + 1) \mathcal S h_{n-1} - \langle u \mathcal S h_{n-1} | 1 \rangle 1 \, , \nonumber
\end{aligned}
\end{equation}
from which we infer that the statements $(G2)$ and $(G3)$ are equivalent. Assuming that $(G2)$ holds,
one has $L_u \mathcal S h_{n-1} = (\lambda_{n-1} + 1) \mathcal S h_{n-1}$. Hence $\lambda_{n-1} + 1$ is an eigenvalue
of $L_u$. Since it is contained in $D_{n}(\rho)$, it follows from Theorem \ref{main result of section 2} that $\lambda_{n-1} + 1 = \lambda_{n}$,
showing that $\gamma_n = 0$, i.e., $(G4)$ holds. Furthermore, it implies that $h_n$ and $\mathcal S h_{n-1}$ are collinear and hence 
$\langle h_n | 1 \rangle = 0$, meaning that $(G1)$ holds.\\
(ii) Assume that $\gamma_n(u) = 0$ for some $n \ge 1$.
Since by Lemma \ref{identity with p k } with $p=n$ and $k = n-1$,
$$
\gamma_n(u) P_n(u) \mathcal S h_{n-1} = - \langle u \mathcal S h_{n-1} | 1 \rangle P_{n}(u)1 \, ,
$$
it then follows that $\langle u \mathcal S h_{n-1} | 1 \rangle P_{n}(u)1 = 0.$ 
If $P_{n}(u)1 = 0$, then clearly, $\langle  P_n(u) 1 | 1 \rangle = 0$. On the other hand 
if $\langle u \mathcal S h_{n-1} | 1 \rangle = 0 $, then by item (i),
 $\langle h_n | 1 \rangle = 0$. Since $P_n(u) 1$ is a scalar mutiple of $h_n$, this implies that $\langle P_{n}1 | 1 \rangle = 0.$ 
\end{proof}
 
  %%%%%%%%%%%%%%%%%%%%%%%%%%%%%%%%%%%%%%%%%%%%%%%%%%%%%
 %%%%%%%%%%%%%%%%%%%%%%%%%%%%%%%%%%%%%%%%%%%%%%%%%%%%%
 
 \section{Quasi-moment map for $u$ small}\label{Quasi-moment map}
 In this section we introduce the so called quasi-moment map for $u$ in a neighborhood of $0$ in $H^{-s}_{c, 0}$, $0 \le s < 1/2$,
 and prove that it is analytic with values in a weighted $\ell^1$ space. It is a key ingredient into the proof of Theorem \ref{main theorem}(ii)
 for potentials near zero. 
 Furthermore, the results obtained in this section will be used in Section \ref{Lax operator.II} to derive corresponding results
 for the quasi-moment map in a neighborhood of $H^{-s}_{r,0}$ in $H^{-s}_{c,0}$.

By Corollary \ref{1/2 bound} with $0 \le s < 1/2$ and $\rho = 1/4$, it follows that for any $n \ge 0$ and $u \in B^{-s}_{c,0}(r_s)$
$$
\| T_u(D - \lambda)^{-1} \|_{H^{-s; n}_+ \to H^{-s; n}_+} \le 1/4\, , \qquad \forall \lambda \in \text{Vert}_n(1/4) \, ,
$$
where $r_s = 1/16C_s$ and $C_s \ge 1$ is the constant of Lemma \ref{multiplication u f}.
As a consequence, for any $\lambda \in \bigcup_{n \ge 0} \text{Vert}_n(1/4)$, 
the inverse of the operator $L_u - \lambda : H^{1-s}_+ \to H^{-s}_+$ is given by the Neumann series
\begin{equation}\label{Neumann series for inverse}
 (L_u - \lambda)^{-1} = \sum_{m = 0}^\infty (D - \lambda)^{-1}( T_u(D - \lambda)^{-1})^m.
\end{equation}
 
 For any $n \ge 0$, introduce the map
 \begin{equation}\label{def G_n}
 F_n : B^{-s}_{c,0}(r_s) \to \mathbb C\ , \ u \mapsto \frac{1}{2\pi i}\int_{\partial D_n(1/3)} \langle (L_u - \lambda)^{-1} 1 | 1 \rangle d \lambda \, ,
 \end{equation}
 where we recall that the circle $\partial D_n(1/3)$ is counterclockwise oriented.
 By Corollary \ref{Riesz projector analytic I}, the map $F_n$ is analytic and by \eqref{Neumann series for inverse} it can be written as a series,
 $$
 F_n(u) =  \sum_{m = 0}^{\infty} \frac{1}{2\pi i } \int_{\partial D_n(1/3)} \langle (D- \lambda)^{-1}(T_u(D - \lambda)^{-1})^m 1 | 1 \rangle d \lambda \, .
 $$
 We note that for a real valued potential $u \in B^{-s}_{r,0}(r_s) \subset H^{-s}_{r, 0}$ with $0 \le s < 1/2$, 
 the value $F_n(u)$, $n \ge 0$, can then be computed as follows: 
 let $(f_j)_{j \ge 0}$ be the $L^2-$orthonormal basis of eigenfunctions, corresponding to the eigenvalues $(\lambda_j(u))_{j\ge 0}$ of $L_u$,
 introduced in \cite{GK} $(s=0)$ and \cite{GKT} $(0 < s < 1/2)$.
 Expressing the constant function $1$ with respect to the basis $(f_j)_{j \ge 0}$, one has
 $ 1 = \sum_{j=0}^{\infty} \langle 1 | f_j \rangle f_j$ and hence
 $$
 (L_u - \lambda)^{-1} 1 = \sum_{j=0}^{\infty} \langle 1 | f_j \rangle (L_u - \lambda)^{-1} f_j 
 = \sum_{j=0}^{\infty} \langle 1 | f_j \rangle \frac{f_j}{\lambda_j - \lambda} \, .
 $$
 By Cauchy's theorem it then follows that
 \begin{equation}\label{identity1 for F_n for u real}
 F_n(u) =  \sum_{j=0}^{\infty} \langle 1 | f_j \rangle \langle f_j | 1 \rangle  \frac{1}{2\pi i } \int_{\partial D_n(1/3)}  \frac{1}{\lambda_j - \lambda} d \lambda
=  - | \langle 1 | f_n \rangle |^2 .
 \end{equation}
  Therefore, for any $n \ge 0$, $F_n$ is the analytic extension of $ -  |\langle 1 | f_n \rangle|^2$ to $B^{-s}_{c, 0}(r_s)$.
 Note that for any $n \ge 0$,  $ |\langle 1 | f_n \rangle|^2$ does not depend on the normalization of the phases of $(f_j)_{j \ge 0}$. Hence
 any other choice of an $L^2-$orthonormal basis of eigenfunctions of $L_u$ would yield the same result.

\begin{proposition}\label{F analytic near 0}
There exists $0 <  r_{s, *} \le r_s$ so that the map
$$
F :  B^{-s}_{c, 0}(r_{s, *}) \subset H^{-s}_{c, 0} \to \ell^{1, 2-2s}_+ , \ u \mapsto F(u):= (F_n(u))_{n \ge 1}
$$
is analytic. Furthermore
$$
\sum_{n=1}^\infty  \sum_{m=0}^\infty n^{2-2s}  \big|  \frac{1}{2\pi i} \int_{\partial D_n(1/3)}  \langle (D- \lambda)^{-1}(T_u(D - \lambda)^{-1})^m 1 | 1 \rangle d \lambda \big|
 $$
converges  uniformly with respect to $u$.
\end{proposition}
\begin{remark}\label{identity2 for F_n for u real}
(i) It follows from  \cite[Corollary 3.4]{GK} $(s=0)$, \cite[(27)]{GKT} $(0 < s < 1/2)$ that for any $n \ge 1$ 
and $u \in H^{-s}_{r, 0}$ with $0 \le s < 1/2$,  
\begin{equation}\label{name quasi-momentum}
%- F_n(u) = 
|\langle 1 | f_n(\cdot, u) \rangle|^2 = \gamma_n(u) \kappa_n(u)
\end{equation}
 where  $\kappa_n(u)> 0$ is a scaling factor of the action variable $\gamma_n(u)$. Since $\kappa_n(u) = O(n^{-1})$ (cf. \cite[Corollary 3.4]{GK}$(s=0)$, 
 \cite[(3.2)]{GKT}$(0 < s < 1/2)$)
  it follows that  for any $u \in H^{-s}_{r,0}$,
  $F(u) = (-\kappa_n(u) \gamma_n(u))_{n \ge 1}  \in \ell^{1, 2-2s}(\N, \R)$.  
 In view of \eqref{name quasi-momentum}, we refer to $F$ as the quasi-moment map. \\
(ii) Our proof shows that for any $0 \le s < 1/2$, the map $F : B^{-s}_{c, 0}(r_{s, *}) \to \ell^{1, 2-2s}_+$ is normally analytic.
We refer to Appendix \ref{normally analytic} for a review of the notion of normally analytic maps.
\end{remark}
 Before proving Proposition \ref{F analytic near 0} we make some preliminary considerations. 
 Let $u \in B^{-s}_{c, 0}(r_{s})$ with $0 \le s < 1/2$.
 Since for any $n \ge 1,$
 the integral $\int_{\partial D_n(1/3)} \langle (D - \lambda)^{-1} 1 | 1 \rangle d \lambda =  
 \int_{\partial D_n(1/3)} - \lambda^{-1} d \lambda$ vanishes, one has
$$
 F_n (u) =  \sum_{m = 1}^{\infty} \frac{1}{2\pi i } \int_{\partial D_n(1/3)} \langle (D- \lambda)^{-1}(T_u(D - \lambda)^{-1})^m 1 | 1 \rangle d \lambda 
 $$
 and since 
 $$
  \langle (D- \lambda)^{-1}T_u(D - \lambda)^{-1}1 | 1 \rangle =  \frac{1}{\lambda^2} \langle \Pi u | 1 \rangle = 0 , 
 $$
 also $\int_{\partial D_n(1/3)} \langle (D- \lambda)^{-1}T_u(D - \lambda)^{-1} 1 | 1 \rangle d \lambda$ vanishes for any $n \ge 1$.
 Changing the index of summation from $m$ to $m-1$ then yields
  \begin{equation}\label{F_n as a series}
 F_n(u) = \sum_{m = 1}^{\infty} \frac{1}{2\pi i } \int_{\partial D_n(1/3)}  \langle (T_u(D - \lambda)^{-1})^m \Pi u | 1 \rangle \frac{1}{\lambda^2} d \lambda \,  ,
 \quad \forall n \ge 1 .
 \end{equation}
 The terms for $m=1$ and $m=2$ in the latter series can be computed as follows: For $m=1$ one has
 $$
 \begin{aligned}
 \langle T_u(D - \lambda)^{-1} \Pi u | 1 \rangle & = \sum_{k=0}^\infty \langle \Pi (u \frac{\widehat u(k)}{k - \lambda} e^{ikx}) | 1 \rangle\\
 & = \sum_{j \in \Z} \sum_{k=0}^\infty  \frac{\widehat u(j) \widehat u(k)}{k - \lambda} \langle e^{i(k+j)x} | 1 \rangle
 = \sum_{k=0}^\infty \frac{\widehat u(-k) \widehat u(k)}{k - \lambda} , 
 \end{aligned}
 $$
 implying that
 $$
 \frac{1}{2\pi i } \int_{\partial D_n(1/3)} \langle T_u(D - \lambda)^{-1} \Pi u | 1 \rangle \frac{1}{\lambda^2} d \lambda
 = - \frac{1}{n^2} \widehat u(n) \widehat u(-n) , \quad \forall n \ge 1 .
 $$
 Similarly, one computes the term for $m=2$,
 $$
 \begin{aligned}
 (T_u(D - \lambda)^{-1})^2 \Pi u &= T_u(D - \lambda)^{-1} \sum_{k_2 = 0}^\infty \sum_{k_1 = 0}^\infty \frac{\widehat u(k_1)}{k_1 - \lambda} \widehat u(k_2 - k_1) e^{ik_2 x} \\
 & = \Pi \big( \sum_{j \in \Z } \widehat u(j)   \sum_{k_2 = 0}^\infty \sum_{k_1 = 0}^\infty \frac{\widehat u(k_1)}{k_1 - \lambda} \frac{\widehat u(k_2 - k_1)}{k_2 - \lambda} e^{i(k_2+j) x}  \big)\\
& = \sum_{k_3=0}^\infty  \sum_{k_2 = 0}^\infty \sum_{k_1 = 0}^\infty \frac{ \widehat u(k_3 - k_2) \widehat u(k_2 - k_1) \widehat u(k_1)}{(k_2 - \lambda) (k_1 - \lambda) } e^{ik_3 x} 
 \end{aligned}
 $$
 and hence 
 $$
 \begin{aligned}
 \frac{1}{2\pi i } & \int_{\partial D_n(1/3)}  \langle (T_u(D - \lambda)^{-1})^2 \Pi u | 1 \rangle \frac{1}{\lambda^2} d \lambda\\
& =  \sum_{k_2 = 0}^\infty \sum_{k_1 = 0}^\infty  \widehat u(0- k_2) \widehat u(k_2 - k_1) \widehat u(k_1)  
\frac{1}{2\pi i } \int_{\partial D_n(1/3)}  \frac{1}{(k_2 - \lambda) (k_1 - \lambda) \lambda^2} d \lambda\\
&= - \frac{1}{n^2} \widehat u (-n) \sum_{k_1 \ge 0, k_1 \ne n} \frac{\widehat u(n - k_1) \widehat u(k_1)}{k_1 - n}
 -   \frac{1}{n^2}  \sum_{k_2 \ge 0, k_2 \ne n} \frac{\widehat u( - k_2) \widehat u(k_2 - n)}{k_2 - n}  \widehat u (n)\\
 &\quad  +  \widehat u (-n)  \widehat u (0)  \widehat u (n) \int_{\partial D_n(1/3)}  \frac{1}{(n - \lambda)^2 \lambda^2} d \lambda \, .
\end{aligned}
 $$
 Since by assumption $\widehat u(0) = 0$, the latter term vanishes. With $\ell_j:= k_j - n$ for $j=1, 2$ and using that 
 \begin{equation}\label{def u_*}
 u_*(x) := u(-x) = \sum_{j \in \Z} \widehat u(-j) e^{ijx}
 \end{equation} 
 one then gets
 $$
 \begin{aligned}
  \frac{1}{2\pi i } & \int_{\partial D_n(1/3)}  \langle (T_u(D - \lambda)^{-1})^2 \Pi u | 1 \rangle \frac{1}{\lambda^2} d \lambda \\
  & =  - \frac{1}{n^2} \widehat u (-n) \sum_{\ell_1 \ge -n, \ell_1 \ne 0} \frac{\widehat u_*(-n - \ell_1) \widehat u_*(\ell_1)}{\ell_1 }\\
& \quad  -   \frac{1}{n^2}  \sum_{\ell_2 \ge -n, \ell_2 \ne 0} \frac{\widehat u(-n  - \ell_2) \widehat u(\ell_2 )}{\ell_2 }  \widehat u_* (-n) \, .
  \end{aligned}
 $$
 It is convenient to introduce for any $n \ge 1$ and any $u \in H^{-s}_{c,0}$ with $0 \le s < 1/2$ the operators 
 $$
 Q_{u, n} : \ z = (z(\ell))_{\ell \in \Z} \mapsto (Q_{u, n}[z](k))_{k \in \Z}
 $$
 and 
 $$
  Q_{u} : \ z = (z(\ell))_{\ell \in \Z} \mapsto (Q_{u}[z](k))_{k \in \Z}
 $$
 where
 \begin{equation}\label{def Q_u,n}
 Q_{u, n}[z](k) := \sum_{\ell \ge -n, \ell \ne 0} \frac{|\widehat u(k-\ell)| \ z(\ell)}{ |\ell |} \, ,
 \end{equation}
 \begin{equation}\label{def Q_u}
  Q_{u}[z](k) := \sum_{ \ell \ne 0} \frac{|\widehat u(k-\ell)| \ z(\ell)}{ |\ell |} \, .
 \end{equation}
 By (the proof of) Lemma \ref{multiplication u f}(ii) , for any $0 \le s < 1/2$,
 $$
 Q_{u, n},   \ Q_{u} : h^{-s}(\Z, \C) \to h^{-s}(\Z, \C)
 $$
 are bounded linear operators. More precisely, by Lemma \ref{multiplication u f}(ii) the following holds:
 \begin{lemma}\label{operator Q_ u,n}
 For any $0 \le s < 1/2,$ there exists a constant $M_s \ge 1$ so that for any $n \ge 1$ and $z \in h^{-s}(\Z, \C)$
 $$
\|  Q_{u, n}[|z|] \|_{-s} \le  \|  Q_{u}[|z|] \|_{-s} \le M_s \| u \|_{-s} \|z \|_{-s} ,
 $$
 where $|z| = (|z_n|)_{n \in \Z}$.
 \end{lemma}
 With this  notation introduced, one obtains the following estimates
 $$
 \big| \frac{1}{2\pi i } \int_{\partial D_n(1/3)} \langle T_u(D - \lambda)^{-1} \Pi u | 1 \rangle \frac{1}{\lambda^2} d \lambda \big|
 =  \frac{1}{n^2} | \widehat u_*(- n)|  | \widehat u(- n)| 
 $$
 and
 $$
 \begin{aligned}
 \big|  \frac{1}{2\pi i } & \int_{\partial D_n(1/3)}  \langle (T_u(D - \lambda)^{-1})^2 \Pi u | 1 \rangle \frac{1}{\lambda^2} d \lambda \big|  \\
  & \le  \frac{1}{n^2}  \ Q_{ u_*, n}[|\widehat u_*|](-n) \, | \widehat u (-n) |
  +    \frac{1}{n^2}   Q_{ u, n}[|\widehat u|](-n) \,  | \widehat u_* (-n)|
   \end{aligned}
 $$
 where 
 $$
 \widehat u:= (\widehat u(k))_{k \in \Z}\ , \qquad  |\widehat u |:= (|\widehat u(k)|)_{k \in \Z} .
 $$
 
 \smallskip
 
 \noindent
 {\em Proof of Proposition \ref{F analytic near 0}.}
 Consider the integrand in \eqref{F_n as a series} with $m \ge 1$ arbitrary.
 Arguing as for the terms with $m = 1$ and $m=2$,
 the term  $\langle  (T_u(D - \lambda)^{-1})^m \Pi u | 1 \rangle $ can be written as
 $$
 \begin{aligned}
 & \big\langle \sum_{k_{m+1}\ge 0} 
 \Big( \sum_{k_j \ge 0} \frac{ \widehat u(k_{m+1} - k_m) \widehat u(k_m - k_{m-1}) \cdots  \widehat u(k_2 - k_1) \widehat u( k_1)}
 {(k_m - \lambda) \cdots (k_1 - \lambda)}  \Big)
 e^{ik_{m+1}x} \ | 1 \big\rangle \\
 & =  \sum_{k_j \ge 0}  \frac{\widehat u( - k_m) \widehat u(k_m - k_{m-1}) \cdots 
 \widehat u(k_2 - k_1) \widehat u( k_1) } {(k_m - \lambda) \cdots (k_1 - \lambda)}
 \end{aligned}
 $$
 where in the sums above $1 \le j \le m$. Using that by Cauchy's formula for derivatives of an analytic function $f$,
 $$
 \frac{1}{(p-1)!} \frac{d^{p-1}}{d\mu^{p-1}} f(\mu)  =  \frac{1}{2 \pi i } \int_{\partial D_n(1/3)}\frac{f(\lambda)}{(\lambda - \mu)^p} d \lambda \,  ,
 $$
 one obtains
 $$
  \begin{aligned}
 & | \frac{1}{2\pi i }  \int_{\partial D_n(1/3)}  \langle (T_u(D - \lambda)^{-1})^m \Pi u | 1 \rangle \frac{1}{\lambda^2} d \lambda | \\
&  \le \sum_{p=1}^{m} \sum_{|J|=p} \sum_{\substack{k_j \ge 0 \\k_j = n \ \forall j \in J\\ k_j \ne n \ \forall j \in J^c}}
|\widehat u( - k_m)|  \big(\prod_{j=2}^{m}|\widehat u(k_j - k_{j-1})| \big)  |\widehat u( k_1)|
g_{k_1, \ldots, k_m; p}(n)
  \end{aligned}
 $$
 where in the latter sum, $J$ runs over all subsets of $\{ 1, \dots , m \}$ with $|J| = p$,
 $J^c := \{ 1, \ldots , m \} \setminus J$, and 
 $$
 g_{k_1, \ldots, k_m; p}(\lambda) := \frac{1}{(p-1)!} \Big| \frac{d^{p-1}}{d\lambda^{p-1}} 
(  \frac{1}{\lambda^2} \prod_{j \in J^c} \frac{1}{k_j - \lambda} ) \Big| \, .
 $$
 Using that $|k_j - n| \ge  1$ for any $j \in J^c$ one obtains the estimate
 $$
 g_{k_1, \ldots, k_m; p}(n) \le \frac{1}{n^2} \big( \prod_{j \in J^c} \frac{1}{|k_j - n|} \big) \  \frac{1}{(p-1)!} \ \#_{m, p}
 $$
 where $\#_{m, p}$ is the number of terms of 
 $\frac{d^{p-1}}{d\lambda^{p-1}} (  \frac{1}{\lambda^2} \prod_{j \in J^c} \frac{1}{k_j - \lambda} )$,
 counted with their multiplicities,
 $$
 \#_{m, p} = (|J^c| + 2) \cdots (|J^c| + 2 +p -2) = (m-p+2) \cdots m \, .
 $$
 Hence 
 $$
 \frac{1}{(p-1)!}  \#_{m, p} = \frac{m!}{(p-1)! (m-(p-1))!} =
  \Big( \begin{array}{ccc}   
  m \\ 
  p-1
  \end{array} \Big) \, ,
 $$
 yielding
 $$
 g_{k_1, \ldots, k_m; p}(n) \le \frac{1}{n^2} \Big( \prod_{j \in J^c} \frac{1}{|k_j - n|} \Big)
 \Big( \begin{array}{ccc}   
  m \\ 
  p-1
  \end{array} \Big) \, .
 $$
 Since by Lemma \ref{key estimate} below, for any $J \subset \{ 1, \ldots, m \}$ with $|J| = p,$
 $$
 \begin{aligned}
 \sum_{n \ge1} & n^{2-2s}  \sum_{\substack{k_j \ge 0\\ k_j   = n \ \forall j \in J\\ k_j \ne n \ \forall j \in J^c}}
\frac{1}{n^2}\frac{|\widehat u( - k_m)| |\widehat u(k_m - k_{m-1})| \cdots  |\widehat u(k_2 - k_1)| |\widehat u( k_1)|}{\prod_{j \in J^c} |k_j - n |}\\
& \le \big( M_s \|u\|_{-s} \big)^{m+1}
\end{aligned}
 $$
 and $\sum_{p=0}^m  
 \Big( \begin{array}{ccc}   
  m \\ 
  p
  \end{array} \Big) = 2^m$
 we then conclude that 
 $$
 \begin{aligned}
  \sum_{n=1}^\infty   n^{2-2s} & \big|\frac{1}{2\pi i} \int_{\partial D_n(1/3)}  \langle (T_u(D - \lambda)^{-1})^m \Pi u  | 1 \rangle  \frac{1}{\lambda^2} d \lambda \ \big| \\
 & \le \Big(\sum_{p=1}^m  
 \Big( \begin{array}{ccc}   
  m \\ 
  p
  \end{array} \Big) \Big) \big( M_s \|u\|_{-s} \big)^{m+1} \sup_{1 \le p \le m}  \Big( \begin{array}{ccc}   
  m \\ 
  p - 1
  \end{array} \Big)\\
 & \le \big( 4 M_s \|u\|_{-s} \big)^{m+1} \, .
 \end{aligned}
 $$
 It implies that for any $u \in H^{-s}_{c, 0}$ with 
 $$ 
 \| u \|_{-s} \le  r_{s, *} := \min(r_s, 1/8M_s)\, ,
 $$
 one has
 $$
 \begin{aligned}
 & \sum_{n=1}^\infty  \sum_{m=0}^\infty n^{2-2s}  \big|\frac{1}{2\pi i} \int_{\partial D_n(1/3)}  \langle (D - \lambda)^{-1} (T_u(D - \lambda)^{-1})^m 1 | 1 \rangle d \lambda \ \big| \\
 & \le  \frac{4M_s \|u\|_{-s}}{1 - 4 M_s \|u\|_{-s}} \le 8 M_s \|u\|_{-s} \le 1\, .
 \end{aligned}
 $$
 This finishes the proof of Proposition \ref{F analytic near 0}.
 \hfill  $\square$
 
 \smallskip
 
 It remains to prove the following lemma, used in the proof of Proposition \ref{F analytic near 0}.
 With the notation established in the proof of that proposition it reads as follows:
 \begin{lemma}\label{key estimate}
 For any $m \ge 1$ and any nonempty subset $J \subset \{ 1, \ldots , m \}$,
 $$
   \sum_{n \ge1}   n^{-2s}  S_{n,m,J} \le \big( M_s \|u\|_{-s} \big)^{m+1} \, ,
 $$
 where
 $$
 S_{n,m,J} := \sum_{\substack{k_1,\ldots ,k_m\ge 0\\
  k_j   = n \ \forall j \in J\\ k_j \ne n \ \forall j \in J^c}}
\frac{|\widehat u( - k_m)| |\widehat u(k_m - k_{m-1})| \cdots  |\widehat u(k_2 - k_1)| |\widehat u( k_1)|}{\prod_{j \in J^c} |k_j - n |} \,  .
$$
 \end{lemma}
 \begin{proof}
 Let $m \ge 1$ and $J \subset \{ 1, \ldots , m \} \ne \emptyset$.
First observe that we can assume that  $J \subset \{ 1, \ldots , m \}$ does not contain consecutive integers $j$, $j+1$, since otherwise $S_{n,m,J}$ vanishes 
for any $n \ge 1$ due to the assumption that $\widehat u(0)=0$. 
We decompose $J^c$ into pairwise disjoint  intervals of integers of maximal lengths, $J^c_a$, $ 1 \le a \le A$, 
$$
J^c =\bigcup_{a=1}^AJ^c_a \, , \qquad r_a:= |J_a^c|
$$
where $J^c_1, \ldots, J^c_A$ are listed in decreasing order. Note that
$$
\sum_{a=1}^A  r_a = |J^c|=m-|J| 
$$
and that $A$ equals $|J|-1$, $|J|$, or $|J|+1$, depending on whether $|\{1, m \} \cap J |$ equals $2$, $1$, or $0$.
For any $r \ge 1$, introduce
$$
S_{n,r}:=\sum_{\substack{k_1,\ldots ,k_r\ge 0\\ k_j \ne n \ \forall j }}
\frac{|\widehat u(n - k_r)| |\widehat u(k_r - k_{r-1})| \cdots  |\widehat u(k_2 - k_1)| |\widehat u( k_1-n)|}{\prod_{j=1}^r |k_j - n |}\,  .
$$
Setting $\ell_j:=k_j-n$ in the latter sum, we obtain
\begin{eqnarray*}
S_{n,r}&=&\sum_{\substack{\ell_1,\ldots ,\ell_r\ge -n\\ \ell_j \ne 0 \ \forall j }}
\frac{|\widehat u(- \ell_r)| |\widehat u(\ell_r - \ell_{r-1})| \cdots  |\widehat u(\ell_2 - \ell_1)| |\widehat u( \ell_1)|}{\prod_{j=1}^r |\ell_j  |}\\
&=&(Q_{u,n})^r[|\hat u|] (0)\, .
\end{eqnarray*}
We recall that $Q_{u,n}$ is defined by \eqref{def Q_u,n} and that for any sequence $z = (z(\ell))_\ell$ of complex numbers, 
$|z|$ denotes the sequence $(|z(\ell )|)_\ell$. Since
$|z(0)|\leq \| z\| _{-s}$, Lemma \ref{operator Q_ u,n} implies that
\begin{equation}\label{estS}
S_{n,r}\leq (M_s\| u\|_{-s})^r \| u\|_{-s}\leq (M_s\| u\|_{-s})^{r +1}\, ,
\end{equation}
where $M_s \ge 1$ is the constant given by Lemma  \ref{operator Q_ u,n}.
Let us first consider the case where $1\in J$ and $m \in J$. Setting $r_a:= |J_a^c|$,  one has
\begin{eqnarray*}
S_{n,m,J}&=&|\widehat u(-n)| \Big(\prod_{a=1}^AS_{n, r_a}\Big)  |\widehat u(n)|\\
&\leq &(M_s\| u\|_{-s})^{|J^c|+A}
|\widehat u(-n)| |\widehat u(n)|\\
&\leq &(M_s\| u\|_{-s})^{m-1} |\widehat u(-n)| |\widehat u(n)|
\end{eqnarray*}
where we used that $A=|J|-1$ in the case at hand. Consequently
$$
\sum_{n=1}^\infty n^{-2s}S_{n,m,J}\leq (M_s\| u\|_{-s})^{m-1}\sum _{n=1}^\infty n^{-2s}
 |\widehat u(-n)||\widehat u(n)|\leq (M_s\| u\|_{-s})^{m+1}\,  .
 $$
In order to deal with the other cases, we introduce
\begin{eqnarray*}
 ^*S_{n,r}&:=&\sum_{\substack{k_1,\ldots ,k_r\ge 0\\ k_j \ne n \ \forall j }}
\frac{|\widehat u(- k_r)| |\widehat u(k_r - k_{r-1})| \cdots  |\widehat u(k_2 - k_1)| |\widehat u( k_1-n)|}{\prod_{j=1}^r |k_j - n |}\\
&=&\sum_{\substack{\ell_1,\ldots ,\ell_r\ge -n\\ \ell_j \ne 0 \ \forall j }}
\frac{|\widehat u(-n- \ell_r)| |\widehat u(\ell_r - \ell_{r-1})| \cdots  |\widehat u(\ell_2 - \ell_1)| |\widehat u( \ell_1)|}{\prod_{j=1}^r |\ell_j  |}\\
&=&(Q_{u,n})^{r}[|\widehat {u}|] (-n)\ 
\end{eqnarray*}
 and
\begin{eqnarray*}
S^*_{n,r}&:=&\sum_{\substack{k_1,\ldots ,k_r\ge 0\\ k_j \ne n \ \forall j }}
\frac{|\widehat u(n- k_r)| |\widehat u(k_r - k_{r-1})| \cdots  |\widehat u(k_2 - k_1)| |\widehat u( k_1)|}{\prod_{j=1}^r |k_j - n |}\\
&=&\sum_{\substack{\ell_1,\ldots ,\ell_r\ge -n\\ \ell_j \ne 0 \ \forall j }}
\frac{|\widehat u(- \ell_r)| |\widehat u(\ell_r - \ell_{r-1})| \cdots  |\widehat u(\ell_2 - \ell_1)| |\widehat u( \ell_1+n)|}{\prod_{j=1}^r |\ell_j  |} \, .
\end{eqnarray*} 
Using $u_*(x) = u(-x)$, defined in \eqref{def u_*}, $S^*_{n,r}$ can be written as
\begin{eqnarray*}
S^*_{n,r} &=&\sum_{\substack{\ell_1,\ldots ,\ell_r\ge -n\\ \ell_j \ne 0 \ \forall j }}
\frac{ |\widehat u_*(-n -  \ell_1)|  |\widehat u_*(\ell_1 - \ell_{2})| \cdots  |\widehat u_* (\ell_{r-1} - \ell_r)| |\widehat u_*( \ell_r)|}{\prod_{j=1}^r |\ell_j  |}\\
&=&(Q_{u_*,n})^r[|\widehat u_*|] (-n)\, .
\end{eqnarray*} 
With these notations introduced, $S_{n,m,J}$ can be expressed as
\begin{eqnarray*}
S_{n,m,J}&=&\, ^*S_{n,r_1} \left (\prod_{a=2}^A S_{n,r_a}\right )|\widehat u(n)| \, , \qquad \ \  {\rm if}\ 1\in J, \, m\not \in J\, ,\\
S_{n,m,J}&=&|\widehat u(-n)|\left (\prod_{a=1}^{A-1}S_{n,r_a}\right ) S^*_{n,r_A}\, , \qquad  {\rm if}\ 1\not \in J, \, m\in J\, ,\\
S_{n,m,J}&=&\, ^*S_{n,r_1} \left (\prod_{a=2}^{A-1}S_{n,r_a}\right )\, S^*_{n,r_A} \, , \qquad \ \  {\rm if}\ 1\not \in J, \, m\not \in J\, .\\
\end{eqnarray*}
Note that by the definitions \eqref{def Q_u,n} - \eqref{def Q_u}, $0\leq Q_{u,n}[|z|](\ell )\leq Q_u[|z|](\ell )$. Hence in the case where $1\in J$ and $m\not \in J$,
$$
S_{n,m,J}\leq  (Q_u)^{r_1}[|\widehat {u}|] (-n)|  \,  \widehat u(n)|  \,  (M_s\| u\|_{-s})^{\sum_{a=2}^A(r_a+1)}
$$
and therefore
\begin{eqnarray*}
\sum_{n=1}^\infty n^{-2s}S_{n,m,J}&\leq &(M_s\| u\|_{-s})^{\sum_{a=2}^A
(r_a+1)}\sum _{n=1}^\infty n^{-2s}
 (Q_u)^{r_1}[|\widehat {u}|] (-n)|\hat u(n)|\\
 &\leq &(M_s\| u\|_{-s})^{\sum_{a=2}^A
(r_a+1)}     \| (Q_u)^{r_1}[|\widehat {u}|] \|_{-s}  \| u\|_{-s}\\
 &\leq &(M_s\| u\|_{-s})^{m+1}\, ,
 \end{eqnarray*}
 since, in the case at hand, $A=|J|$. Using that $\| u_*\|_{-s} =  \| u \|_{-s}$, the remaining two cases can be dealt with in a similar way.
\end{proof}

  \section{Proof of Theorem \ref{main theorem}$(ii)$ and Corollary \ref{product representations} for $u$ small}\label{proof of main results near zero} 
 In this section we prove Theorem \ref{main theorem}(ii) and  Corollary \ref{product representations} 
  for potentials in $H^{-s}_{c,0}$, $0 \le s < 1/2$, near zero. In addition we discuss properties of 
 the scaling factors $\kappa_n(u)$, $n \ge 1$, and the normalizing constants $\mu_n(u)$, $n \ge 1$.

 As a first step, we study the analytic extension of the generating function, introduced in
 \cite{GK} and then further analyzed in \cite{GKT}. This function is then used
 to define the scaling factors $\kappa_n$, $n \ge 1$, for the components of the quasi-moment map $F$, needed for the proof of 
 Theorem \ref{main theorem}(ii) for $u$ small.
 
 For $u \in B^{-s}_{c,0}(r_{s,*})$ with $r_{s,*}>0 $ given by Proposition \ref{F analytic near 0}, and $\lambda$ in the resolvent set of $L_{u}$, define 
 \begin{equation}\label{generating function}
 \mathcal H_\lambda(u) :=  \langle (L_u - \lambda)^{-1} 1 | 1 \rangle .
 \end{equation}
Up to a sign, the function $\mathcal H_\lambda(u)$ equals the generating function, introduced in \cite{GK} for $u \in H^0_r$.
 In view of Proposition \ref{Lu for u small} and Theorem \ref{main result of section 2}
 it follows that  $\mathcal H_\lambda(u)$ is a meromorphic function of $\lambda$ with the set of poles
 contained in the spectrum of $L_{u}$, all of them being simple. By the definition \eqref{def G_n}, for any $n \ge 0$,
  $F_n(u)$ is the residue of $\mathcal H_\lambda(u)$ at $\lambda = \lambda_n(u)$.
  We point out that $F_n(u)$ might vanish. Our first result says that $ \mathcal H_\lambda(u)$ can be written as an infinite sum:
  \begin{lemma}\label{mathcal H as a sum}
  For any $u$ in $B^{-s}_{c,0}(r_{s,*}),$ the following identities hold:
  \begin{itemize}
  \item[(i)] $F_n(u) =  - \langle P_n(u) 1 | 1 \rangle$,  \quad $\forall \, n \ge 0$;
  \item[(ii)]  $\mathcal H_\lambda(u) = \sum_{n \ge 0} \frac{F_n(u)}{\lambda - \lambda_n(u)} $.
  \end{itemize}
  \end{lemma}
  \begin{proof}
Item (i)  follows from the definition \eqref{def G_n} of $F_n(u)$  and definition \eqref{Riesz projector for u small} of the projection $P_n(u)$.
Towards (ii), note that by Proposition \ref{F analytic near 0}, for any $u \in B^{-s}_{c,0}(r_{s,*}),$
 $ \sum_{n \ge 0} \frac{F_n(u)}{\lambda - \lambda_n(u)} $ is a meromorphic function, having the same poles as  
 the meromorphic function $\mathcal H_\lambda(u)$.
For any $\lambda$  with $\Re \lambda \le - 1/3, $ the functions $ \sum_{n \ge 0} \frac{F_n}{\lambda - \lambda_n} $ and $\mathcal H_\lambda$
are both analytic functions on $B^{-s}_{c,0}(r_{s,*})$. It then suffices to show that they coincide for $u$ real valued:
for any real valued potential $u$ in $B^{-s}_{c,0}(r_{s,*}),$ let $f_n \equiv f_n(\cdot, u)$, $n \ge 0,$ be the $L^2-$orthonormal basis of $H_+$,
 introduced in \cite{GK} $(s=0)$ and \cite{GKT} $(0 < s < 1/2)$, where for any $n \ge 0$, $f_n$ is
 an eigenfunction of $L_{u}$, corresponding to the eigenvalue $\lambda_n(u)$. Then $1 = \sum_{n \ge 0} \langle 1 | f_n \rangle f_n$
 and hence for any $\lambda$ with  $\Re \lambda \le - 1/3,$
 \begin{equation}\label{identity 1}
 \mathcal H_\lambda(u) = \sum_{n \ge 0}  \langle 1 | f_n \rangle \langle (L_u -  \lambda)^{-1} f_n | 1 \rangle
 =  \sum_{n \ge 0} \frac{ \langle 1 | f_n \rangle \langle f_n | 1 \rangle}{\lambda_n - \lambda}\,  .
 \end{equation}
 Since $ F_n(u)$ is the residue of  $\mathcal H_\lambda(u)$ at $\lambda = \lambda_n(u)$, it then follows that
  $\langle 1 | f_n \rangle \cdot \langle f_n | 1 \rangle = - F_n(u)$, showing that the claimed identity
  holds for $u$ real valued.
   \end{proof}
 Note that  by the definition \eqref{generating function}, one has $ \mathcal H_\lambda(0) = - 1/ \lambda.$
 Hence $F_n(0) = 0$ for any $n \ge 1$ and $F_0(0) = - 1$.
 
 \begin{lemma}\label{zeros of mathcal H} 
 Assume that $u \in B^{-s}_{c,0}(r_{s,*})$ with $0 \le s < 1/2$. Then for any $n \ge 1$  
 with $\gamma_n(u) \ne 0$ one has
 $$
 \mathcal H_{ \lambda_{n-1}(u) + 1} (u) = 0 \, .
 $$
 \end{lemma}
 \begin{proof}  Let  $u \in B^{-s}_{c,0}(r_s)$ with $0 \le s < 1/2$ and $n \ge 1$.
 Assume that  $\gamma_n(u) \ne 0$ and 
 let $h_k$, $k \ge 0$, be a set of eigenfunctions of $L_{u}$, corresponding to the eigenvalues 
 $\lambda_k \equiv \lambda_k(u)$.
 By \eqref{shift identity} one has 
 \begin{equation}\label{identity for L_uSh_(n-1)}
L_u\mathcal S h_{n-1} 
%= SL_uh_{n-1} + Sh_{n-1} - \langle uS h_{n-1} | 1 \rangle 1 
= (\lambda_{n-1} + 1) \mathcal S h_{n-1}  - \langle u \mathcal S h_{n-1} | 1 \rangle 1 \, . 
\end{equation}
Since by assumption, $\gamma_n(u) \ne 0$ and hence $\lambda_{n-1} + 1$ is in the resolvent set of $L_{u}$, it follows that
$$
\mathcal S h_{n-1} =  - \langle u \mathcal S h_{n-1} | 1 \rangle (L_u - \lambda_{n-1} - 1)^{-1} 1\, .
$$
 Taking the inner product with the constant function $1$ of both sides of the latter identity and noting that  $\langle \mathcal S h_{n-1} | 1 \rangle = 0$ yields
 $$
 0 =  - \langle u \mathcal S h_{n-1} | 1 \rangle \langle (L_u - \lambda_{n-1} - 1)^{-1} 1 | 1 \rangle = - \langle u \mathcal S h_{n-1} | 1 \rangle \mathcal H_{ \lambda_{n-1} +1}(u) \, .
 $$
 By Lemma \ref{characterization gamma_n = 0}(i), $(G3)$ implies $(G4)$. Since by assumption $\gamma_n(u) \ne 0,$
 it means that $\langle u \mathcal S h_{n-1} | 1 \rangle \ne 0$, implying that $ \mathcal H_{ \lambda_{n-1} +1}(u) = 0$.
 \end{proof}
 
 \begin{lemma}\label{estimate of mathcal H}
  For any $u \in B^{-s}_{c,0}(r_{s,*})$ and any $n \ge 0$, 
 $$
| \mathcal H_\lambda(u) | \ge  \frac{1}{|\lambda|} \cdot \frac{2}{3} \ , \qquad \forall \lambda \in  \rm{Vert}_n(1/4) \, ,
$$
and
$$
| \mathcal H_\lambda(u) | \le  \frac{1}{|\lambda|}  \cdot \frac{4}{3} \ , \qquad \forall \lambda \in  \rm{Vert}_n(1/4) \, .
$$
 \end{lemma}
 \begin{proof}
Let  $u \in B^{-s}_{c,0}(r_{s,*})$ and $n \ge 0$. 
Expanding $(L_u - \lambda)^{-1} = (D - \lambda)^{-1}(Id - T_u(D - \lambda)^{-1})^{-1}$ in its Neumann series
one gets
$$
\mathcal H_\lambda(u) = - \frac{1}{\lambda} \big( 1 + \sum_{k \ge 1} \langle (T_u (D - \lambda)^{-1})^k 1 | 1 \rangle \big) \, .
$$
 Since by Corollary \ref{1/2 bound},  
$$
\| T_u(D - \lambda)^{-1} \|_{H^{-s; n}_+ \to H^{-s; n}_+} \le \frac{1}{4} \, , \qquad \forall  \lambda \in  \rm{Vert}_n(1/4),
$$
and by the definition \eqref{shifted norm},   $ \| 1 \|_{-s; n} = \langle n \rangle^{-s}$ as well as  $ \| 1 \|_{s; n} = \langle n \rangle^{s} $,
 it follows that 
$$
|\langle (T_u(D - \lambda)^{-1})^k 1 | 1 \rangle | \le  (\frac{1}{4})^k \| 1 \|_{-s; n} \| 1 \|_{s; n}  \le (\frac{1}{4})^k
$$
and hence
$$
| \mathcal H_\lambda(u) | \ge  \frac{1}{|\lambda|} \big( 1 - \sum_{k \ge 1} (\frac{1}{4})^k \big) =  \frac{1}{|\lambda|} \cdot \frac{2}{3} \ , \qquad \forall \lambda \in  \rm{Vert}_n(1/4) \, .
$$
 Similarly, one gets
 $$
 | \mathcal H_\lambda(u) | \le  \frac{1}{|\lambda|} \big( 1 + \sum_{k \ge 1} (\frac{1}{4})^k \big) =  \frac{1}{|\lambda|} \cdot \frac{4}{3} \ , \qquad \forall \lambda \in  \rm{Vert}_n(1/4) \, .
 $$
 \end{proof}
 
 \begin{lemma}\label{argument principle}
 Let $u \in B^{-s}_{c,0}(r_{s,*})$. Then for any $n \ge 1$, the difference $ZP_{u;n}$ of the number of zeroes of $\mathcal H_\lambda(u)$ in $D_n(1/3)$
 and the number of its poles  in $D_n(1/3)$ vanishes. For $n=0,$ the difference $ZP_{u; 0}$ equals $-1$.
 As a consequence, $\mathcal H_\lambda(u)$ has no zeroes in $D_0(1/3)$ and for any $n \ge 1$, at most one zero in $D_n(1/3)$.
 \end{lemma}
 \begin{proof}
 Let $u$ be in $B^{-s}_{c,0}(r_{s,*})$ with $0 \le s < 1/2$ and $n \ge 0.$ 
 By Lemma \ref{estimate of mathcal H},  $ \mathcal H_\lambda(u)$ has no poles and no zeroes
 on the circle $\partial D_n(1/3)$. Since $ \mathcal H_\lambda(u)$ is meromorphic, it then
 follows by the argument principle that $ZP_{u;n}$ is constant on $B^{-s}_{c,0}(r_{s,*})$.
 Recall that for $u=0$, one has $ \mathcal H_\lambda(0) = - 1/\lambda$, implying that
 $ZP_{0; 0} = -1$ and for any $n \ge 1$, $ZP_{0; n} = 0$.
 \end{proof}
 
 \begin{corollary}\label{residue and gamma}
 Assume that $u \in B^{-s}_{c,0}(r_{s,*})$ with $0 \le s < 1/2$. For any $n \ge 1,$  $F_n(u) = 0$ if and only if $\gamma_n(u)=0$.
 \end{corollary}
% \begin{remark}
% Later in this section, we will derive a formula of the type $F_n(u) = \kappa_n(u) \gamma_n(u)$ where
% $\kappa_n$ is a non-vanishing analytic function on  $B^{-s}_{c,0}(r_{s,*})$. 
% \end{remark}
 \begin{proof} Let $u$ be in $B^{-s}_{c,0}(r_{s,*})$ with $0 \le s < 1/2$ and $n \ge 1.$ 
 Assume that $\gamma_n(u) \ne 0$. By Lemma \ref{zeros of mathcal H}, the generating function
 vanishes at $\lambda = \lambda_{n-1}(u) + 1$ and hence by Lemma \ref{argument principle},
 $\lambda_n(u)$ is a pole of  $\mathcal H_{\lambda}$, implying together with Lemma \ref{mathcal H as a sum}(ii) that  $F_n(u) \ne 0$.
 To prove the converse, assume that $\gamma_n(u) = 0.$ Then by Lemma \ref{characterization gamma_n = 0}(ii) 
 we conclude that $\langle P_n(u) 1 | 1 \rangle = 0$
 and hence by Lemma \ref{mathcal H as a sum}(i)  that $F_n(u) = 0$. 
 \end{proof}
  
  For any $u \in B^{-s}_{c,0}(r_{s,*})$, define for $n \ge 1$ and $\lambda \in \rm{Vert}_n(1/4)$,
  \begin{equation}\label{def eta near zero}
  \eta_n(\lambda, u) := - 
  \frac{\lambda -  \lambda_n(u)}{\lambda -  \lambda_{n-1} (u) - 1 } (\lambda - \lambda_0(u))\mathcal H_\lambda(u) \, ,
 \end{equation}
 and for $n = 0$, 
 $$
  \eta_0(\lambda, u) := -   (\lambda - \lambda_0(u))\mathcal H_\lambda(u) \, ,
  \qquad \forall \lambda \in \rm{Vert}_0(1/4).
  $$
The functions $\eta_n$, $n \ge 0$, are analytic in $\lambda$ on their domains of definition. 
By Lemma \ref{zeros of mathcal H}  and Lemma \ref{argument principle}, for any $n \ge 0,$
the function $\eta_n(\lambda, u)$ extends analytically to $D_n(1/4)$ and hence for any $\lambda \in D_n(1/4)$
one has by Cauchy's formula
\begin{equation}\label{Cauchy eta_n}
\eta_n(\lambda, u) = \frac{1}{2\pi i} \int_{\partial D_n(1/3)} \frac{\eta_n(\mu, u)}{\mu - \lambda} d \mu .
\end{equation}
As a consequence, for any $n \ge 0$,
$$
\eta_n : \big( D_n(1/4) \cup \rm{Vert}_n(1/4) \big) \times B^{-s}_{c,0}(r_{s,*}) \to \C
$$
is analytic and by Lemma \ref{argument principle} vanishes nowhere.
Since $ \mathcal H_\lambda(0) = - 1/\lambda$
one infers that  for any $n \ge 0$,  $\eta_n(\lambda, 0) = 1$  for $\lambda$ in $\rm{Vert}_n(1/4)$.
\begin{lemma}\label{estimate of eta_n}
For any $u \in B^{-s}_{c,0}(r_{s,*})$ and $n \ge 0$,
$$
\frac{1}{C} \le | \eta_n(\lambda, u) | \le C\, , \qquad \forall \lambda \in D_n(1/4) \, ,
$$
where $C:=  5 \cdot 7 \cdot 4$.
\end{lemma}
\begin{proof}
Let  $u \in B^{-s}_{c,0}(r_{s,*})$ and $n \ge 0$. 
By \eqref{Cauchy eta_n}
$$
\eta_n(\lambda, u) =  \frac{1}{2\pi i} \int_{\partial D_n(1/3)} 
\frac{\eta_n(\mu, u)}{\mu - \lambda} d \mu \, , \qquad \forall \lambda \in D_n(1/4) \, .
$$
We first prove the claimed upper bound for $| \eta_n(\lambda, u) |$. 
Let $n \ge 1.$
Since for any $k \ge 0,$ $\lambda_k(u) \in D_k(1/4)$ and $\lambda_k(u) + 1 \in D_{k+1}(1/4)$ one has
for any $\mu \in \partial D_n(1/3)$ and $\lambda \in D_n(1/4)$
$$
  \frac{| \mu - \lambda_n(u)|}{| \mu - \lambda_{n-1}(u) -1|}  \frac{1}{| \mu - \lambda |} 
 \le   \frac{\frac{1}{3} + \frac{1}{4}}{(\frac{1}{3} - \frac{1}{4})^2}  = 7 \cdot 12 
$$
and  by Lemma \ref{estimate of mathcal H} 
$$ 
| (\mu - \lambda_0 )\mathcal H_\mu(u)| \le \frac{n + \frac{1}{3} + \frac{1}{4}}{n - \frac{1}{3}} \  \frac{4}{3} \le 4 \, .
$$
Hence for any $\mu \in \partial D_n(1/3)$ and $\lambda \in D_n(1/4)$, one has by the definition of $\eta_n$
$$
\frac{ |\eta_n(\mu, u)|} {|\mu - \lambda |} =  \frac{| \mu - \lambda_n(u)|}{| \mu - \lambda_{n-1}(u) -1|} 
\frac{| (\mu - \lambda_0(u) ) \mathcal H_\mu(u)|}{|\mu - \lambda |}
 \le  4 \cdot 7 \cdot 12 \, .
$$
Formula \eqref{Cauchy eta_n} then yields
$$
 |\eta_n(\lambda, u)|  \le  C\, , \qquad  \forall \lambda \in D_n(1/4)\, ,
$$
where $C=  5 \cdot 7 \cdot 4$. In case $n = 0$, one argues similarly
and obtains
$$
\frac{ |\eta_0(\mu, u)|} {|\mu - \lambda |} = 
\frac{| (\mu - \lambda_0(u)) \mathcal H_\mu(u)|}{|\mu - \lambda |}
 \le  \frac{ \frac{1}{3} + \frac{1}{4}}{ \frac{1}{3}} \cdot \frac{4}{3} \cdot 12 \le C \, ,
$$
yielding also $ |\eta_0(\lambda, u)|  \le  C$ for any $\lambda \in D_n(1/4)$.
It remains to prove the claimed lower bound for $|\eta_n(\lambda, u) |$.
Since $\eta_n(\lambda, u)$ is analytic and does not vanish in 
$D_n(1/4) \cup \rm{Vert}_n(1/4)$, its inverse is analytic on $D_n(1/4) \cup \rm{Vert}_n(1/4)$. 
By Cauchy's theorem we have for any $\lambda \in D_n(1/4)$
\begin{equation}\label{Cauchy for 1/eta}
\frac{1}{\eta_n(\lambda, u)} =  \frac{1}{2\pi i} \int_{\partial D_n(1/3)} \frac{1}{\eta_n(\mu, u)} \frac {1}{\mu - \lambda} d \mu .
\end{equation}
Let us again first consider the case $n \ge 1$. For any $\mu \in \partial D_n(1/3)$ and $\lambda \in D_n(1/4)$,
$$
  \frac{| \mu - \lambda_{n-1}(u) -1|}{| \mu - \lambda_n(u)|}  \frac{1}{| \mu - \lambda |} 
 \le   \frac{\frac{1}{3} + \frac{1}{4}}{(\frac{1}{3} - \frac{1}{4})^2}  = 7 \cdot 12 
$$
and  by Lemma \ref{estimate of mathcal H} 
$$ 
\frac{1}{|(\mu - \lambda_0)\mathcal H_\mu(u)|} \le \frac{|\mu |}{| \mu - \lambda_0 |} \  \frac{3}{2}
 \le  \frac{n + \frac{1}{3}}{n - \frac{1}{3} - \frac{1}{4}} \  \frac{3}{2} \le 5 \, .
$$
Hence for any $\mu \in \partial D_n(1/3)$ and $\lambda \in D_n(1/4)$, one has by the definition of $\eta_n$
$$
\begin{aligned}
\frac{1}{ |\eta_n(\mu, u)|} \frac{1}{|\mu - \lambda |} & =  
\frac{| \mu - \lambda_{n-1}(u) -1|}{| \mu - \lambda_n(u)|}
\frac{1}{| (\mu - \lambda_0(u)) \mathcal H_\mu(u)|} \frac{1}{|\mu - \lambda |} \\
& \le  7 \cdot 12 \cdot 5 \, .
 \end{aligned}
$$
Formula \eqref{Cauchy for 1/eta} then yields
$$
\frac{1}{ |\eta_n(\lambda, u)|}  \le  C\, , \qquad  \forall \lambda \in D_n(1/4) \, ,
$$
where $C$ is again the constant $ 5 \cdot 7 \cdot 4$. In case $n = 0$, one argues 
similarly and obtains
$$
\begin{aligned}
\frac{1}{ |\eta_0(\mu, u)|} \frac{1}{|\mu - \lambda |} & = 
\frac{1}{| (\mu - \lambda_0(u) ) \mathcal H_\mu(u)|} \frac{1}{|\mu - \lambda |} \\
&  \le  \frac{|\mu |}{| \mu -  \lambda_0  |} \  \frac{3}{2} 
\le \frac{1}{3} \cdot 12 \cdot \frac{3}{2} \le C \, ,
 \end{aligned}
$$
yielding also $ \frac{1}{|\eta_0(\lambda, u)|}  \le  C$ for any $\lambda \in D_n(1/4)$.
\end{proof}

\smallbreak

Define for any $n \ge 1$
\begin{equation}\label{def kappa_n near zero}
\kappa_n : B^{-s}_{c,0}(r_{s,*}) \to \C, \ u \mapsto  \frac{1}{\lambda_n(u) - \lambda_0(u)} \eta_n(\lambda_n(u), u) 
\end{equation}
and for $n = 0,$
$$
\kappa_0 : B^{-s}_{c,0}(r_{s,*}) \to \C, \ u \mapsto    \eta_0(\lambda_0(u), u) \, .
$$
Being a composition of analytic functions, the $\kappa_n$ are analytic maps.
They satisfy the following estimates.
\begin{proposition}\label{proposition kappa_n}
For any $n \ge 0,$ the map $\kappa_n :  B^{-s}_{c,0}(r_{s,*}) \to \C$ is analytic
and satisfies for any $u \in B^{-s}_{c,0}(r_{s,*})$ the estimate
$$
| \kappa_n(u) | \le 2 C \frac{1}{\langle n \rangle} \, , \qquad
\frac{1} {| \kappa_n(u) |} \le 2C \langle n \rangle \, ,  \qquad  \forall \, n \ge 0 \, ,
$$
where $C =  5 \cdot 7 \cdot 4$ is the constant of Lemma \ref{estimate of eta_n}.
In particular, the map $B^{-s}_{c,0}(r_{s,*}) \to \ell^{\infty}_+, u \mapsto (\frac{1}{n \kappa_n(u)})_{n \ge 1}$ 
is analytic.
\end{proposition}
\begin{proof}
Let $n \ge 0$. We have already seen that $\kappa_n$ is analytic. In the case $u$ real valued,
$\kappa_n(u)$ is real valued as well (cf. also \cite{GK}($s=0$), \cite{GKT}($0< s < 1/2$)).
Let  $u \in B^{-s}_{c,0}(r_{s,*})$. 
By the definition of $\kappa_n(u)$ and Lemma \ref{estimate of eta_n}, the claimed estimates  hold
for $n = 0$. If $n \ge 1$, then
$$
|\kappa_n(u)| =  \frac{1}{|\lambda_n(u) -  \lambda_0(u)|} | \eta_n(\lambda_n(u), u) | 
\le \frac{1}{n - \frac{1}{2}} C \le \frac{2C}{n}
$$
and
$$
|\kappa_n(u)| =  \frac{1}{|\lambda_n(u) -  \lambda_0(u)|} |\eta_n(\lambda_n(u), u) | 
\ge \frac{1}{n + \frac{1}{2}} \,  \frac{1}{C} \ge  \frac{1}{2C} \frac{1}{n} \, .
$$
By \cite[Theorem A.3]{KP}, the latter estimate implies that the map
$$
B^{-s}_{c,0}(r_{s,*}) \to \ell^{\infty}_+, u \mapsto (\frac{1}{n \kappa_n(u)})_{n \ge 1}
$$ 
is analytic.
\end{proof}
When combined with our results on the residues $F_n(u)$ 
of the generating function and the gap lengths $\gamma_n(u)$, Proposition \ref{proposition kappa_n} yields the following 
\begin{corollary}\label{identity residues}
Let $0 \le s < 1/2.$
  For any $u$ in $B^{-s}_{c,0}(r_{s, *})$,
  \begin{equation}\label{identity quasi moments near 0}
  F_n(u) = - \kappa_n(u) \gamma_n(u)\, , \qquad \forall n \ge 1\,  .
  \end{equation}
\end{corollary}
\begin{proof}
 Let $n \ge 1$.
 The identity \eqref{identity quasi moments near 0} holds for $u$ real valued in $B^{-s}_{c,0}(r_{s, *})$
 (cf.  \eqref{identity1 for F_n for u real} and Remark \ref{identity2 for F_n for u real}).
 By Proposition \ref{F analytic near 0}, Proposition \ref{proposition kappa_n}, and Theorem \ref{main result of section 2},
the functions $F_n,$ $\kappa_n$, and $\gamma_n$ are analytic on $B^{-s}_{c,0}(r_{s, *})$
 and hence the identity continues to hold on $B^{-s}_{c,0}(r_{s, *})$.
 \end{proof}
 
We now have all the ingredients to prove Theorem \ref{main theorem}(ii) for potentials in a neighborhood of $0$ in $H^{-s}_{c,0}$.
It reads in this case as follows.
 
 \begin{theorem}\label{main theorem near zero}
Let $0 \le s < 1/2$ and let $r_{s,*}>0 $ be the radius, given by Proposition \ref{F analytic near 0}.
Then the moment map
$$
\Gamma : B^{-s}_{c, 0}(r_{s, *}) \to \ell^{1, 1- 2s}_+, \ u \mapsto (\gamma_n(u))_{n \ge 1}
$$
is analytic. In addition, $\Gamma(B^{-s}_{c, 0}(r_{s, *}))$ is bounded in $\ell^{1, 1- 2s}_+$.
\end{theorem}
\begin{proof}
Since 
 $B^{-s}_{c,0}(r_{s,*}) \to \ell^{1, 2-2s}_+, u \mapsto (F_n(u))_{n \ge 1}$ (Proposition \ref{F analytic near 0})
 and
$B^{-s}_{c,0}(r_{s,*}) \to \ell^{\infty}_+, u \mapsto (\frac{1}{n \kappa_n(u)})_{n \ge 1}$ 
(Proposition \ref{proposition kappa_n})
 are analytic, it follows from  \eqref{identity quasi moments} that
 $\Gamma : B^{-s}_{c,0}(r_{s, *})  \to \ell^{1, 1-2s}_+$ is analytic.
 Furthermore, from Proposition \ref{F analytic near 0} and Proposition \ref{proposition kappa_n}
 one infers that $\Gamma(B^{-s}_{c, 0}(r_{s, *}))$ is bounded in $\ell^{1, 1- 2s}_+$.
 \end{proof}
 
  Theorem \ref{main result of section 2} and Theorem \ref{main theorem near zero} are used to prove Corollary \ref{product representations} for potentials near zero.
  We reformulate it slightly and include a result on the product representation of $\kappa_n(u)$, $n \ge 0$:
 \begin{corollary}\label{product representations near zero}
 Let $0 \le s < 1/2$ and let $r_{s,*}>0 $ be the radius, given by Proposition \ref{F analytic near 0}.
  For any $u$ in $B^{-s}_{c,0}(r_{s, *})$, the following holds:\\
(i) 
The generating function $\mathcal H_\lambda(u)$ admits the  product representation
$$
\mathcal H_\lambda(u) = \frac{1}{\lambda_0(u) - \lambda} \prod_{p \ge 1} \Big(1 - \frac{\gamma_p(u)}{\lambda_p(u) - \lambda} \Big) \, 
$$
where the infinite product is absolutely convergent.\\
(ii)   For any $n \ge 0,$
 $$
 \lambda_n(u) = n -  \sum_{k \ge n+1} \gamma_k(u)
 $$
 where the infinite sum is absolutely convergent.\\
 (iii)  If $s=0,$ one has
 $$
\frac{1}{2\pi} \int_0^{2\pi} u^2 dx = 2 \sum_{k \ge 1} k \gamma_k 
 $$
 where the infinite sum is absolutely convergent.\\
 (iv) For any $n \ge 1$, the function $\kappa_n(u)$  admits the  product representation
$$
\kappa_n(u) = \frac{1}{\lambda_n(u) - \lambda_0(u)}  \prod_{p \ne n} \Big(1 - \frac{\gamma_p(u)}{\lambda_p(u) - \lambda_n(u)} \Big)
$$
whereas for $n=0$ one has
$$
\kappa_0(u) =  \prod_{p \ge 1} \Big(1 - \frac{\gamma_p(u)}{\lambda_p(u) - \lambda_0(u)} \Big) \, .
$$
All these infinite products are absolutely convergent.\\
For $u=0$, one has
 $\kappa_0(0) =1$ and $\kappa_n(0) = 1/n$ for any $n \ge 1$.\\

 \end{corollary}
 \begin{proof}
(i)  First note that the claimed identity holds for real valued poentials $u \in B^{-s}_{c,0}(r_{s, *})$
 \big(cf. \cite[Proposition 3.1]{GK} $(s=0)$, \cite[(32)]{GKT} $(0 < s < 1/2)$\big)
 and that for any $\lambda \in \C$ with  $\Re \lambda < -1/3$, 
$\mathcal H_\lambda(u)$ is analytic on $B^{-s}_{c,0}(r_{s, *})$. Since
 $\lambda_p(u) \in D_n(1/4)$ is analytic for any $p \ge 0$ and by Theorem \ref{main theorem near zero},
 $\Gamma : B^{-s}_{c,0}(r_{s, *})  \to \ell^{1, 1-2s}_+$ 
 is analytic, it then follows that for any such $\lambda$, the infinite product $\prod_{p \ge 1} \big(1 - \frac{\gamma_p(u)}{\lambda_p(u) - \lambda} \big)$
 is absolutely convergent and induces an analytic map $B^{-s}_{c,0}(r_{s, *}) \to \C$. Altogether one concludes that the claimed
 product representation holds on $B^{-s}_{c,0}(r_{s, *})$.
 Item (ii), (iii), and (iv) are proved by similar arguments, using that the stated identities hold for real valued $u$:
 For item (ii) cf. \cite[(3.13)]{GK} $(s = 0)$, \cite[(29)]{GKT} $(0 < s < 1/2)$,
 for item (iii) cf. \cite[Proposition 3.1]{GK}, and for  item (iv) cf. \cite[Corollary 3.4]{GK} $(s = 0)$, \cite[(26), (27)]{GKT}$(0 < s < 1/2)$.
  \end{proof}
  
  \medskip
  
  Theorem \ref{main theorem near zero} and Corollary \ref{product representations near zero} 
  lead to estimates of the maps $\kappa_n,$ $n \ge 1,$
  which will be used in subsequent work for the analytic extension of the Birkhoff map, constructed on $H^{-s}_{r,0}$ 
   in \cite{GK} $(s = 0)$ and \cite{GKT} $(0 < s < 1/2)$.\\
   By  Theorem \ref{main theorem near zero}, for any $0 \le s < 1/2$, there exists $0 <  r_{s, **} \le r_{s, *}$ so that
   \begin{equation}\label{def r_ s,**}
   \sum_{k \ge 1} |\gamma_k(u)| \le \frac{1}{5} \ , \qquad \forall u \in B^{-s}_{c,0}(r_{s, **}) \, .
  \end{equation}
   \begin{proposition}\label{estimate for kappa_n}
   Let $0 \le s < 1/2.$
  For any $u$ in $B^{-s}_{c,0}(r_{s, **})$,
  $$
 | n\kappa_n(u) - 1 | \le \frac{7}{12} e^{1/3} < 1 \,  , \qquad \forall n \ge 1\, .
  $$ 
  Hence  for any $n \ge 1$, the principal  branch of the square root of $n\kappa_n$ is well defined
  on $B^{-s}_{c,0}( r_{s, **})$ and it follows that
  $$
 \big(\sqrt{n \kappa_n} \big)_{n \ge 1} :  B^{-s}_{c,0}(r_{s, **}) \to \ell^\infty_+ , \, 
 u \mapsto   \big(\sqrt{n \kappa_n(u)} \big)_{n \ge 1}
  $$
 is analytic.
 \end{proposition}
 \begin{proof}
 Let $0 \le s < 1/2$, $n \ge 1$, and $u \in B^{-s}_{c,0}(r_{s, **})$.
 To simplify notation, we do not indicate the dependence on $u$ in the course of the proof. 
 By the product representation of $\kappa_n$, $n \ge 1$, of Corollary \ref{product representations near zero}(iv), $ n \kappa_n - 1 =  I_n + II_n$ where
  $$
 I_n := \big (\frac{n}{\lambda_n - \lambda_0} -1 \big)\prod_{p \ne n} \big(1 - \frac{\gamma_p}{\lambda_p - \lambda_n} \big) \, , \qquad
II_n := \prod_{p \ne n} \big(1 - \frac{\gamma_p}{\lambda_p - \lambda_n} \big) - 1 \, .
 $$
 We begin with estimating the term $I_n$.
 By Corollary \ref{product representations near zero}(ii), 
 $\lambda_n - \lambda_0 = n + \sum_{k=1}^{n} \gamma_k$ and hence
 $$
 \frac{n}{\lambda_n - \lambda_0} - 1 = - \frac{\frac{1}{n}  \sum_{k=1}^{n} \gamma_k }{ 1 +  \frac{1}{n} \sum_{k=1}^{n} \gamma_k} \, .
 $$ 
 Since by assumption,  $\sum_{k \ge 1} |\gamma_k| \le \frac{1}{5}$, it follows that
 $$
 | \frac{n}{\lambda_n - \lambda_0} - 1 | \le \frac{1/5}{1 - 1/5} \le \frac{1}{4}.
 $$
 Next we estimate the infinite product $\prod_{p \ne n} \big(1 - \frac{\gamma_p}{\lambda_p - \lambda_n} \big)$.
 One has
 $$
 \begin{aligned}
| \prod_{p \ne n} \big(1 - \frac{\gamma_p}{\lambda_p - \lambda_n} \big)|  & \le 
 \prod_{p \ne n} \big(1 + \frac{|\gamma_p|}{|\lambda_p - \lambda_n|} \big) \\
& \le \exp\big(  \sum_{p \ne n} \log(1 + \frac{|\gamma_p|}{|\lambda_p - \lambda_n|})\big) \, .
\end{aligned}
 $$
 Since $\log(1 +  |x|) \le |x|$ and 
 $$
 |\lambda_p - \lambda_n| \ge  |p - n| - \sum_{k=1}^{n} |\gamma_k| \ge 1 - 1/5\, ,
 $$ 
 one infers that
 $$
 | \prod_{p \ne n} \big(1 - \frac{\gamma_p}{\lambda_p - \lambda_n} \big)|   \le
\exp\big(  \frac{5}{4} \sum_{p \ne n} |\gamma_p| \big)  \le e^{1/4} \, .
$$ 
Altogether we have shown that
$$
| I_n | \le \frac{1}{4} e^{1/4}\, , \qquad \forall \, n \ge 1\, .
$$
 It remains to estimate $II_n =  \prod_{p \ne n} \big(1 - \frac{\gamma_p}{\lambda_p - \lambda_n} \big) - 1$.
 First note that since for any $p \ne n,$
 $$
 | \frac{\gamma_p}{\lambda_p - \lambda_n} | \le \frac{|\gamma_p|}{4/5}  \le 1/4\, , 
 $$
 the principal branch of the logarithm of $ 1 -  \frac{\gamma_p}{\lambda_p - \lambda_n}$
 is well defined and hence one has
 $$
  \prod_{p \ne n} \big(1 - \frac{\gamma_p}{\lambda_p - \lambda_n} \big) = 
 \exp\big(  \sum_{p \ne n} \log\big(1 - \frac{\gamma_p}{\lambda_p - \lambda_n}\big) \big) \, .
 $$
 Using that for any $x \in \C$, 
 $$
 |e^x - 1 | = | x \int_0^1 e^{tx} \ d t | \le |x| e^{|x|}
 $$ 
 and that for any $y \in \C$ with $|y| \le 1/4$,
 $$
| \log (1 - y) | = | \int_0^1 \frac{-y}{1 - ty} \ d t | \le \frac{4}{3} |y| \, ,
 $$
one concludes that
$$
|  \prod_{p \ne n} \big(1 - \frac{\gamma_p}{\lambda_p - \lambda_n} \big) - 1 |
\le  \sigma \exp \sigma \, , \qquad  \sigma:= \frac{4}{3} \sum_{p \ne n} \frac{|\gamma_p|}{|\lambda_p - \lambda_n|} \, , 
$$
 where
 $$
 \sigma \le \frac{4}{3} \frac{5}{4} \sum_{p \ne n} |\gamma_p| \le \frac{1}{3}  \, .
 $$
 Combining the estimates derived yields the desired bound,
 $$
 | n \kappa_n(u) - 1  | \le  \frac{1}{4} e^{1/4} +  \frac{1}{3} e^{1/3} \le  \frac{7}{12} e^{1/3} \, .
 $$
 The claimed analyticity of the map $\big(\sqrt{n \kappa_n} \big)_{n \ge 1}$ then follows from \cite[Theorem A.3]{KP}.
 \end{proof}  
   
\medskip 
  
We finish this section with a discussion of the spectral invariants $\mu_n \equiv \mu_n(u)$, $n \ge 1$,
 introduced in \cite[(4.9)]{GK} for potentials $u \in L^2_{r,0}$,
  $$
  \mu_n = \big( 1 - \frac{\gamma_n}{\lambda_n - \lambda_0} \big) 
  \prod_{p \ne n} \frac{1 - \frac{\gamma_p}{\lambda_p - \lambda_n}}{1 - \frac{\gamma_p}{\lambda_p - \lambda_{n-1} - 1}}.
  $$
  By algebraic transformations one obtains
  $$
    \mu_n = \big( 1 - \frac{\gamma_n}{\lambda_n - \lambda_0} \big) 
  \prod_{p \ne n} \big(1 - \frac{\gamma_n}{\lambda_{p-1} - \lambda_{n-1}} \big)\big(1 + \frac{\gamma_n}{\lambda_{p} - \lambda_{n}} \big)\, ,
  $$
  or, by writing $\lambda_{p-1} - \lambda_{n-1} = \lambda_{p} - \lambda_{n}  + \gamma_n - \gamma_p$,
  \begin{equation}\label{formula m_n for u real}
    \mu_n = \big( 1 - \frac{\gamma_n}{\lambda_n - \lambda_0} \big) 
  \prod_{p \ne n} \big(1 - \frac{\gamma_n \gamma_p}{(\lambda_{p-1} - \lambda_{n-1})(\lambda_{p} - \lambda_{n})} \big) \, .
  \end{equation}
  By Theorem \ref{main theorem near zero} one sees that  $\mu_n$, $n \ge 1$, extend analytically to $B^{-s}_{c,0}(r_{s,*})$
  for any $0 \le s < 1/2$. Note that for any $u \in B^{-s}_{c,0}(r_{s,*})$ and $n \ge 1$, 
  $\mu_n(u) \ne 0$ and in case $\gamma_n(u) = 0$, $\mu_n(u)=1.$\\
 
  Theorem \ref{main theorem near zero} and Corollary \ref{product representations near zero} lead to estimates 
  of the maps $\mu_n,$ $n \ge 1,$
  which will be used in subsequent work to show that for any $0 \le s < 1/2$, the Birkhoff map analytically extends to a neighborhood of zero in $H^{-s}_{c,0}$.
  \begin{proposition}\label{proposition mu_n}
Let $0 \le s < 1/2$ and  let $r_{s, **}$ be given as in  \eqref{def r_ s,**}.
Then for any $u$ in $B^{-s}_{c,0}(r_{s, **})$,
  $$
| \mu_n(u) - 1 | \le \frac 53 e^{1/15} | \gamma_n(u) | \le \frac13 e^{1/15} < \frac{1}{2} \, ,   \qquad \forall n \ge 1\, .
 $$ 
 Hence for any $n \ge 1$, the principal  branch of the square root of $\mu_n$ is well defined
on $B^{-s}_{c,0}(r_{s, **})$ and it follows that
 $$
 \big(\sqrt{\mu_n} \big)_{n \ge 1} :  B^{-s}_{c,0}(r_{s, **}) \to \ell^\infty_+ , \ 
 u \mapsto   \big(\sqrt{\mu_n(u)} \big)_{n \ge 1}
 $$
is analytic.
\end{proposition}  
 \begin{proof}
 We argue as in the proof of Proposition \ref{estimate for kappa_n}.
  Let $0 \le s < 1/2$, $n \ge 1$, and  $u \in B^{-s}_{c,0}(r_{s, **})$.
  Again, to simplify notation, we do not indicate the dependence on $u$ in the course of the proof. 
 By the definition of $\mu_n$
 one has $\mu_n -1 = I_n + II_n$ where
 $$
 I_n:=   \big( \big( 1 - \frac{\gamma_n}{\lambda_n - \lambda_0} \big)  - 1 \big)
  \prod_{p \ne n} \big(1 - \frac{\gamma_n \gamma_p}{(\lambda_{p-1} - \lambda_{n-1})(\lambda_{p} - \lambda_{n})} \big),
  $$
  $$
  II_n:=   \prod_{p \ne n} \big(1 - \frac{\gamma_n \gamma_p}{(\lambda_{p-1} - \lambda_{n-1})(\lambda_{p} - \lambda_{n})} \big) - 1 \, . \qquad \qquad \qquad \quad
 $$
 We begin with estimating the term $I_n$.
 Since by Corollary \ref{product representations near zero}(ii), 
 $\lambda_n - \lambda_0 = n + \sum_{k=1}^{n} \gamma_k $ and 
 by asssumption $ \sum_{k=1}^{n} |\gamma_k | \le 1/5$ one has
 $$
| \big( 1 - \frac{\gamma_n}{\lambda_n - \lambda_0} \big)  - 1|
 =| -  \frac{\gamma_n}{ n +  \sum_{k=1}^{n} \gamma_n}| 
 \le \frac{|\gamma_n|}{ n - 1/5}
 \le \frac{5}{4} | \gamma_n | \, .
 $$ 
  Next we estimate the infinite product $\prod_{p \ne n}  \big(1 - \frac{\gamma_n \gamma_p}{(\lambda_{p-1} - \lambda_{n-1})(\lambda_{p} - \lambda_{n})} \big)$.
 One has
 $$
 \begin{aligned}
| \prod_{p \ne n} \big(1  & - \frac{\gamma_n \gamma_p}{(\lambda_{p-1} - \lambda_{n-1})(\lambda_p - \lambda_n)} \big)|  \\
& \le 
 \prod_{p \ne n} \big(1 + |\gamma_n| \frac{|\gamma_p|}{|\lambda_{p-1} - \lambda_{n-1}| |\lambda_p - \lambda_n|} \big) \\
& \le \exp\big(  \sum_{p \ne n} \log(1 + |\gamma_n| \frac{|\gamma_p|}{|\lambda_{p-1} - \lambda_{n-1}| |\lambda_p - \lambda_n|})\big) \, .
\end{aligned}
 $$
 Since $\log(1 +  |x|) \le |x|$ and for any $p \ne n,$
 $$
 |\lambda_{p-1} - \lambda_{n-1}| ,  \  |\lambda_p - \lambda_n| \ge  |p - n| - \sum_{k=1}^{\infty} |\gamma_k| \ge 1 - 1/5 \, ,
 $$ 
 one infers that
 $$
 | \prod_{p \ne n} \big(1 - \frac{\gamma_n \gamma_p}{ (\lambda_{p-1} - \lambda_{n-1})(\lambda_p - \lambda_n)} \big)|   \le
\exp\big(  |\gamma_n| \frac{5}{4}   \frac{5}{4} \sum_{p \ne n} |\gamma_p| \big)  \le e^{1/16} \, .
$$ 
 It remains to estimate $II_n =  \prod_{p \ne n}  \big(1 - \frac{\gamma_n \gamma_p}{(\lambda_{p-1} - \lambda_{n-1})(\lambda_{p} - \lambda_{n})} \big)  - 1$.
 First note that since for any $p \ne n,$
 $$
 | \frac{\gamma_n \gamma_p}{(\lambda_{p-1} - \lambda_{n-1})(\lambda_p - \lambda_n)} | \le \frac{|\gamma_n|}{4/5}  \frac{|\gamma_p|}{4/5}  \le 1/16\, , 
 $$
 the principal branch of the logarithm of $\big(1 - \frac{\gamma_n \gamma_p}{(\lambda_{p-1} - \lambda_{n-1})(\lambda_{p} - \lambda_{n})} \big)$
 is well defined and hence one has
 $$
 \begin{aligned}
  & \prod_{p \ne n} \big( 1 - \frac{\gamma_n \gamma_p}{(\lambda_{p-1} - \lambda_{n-1})(\lambda_{p} - \lambda_{n})}  \big) \\
  & = 
 \exp\big(  \sum_{p \ne n} \log\big(1 - \frac{\gamma_n \gamma_p}{(\lambda_{p-1} - \lambda_{n-1})(\lambda_{p} - \lambda_{n})} \big) \big) \, .
 \end{aligned}
 $$
 Note that for any $x \in \C$, 
 $$
 | e^x - 1 | = | x \int_0^1 e^{tx} \ d t | \le |x| e^{|x|}
 $$ 
 and that for any $y \in \C$ with $|y|  < 1$,
 $$
| \log (1 - y) | = | \int_0^1 \frac{-y}{1 - ty} \ d t | \le \frac{1}{1-|y|} |y|\, .
 $$
 Since $y := \frac{\gamma_n \gamma_p}{(\lambda_{p-1} - \lambda_{n-1})(\lambda_{p} - \lambda_{n})}$ satisfies $|y| \le \frac{\frac 15 \frac 15}{\frac 45 \frac 45} \le \frac{1}{16}$ one has
 $$
 \frac{1}{1 - |y|} |y| \le \frac{16}{15} \frac 54 |\gamma_n|  \frac{|\gamma_p|}{|\lambda_{p} - \lambda_{n}|} \le \frac 43 |\gamma_n|  \frac{|\gamma_p|}{|\lambda_{p} - \lambda_{n}|} \, .
 $$
One concludes that
$$
|  \prod_{p \ne n} \big(1 -  \frac{\gamma_n \gamma_p}{(\lambda_{p-1} - \lambda_{n-1})(\lambda_{p} - \lambda_{n})}  \big) - 1 |
\le \frac{4}{3} |\gamma_n| \sigma \exp(\frac{4}{3} |\gamma_n| \sigma) \, , 
$$
where
$ \sigma:= \sum_{p \ne n} \frac{|\gamma_p|}{|\lambda_p - \lambda_n|}$
can be estimated as
 $ \sigma \le \frac{5}{4} \sum_{p \ne n} |\gamma_p| \le \frac{1}{4},$
 implying that
 $$
 |  \prod_{p \ne n} \big(1 -  \frac{\gamma_n \gamma_p}{(\lambda_{p-1} - \lambda_{n-1})(\lambda_{p} - \lambda_{n})}  \big) - 1 |
\le \frac{1}{3} |\gamma_n|  e^{1/15} \ .
 $$
 Combining all the estimates derived yields the desired bound,
 $$
 |\mu_n(u) - 1  | \le  \frac{5}{4} |\gamma_n| e^{1/16} +  \frac{1}{3} |\gamma_n | e^{1/15}  \le \frac 53 |\gamma_n|  e^{1/15} \ .
 $$
 The claimed analyticity of the map $u \mapsto  \big(\sqrt{\mu_n(u)} \big)_{n \ge 1}$ then follows from \cite[Theorem A.3]{KP}.
 \end{proof}  
  \begin{remark}\label{remark to proposition mu_n}
 Going through the proof of Proposition \ref{proposition mu_n} one verifies that for any $0 < s < 1/2,$
 there exists a constant $C >0$ so that for any $u \in B^{-s}_{c,0}(r_{s, **})$,
$$
| \mu_n(u) - 1 | \le C \Big( \frac{|\gamma_n(u)|}{n} + \frac{|\gamma_n(u)|}{n^{1-2s}} \ell^1_n \Big) \, , \qquad \forall n \ge 1\, ,
$$
where the constant $C > 0$ only depends on $s$.
Indeed, one verifies that for any $0 < s < 1/2,$ there exists a constant $\tilde C$ so that the term $I$ in the proof of Proposition \ref{proposition mu_n} can be bounded by 
 $\tilde C \frac{|\gamma_n(u)|}{n}$, whereas
 $$
 | II | \le \tilde C \frac{|\gamma_n|}{n^{1-2s}}  A_n \, , \qquad A_n := n^{1-2s}\sum_{p \ne n}\frac{ | \gamma_p |}{ (p-n)^2}  \, ,
 $$
 and $\sum_{n \ge 1}  A_n$ is estimated by splitting it  into the following three parts,
 $$
  \sum_{n \ge 1}  n^{1-2s}\sum_{p > n}\frac{ | \gamma_p |}{ (p-n)^2}  \le \sum_{p \ge 1} p^{1-2s}| \gamma_p | \sum_{n \ne p} \frac{1}{ (p-n)^2} < \infty \,  ,
 $$
 $$
  \sum_{n \ge 1}  n^{1-2s}\sum_{n/2 \le p < n}\frac{ | \gamma_p |}{ (p-n)^2}  \lesssim \sum_{p \ge 1}  p^{1-2s}| \gamma_p |  \sum_{n \ge 1} \frac{1}{ n^2} <\infty \, ,
 $$
 $$
  \sum_{n \ge 1}  n^{1-2s}\sum_{p < n/2}\frac{ | \gamma_p |}{ (p-n)^2}  \lesssim \sum_{p \ge 1} | \gamma_p | \sum_{n \ne p} \frac{n^{1-2s}}{ n^2} <\infty \, . \qquad
 $$
 Combining these estimates yields the claimed one.
\end{remark}

 %%%%%%%%%%%%%%%%%%%%%%%%%%%%%%%%%%%%%%%%%%%%%%%%%%%%%
 %%%%%%%%%%%%%%%%%%%%%%%%%%%%%%%%%%%%%%%%%%%%%%%%%%%%%
 
 \section{Proof of Theorem \ref{main theorem}($i$)}\label{Lax operator.II}
 In this section we analyze the spectrum $\text{spec}(L_u)$ of the Lax operator $L_{u}$
 for $u \in H^{-s}_{c, 0}$ with imaginary part $\Im u$ which is small in $H^{-s}_{c, 0}$.  
 Whereas the part of $\text{spec}(L_u)$  in any half plane
 $\{ \Re \lambda \le c\}$ can be analyzed by standard perturbation theory,  the part in the half space
 $\{ \Re \lambda \ge c \}$ relies on the approximation of $u$ by finite gap potentials. 
 
We first need to make some preliminary considerations and introduce some additional notation.
 Recall that for any $s \in \R,$ $h^s \equiv h^s(\Z, \C)$ denotes the weighted $\ell^2-$sequence space
 of sequences $z = (z(k))_{k \in \Z}$ with values in $\C$. It is convenient to extend the $n-$shift operator
 $\mathcal S_n$, introduced in \eqref{n-shift}, to all of $h^s,$
 $$
 \mathcal S_n : h^s \to h^s, \, \qquad (\mathcal S_n z)(k) := z(k+ n)\, , 
 $$
 and correspondingly, to define the shifted norm $\|f \|_{s; n}$ for any $f \in H^s_c$ by
 $$
 \|f \|_{s; n} := \| e^{-inx} f\|_s =  \| \mathcal S_n \widehat f \|_s \, , \qquad   \widehat f := (\widehat f(k))_{k \in \Z} \, .
 $$
 It is straightforward to verify that for any two sequences $z^{(j)} = (z^{(j)}(k))_{k \in \Z}$, $j=1,2$,
 for which the convolution $z^{(1)} \ast z^{(2)}$ is well defined, one has for any $n \ge 1,$
 $$
  \mathcal S_n(z^{(1)} \ast z^{(2)}) =  (\mathcal S_n z^{(1)}) \ast z^{(2)} \, .
 $$
% From Lemma \ref{multiplication u f} one then infers 
 We claim that for any $0 \le s < 1/2$,
 \begin{equation}\label{product estimate}
 \| fg \|_{-s;n} \le C_s \|f\|_{-s; n} \|g \|_{1-s}\, , \qquad \forall \,f \in H^{-s}_c , \, g \in H^{1-s}_c \, ,
 \end{equation}
  where $C_s> 0$ is the constant of Lemma \ref{multiplication u f}. Indeed, 
% $$
% \mathcal S_n \widehat{ f g } = \mathcal S_n (\widehat f \ast \widehat g) = (\mathcal S_n \widehat f) \ast \widehat g 
% =\mathcal F \big(  \mathcal F^{-1} (\mathcal S_n \widehat f) \, g \big) \, ,
% $$ 
%where $\mathcal F$ denotes the Fourier transform and $\mathcal F^{-1}$ its inverse, one has by Lemma \ref{multiplication u f},
 $$
  \| fg \|_{-s;n}  = \| (e^{-inx}f) g \|_{-s} \le C_s \|e^{-inx}f\|_{-s} \|g\|_{1-s} = C_s \|f \|_{-s;n} \|g\|_{1-s} \, .
 $$
We will need the following 
\begin{lemma}\label{estimate for N-gap potentials}
For any $0 \le s < 1/2$, $f, v \in H^{-s}_c$, $g \in H^{1-s}_c$, $N \ge 1$, $n \ge N$,
and $0 < \rho \le 1/3$,
the function 
$$
I(\lambda) (x) := \sum_{k \ge N} \frac{\langle f \, | \, g e^{ikx} \rangle}{k - \lambda} \Pi(v g e^{ikx})\, , \qquad \lambda \in {\rm Vert}_n(\rho) \, ,
$$ 
satisfies the following estimate
$$
\| I(\lambda) \|_{-s; n} \le   \frac{ C_s^3}{\rho}  \|v\|_{-s}  \| g\|_{1-s}^2  \|  f\|_{-s; n} \, ,    \qquad \, \forall \, \lambda \in {\rm Vert}_n(\rho), \, \forall \, n \ge N \, ,
$$
where $C_s \ge 1$ is the constant of Lemma  \ref{multiplication u f}.
\end{lemma}
 \begin{proof} 
 By Lemma \ref{multiplication u f}, $f\bar g$ and $v g$ are in $H^{-s}_c$.
 Since$\widehat{f\bar g} (k' + n) = \mathcal S_n (\widehat{f\bar g}) (k')$ one has by the definition of the shifted norm
 that for any $\lambda \in {\rm Vert}_n(\rho)$
 $$
 \| I(\lambda) \|_{-s; n}^2 = \sum_{\ell' \ge -n} \frac{1}{\langle \ell' \rangle^{2s}}
 \big( \sum_{k' \ge N - n} | \widehat{v g}(\ell' - k')| \frac{|\mathcal S_n(\widehat{f\bar g}) (k' )|}{|k' - (\lambda - n)|}    \big)^2 \, .
 $$
 Applying Lemma \ref{multiplication u f} and Lemma \ref{estimate of lambda in Vert}
 %(cf. also the proof of Lemma \ref{estimate of T_u(D- lambda)^-1}) 
 one concludes that
 $$
 \| I(\lambda) \|_{-s; n}  \le  \frac{C_s}{\rho} \|v g\|_{-s} \|\mathcal S_n(\widehat{f \bar g})\|_{-s} \, , \qquad \forall \, \lambda \in {\rm Vert}_n(\rho) \, .
 $$
 Using Lemma  \ref{multiplication u f} once more it then follows from \eqref{product estimate} that
 $$
  \| I(\lambda) \|_{-s; n}  \le  \frac{C_s^3}{\rho}  \|v \|_{-s} \|g\|_{1-s}^2 \|f\|_{-s; n} \, .
 $$
 \end{proof}
 We will apply Lemma \ref{estimate for N-gap potentials} with $g$ being of the form $g(x)= e^{i h(x)}$ where $h$ is a real valued function in $H^{1-s}_{r,0}$.
 By Moser's composition estimates one has
 \begin{equation}\label{Moser}
\|  e^{i h(\cdot)} \|_{1-s} \le C_{s,1}(1 + \|h\|_{1-s})
 \end{equation}
 where $C_{s,1} \ge 1$ is a constant, depending only on $0 \le s < 1/2$.
 In order to show that for any $\lambda \in {\rm Vert}_n(\rho)$, $L_u - \lambda : H^{1-s}_+ \to H^{-s}_+$ is invertible for $u \in H^{-s}_{c,0}$ near 
 the real valued potential $w \in H^{-s}_{r,0}$, we argue as follows: 
 From \cite[Theorem 1.1]{GK} it follows that for any $0 \le s < 1/2$, the set $\cup_{N \ge 1} \mathcal U_N$ is dense in $H^{-s}_{r,0}$, where according to \cite[Section 7]{GK}
 for any $N \ge 1,$ $\mathcal U_N$ denotes the set of potentials $w \in H^0_{r,0}$  so that
 $$
 \gamma_N(w) > 0 \, , \qquad \gamma_n(w) = 0 \, , \quad \forall \, n \ge N+1  \, .
 $$
Furthermore, for any $N \ge 1$,
 $\mathcal U_N \subset \cap_{m \ge 1} H^m_{r,0}$ and for any $w \in \mathcal U_N$  
 $$
 \lambda_n(w) = n\, , \quad 
 f_n(\cdot, w) = e^{i \partial_x^{-1}w} e^{in x}\, , \qquad \forall \, n \ge N \, ,
 $$
 where we recall that for any $k \ge 0$, $f_k(\cdot, w)$ denotes the eigenfunction, corresponding to the eigenvalue
 $\lambda_k(w)$ of $L_{w}$, canonically normalized as in \cite{GK}.  We remark that $\partial_x^{-1}w(x)$ is in $H^{1-s}_{r,0}$.
 For any $M > 0$, $0 \le s < 1/2$, and $0 < \rho \le 1/3$, one has
 \begin{equation}\label{C(M, s, rho)}
 C_{M, s,\rho}:= \frac{1}{\rho} C_s^3C^2_{s,1}(2 + M)^2 > 12 \, .
 \end{equation}
 Given any  $w \in H^{-s}_{r,0}$  with $w \ne 0$ and  $\|w\|_{-s} \le M$, it follows from \cite[Theorem 1.1]{GK} ($s=0$),
 \cite[Theorem 6]{GKT} ($ -1/2 < s < 0$)
 that there exist $N \ge 1$  so that the potential $w_0 \in \mathcal U_N$, characterized by 
 $\Phi_n(w_0) = \Phi_n(w)$, $1 \le  n \le N$, and $\Phi_n(w_0) = 0$, $n > N$, satisfies $ \|w -w_0 \|_{-s} \le  1/ 8 C_{M, s, \rho}$.
 Note that  $\gamma_n(w_0) = 0$ for any $n > N$ 
 and that for $w_0$ to be in $ \mathcal U_N$, $N$ has to be chosen in such a way that $\gamma_N(w) > 0$. This is possible
 since by assumption $w \ne 0$.
  If $w$ is itself a finite gap potential, we choose $w_0$ to be $w$ itself and $N \ge 1$ accordingly.
 For later reference we record the properties of $w_0 \in \mathcal U_N$ as follows
 \begin{equation}\label{estimate w_0}
 \|w -w_0 \|_{-s} \le  1/ 8 C_{M, s, \rho} , \qquad \Phi_n(w_0) = \Phi_n(w), \ 1 \le  n \le N.
 \end{equation}
\begin{lemma}\label{T u (D - lambda)-1 general case}
For any $M > 0$, $0 \le s < 1/2$, $0 < \rho \le 1/3$, $w \in H^{-s}_{r,0}$ with $\|w\|_{-s} \le M$, and $w_0 \in \mathcal U_N$ with $N \ge 1$,
satisfying \eqref{estimate w_0}, there exists $n_0 > N$ so that for any $v \in H^{-s}_{c, 0}$ and $n \ge n_0$,
$$
\| T_v(L_{w_0} - \lambda)^{-1} \|_{H^{-s; n}_+ \to H^{-s; n}_+}
\le  2 C_{M, s, \rho} \|v \|_{-s} \ , \qquad \forall \lambda \in \text{Vert}_n(\rho) \, ,
$$
where $ H^{-s; n}_+$ is the Hilbert space $H^{-s}_+$,  endowed with the inner product, 
associated with the shifted norm $\| \cdot \|_{-s; n}$ defined in \eqref{shifted norm}.
\end{lemma}
\begin{proof}
Let $f_k \equiv f_k(\cdot, w_0)$, $k \ge 0$, be the eigenfunctions of $L_{w_0}$,  canonically normalized as in \cite{GK}.
Since $w_0 \in \mathcal U_N$, one has for any $n \ge N$, $\lambda_n \equiv \lambda_n(w_0) = n$ and 
$f_n = g_\infty e^{inx}$ where $g_\infty := e^{i \partial_x^{-1}w_0}$.
Furthermore, for any $\lambda \in {\rm Vert}_n(\rho),$  $L_{w_0} -  \lambda : H^{1-s}_+  \to H^{-s}_+$ is invertible. 
For any $v \in H^{-s}_{c,0}$ and $f \in H^{-s}_+$, write $T_v(L_{w_0} - \lambda)^{-1} f = I(\lambda) + II(\lambda)$ where
$$
I(\lambda) := \sum_{k \ge N} \frac{\langle f \, | \, f_k \rangle}{\lambda_k - \lambda} T_v(f_k) \, , \qquad
II(\lambda) := \sum_{k =0}^{N-1} \frac{\langle f \, | \, f_k \rangle}{\lambda_k - \lambda} T_v(f_k) \, .
$$
Then
$$
I(\lambda) = \sum_{k \ge N} \frac{\langle f \, | \, g_\infty e^{ikx} \rangle}{k - \lambda} \Pi(vg_\infty e^{ikx})
$$
and since by \eqref{Moser} and  \eqref{C(M, s, rho)} -- \eqref{estimate w_0},
$$ 
\|g_\infty\|_{1-s} \le C_{s,1}(1 + \|\partial_x^{-1}w_0 \|_{1-s}) \le 
 C_{s,1}(1 + \|w_0 \|_{-s})  \le  C_{s,1}(2 + M)\, ,
$$
it then follows from Lemma \ref{estimate for N-gap potentials} that for any $n \ge N,$
$$
\| I(\lambda) \|_{-s; n}   \le \frac{C_s^3}{\rho}  \|v \|_{-s} \|g_\infty\|_{1-s}^2 \|f\|_{-s; n} 
\le C_{M, s, \rho}  \|v \|_{-s}  \|f\|_{-s; n} \, .
$$
To obtain the required estimate of $II(\lambda)$ we will choose $n_0 \ge N$ sufficiently large. Note that
$$
\| II(\lambda) \|_{-s; n} \le  \sum_{k =0}^{N-1} \frac{| \langle f \, | \, f_k \rangle |}{ |\lambda_k - \lambda|} \| T_v(f_k)\|_{{-s; n}}.
$$
%Since $f_n$, $n \ge 0$, is a basis of $H^{-s}_+$ there exists $C_{s,2} \equiv C_{s, 2}(w_0)  \ge 1$ so that, 
Taking into account Lemma \ref{equivalence of norms}, one has
$$
\sum_{k =0}^{N-1} | \langle f \, | \, f_k \rangle | \le  
\sum_{k =0}^{N-1}  \|f \|_{-s} \| f_k \|_{s} \le 
N \big( \max_{0 \le k \le N -1 } \|f_k\|_{s}  \big) \|f\|_{-s; n}  2^s \langle n \rangle^s 
$$
and by Lemma \ref{multiplication u f} (with $\sigma := (s + 1/2)/2$), for any $0 \le k \le N - 1$,
$$
\| T_v(f_k) \|_{-s; n} \le C_{s}  \|v \|_{-s} \|f_k \|_{1-\sigma; n}
\le C_{s}  \| v \|_{-s} \max_{0 \le k \le N -1 } \|f_k\|_{1-\sigma; n} \, .
$$
Furthermore,  $\| f_k \|_{1-\sigma; n} \le 2^s \langle n \rangle^{1-\sigma} \| f_k \|_{1-\sigma}$ (cf.  Lemma \ref{equivalence of norms}) 
and for any $n \ge N,$ $\lambda \in {\rm Vert}_n(\rho)$ (cf. by Lemma \ref{estimate of lambda in Vert}),
$$
|\lambda - \lambda_k| \ge | n-k | - \frac 12 \ge n  - N + \frac 12 , \qquad \forall \, 0 \le k \le N - 1 \, ,
$$
where we used that $\lambda_k(w_0) = k - \sum_{\ell > k} \gamma_\ell(w_0) \le k$.
Combining the above estimates yields
$$
\| II(\lambda) \|_{-s; n} \le \frac{n}{ n - N + 1/2} N C_{s} \big( \max_{0 \le k \le N} \|f_k\|_{1-\sigma} \big)^2
\frac{4^s}{\langle n \rangle^{\sigma - s}} \| v \|_{-s} \|f \|_{-s; n} \, .
$$
By choosing $n_0 > N$  sufficiently large it follows that 
$$
\| II(\lambda) \|_{-s; n} \le  C_{M, s, \rho} \|v \|_{-s}  \|f \|_{-s; n} \, ,
$$
yielding the claimed estimate when combined with the one obtained for $\| I(\lambda) \|_{-s; n}$. 
\end{proof}
 Lemma \ref{T u (D - lambda)-1 general case} yields the following
 \begin{corollary}\label{invertibility general case}
 Let $M > 0$, $0 \le s < 1/2$, $0 < \rho \le 1/3$, $w \in H^{-s}_{r,0}$ with $\|w\|_{-s} \le M$, $w_0 \in \mathcal U_N$ with $N \ge 1$,
satisfying \eqref{estimate w_0}, and $n_0 > N$ so that  Lemma \ref{T u (D - lambda)-1 general case} holds.
Then for any $n \ge n_0$ and $\lambda \in \text{Vert}_n(\rho)$, 
$$
\| T_v(L_{w_0} - \lambda)^{-1} \|_{H^{-s; n}_+ \to H^{-s; n}_+} \le  1/2   \, , \qquad \forall \, w_0 + v \in B^{-s}_{c,0}(w_0, 1/4C_{M, s, \rho})\, ,
$$
 where $C_{M, s, \rho}$ is given by \eqref{C(M, s, rho)}. 
 \end{corollary}
 \begin{remark}
 Note that by \eqref{estimate w_0}, 
 $B^{-s}_{c,0}(w, 1/8C_{M, s, \rho}) \subset B^{-s}_{c,0}(w_0, 1/4C_{M, s, \rho})$.
 \end{remark}
 
 As a last ingredient of the proof of the Counting Lemma for potentials in $H^{-s}_{c,0}$ near $H^{-s}_{r,0}$  
 we need an extension of  \cite[Lemma 5.1]{GK} to $H^{-s}_{r,0}$. To state it denote for any $w \in H^{-s}_{r,0}$,
 $n \ge 1$, and $K \ge 1$, by ${\rm Box}_{K, n}(w)$ the closed rectangle in $\C$, given by 
 the set of complex numbers, satisfying
 $$
 - K + \lambda_0(w) \le \Re(\lambda) \le \lambda_n(w) + 1/2\, , \qquad |\Im(\lambda) | \le K. 
 $$
Let $r_0:= 0$, $\tau_0 := \lambda_0(w)$ and for any $k \ge 1,$  
 $$
 r_k := \gamma_k(w)/2 \, , \quad  \tau_k := \lambda_k(w) - \gamma_{k}(w)/2\, .
$$
Furthermore, recall that for any given $\tau \in \C$ and $r > 0$, $D_\tau(r)$
denotes the open disc $ \{ \lambda \in \C \, : \, |\lambda - \tau| < r \}$.
\begin{lemma}\label{lambda analytic}
Let $0 \le s < 1/2$, $0 < \rho \le 1/3,$ $K \ge 1,$ and $n \ge 0.$
 For any $w \in H^{-s}_{r, 0}$, there exists
$\e \equiv \e_{s, \rho, K, n} > 0$ so that for any $u \in B^{-s}_{c, 0}(w, \e)$ and $0 \le k \le n,$
$\# \big( spec(L_{u}) \cap D_{\tau_k}(r_k + \rho) \big) = 1$
and
$$ spec(L_{u}) \cap  {\rm Box}_{K,n}(w)
\subset \bigcup_{0 \le k \le n} D_{\tau_k}(r_k + \rho) .
$$
The unique eigenvalue of $L_{u}$ in $D_{\tau_k}(r_k + \rho)$ is 
denoted by $\lambda_k(u)$.
\end{lemma}
\begin{proof}
The proof of Lemma 5.1 in \cite{GK} can easily be adapted to the setup at hand
and hence is omitted.
\end{proof}

Now we have all the ingredients to prove the Counting Lemma. To state it, we define
 for any $\tau \in \C$, $\nu, \nu' > 0$, and $r > 0$,
 $$ 
 {\rm Vert}_\tau(r; \nu, \nu'):= \{ \lambda \in \C \, : \, |\lambda - \tau| \ge r; \,\,  
 \tau - \nu \le \Re \lambda  \le \tau + \nu' \}.
 $$
 It is convenient to also define
 $$
  {\rm Vert}_\tau(r;  \infty, \nu'):= \{ \lambda \in \C \, : \, |\lambda - \tau| \ge r; \,\,  
  \Re \lambda  \le \tau + \nu' \}\, .
  $$
 Lemma \ref{resolvent set 1}, Corollary \ref{invertibility general case}, and Lemma \ref{lambda analytic}
 yield a proof of Theorem \ref{main theorem}(i), which we reformulate slightly as follows.
 \begin{theorem}\label{general counting lemma}(Counting Lemma)
 Let $0 \le s < 1/2$, $0 < \rho \le 1/3$, and $M > 0$. For any $w \in H^{-s}_{r,0}$ with $\|w\|_{-s} \le M$, 
 choose $w_0 \in \mathcal U_N$ with $N\ge 1$, satisfying \eqref{estimate w_0}, and $n_0 > N$ 
 so that  Lemma \ref{T u (D - lambda)-1 general case} holds. Then there exists a neighborhood
 $U^{-s}_w$ of $w$ in $H^{-s}_{c,0}$, contained in $B^{-s}_{c,0}(w, 1/8C_{M, s, \rho})$, with the following properties:
 for any $u \in U^{-s}_w$, 
 the spectrum ${\rm spec}(L_{u})$ of the operator $L_{u}$ is discrete, consists of simple eigenvalues only, and
  for any $ n \ge 0 $, one has
 $$
  \#\big({\rm spec}(L_{u}) \cap D_{\tau_n}(r_n+ \rho) \big) = 1\,
  $$
  and
 $$
   {\rm spec}(L_{u}) \cap {\rm Vert}_{\tau_n}(r_n + \rho; \nu_{n}, \nu_{n+1}) = \emptyset . 
$$
 Here $\tau_0 := \lambda_0(w)$, $\nu_{0} =  \infty$, $r_0:= 0$  and for any $ 1 \le n <  n_0 $,
$$
 r_n := \gamma_n(w)/2, \quad  \tau_n := \lambda_n(w) - \gamma_{n}(w)/2 \, , \quad \nu_n := ( \tau_{n} - \tau_{n-1})/2\, ,
$$
 whereas
$$
 r_{n_0} := 0\, , \quad \tau_{n_0 } := n_0 \, , \quad  \nu_{n_0} := \frac 12 \big( 1 + \gamma_{n_0 - 1}(w)/2 + \sum_{k \ge n_0 } \gamma_k(w) \big) \, ,
$$  
and for any $n \ge n_0 + 1$, 
$$
  r_n:= 0\, , \qquad  \tau_n := n\, , \qquad \nu_n := 1/2 \, . \qquad \qquad
$$
In particular, $D_{\tau_n}(r_n + \rho)=   D_n(\rho)$ for any $n \ge n_0 $ and
${\rm Vert}_{n_0}(\rho) \subset   {\rm Vert}_{\tau_n}(r_{n_0} + \rho; \nu_{n_0}, \nu_{n_0+1}) $
 whereas
$$
  {\rm Vert}_{\tau_n}(r_n + \rho; \nu_{n}, \nu_{n+1}) = {\rm Vert}_n(\rho) \, , \qquad \forall \, n \ge n_0 + 1\, .
$$
  \end{theorem}
  
\begin{figure}
\begin{tikzpicture}[scale=0.75]
\fill [color=gray!30] (-8,-8) -- (-8,8) -- (4,8) -- (4,-8) -- cycle ;
\draw (4,0) node {$\bullet$} ;
\draw (4,0) node[below right] {$\lambda_n(w)+1/2$} ;
\draw (4,8) node {$\bullet$} ;
\draw (4,8) node[above right] {$(\lambda_n(w)+1/2, K)$} ;
\draw (0,8) node {$\bullet$} ;
\draw (0,8) node[above right] {$K$} ;
\draw (-8,0) node {$\bullet$} ;
\draw (-8,0) node[below left] {$-K + \lambda_0(w)$} ;
\filldraw [ fill=gray!00, draw=black] (-7,0) circle (0.5) ;
\filldraw [ fill=gray!00, draw=black] (-5,0) circle (1) ;
\filldraw [ fill=gray!00, draw=black] (-2,0) circle (0.75) ;
\filldraw [ fill=gray!00, draw=black] (2,0) circle (0.5) ;
\draw  [->]  (-8,0) -- (7,0) ;
\draw  [->]   (0,-9) -- (0,9) ;
\draw (7,0) node[above] {$\mu =\Re \lambda$} ;
\draw (0,9) node[above right] {$\nu =\Im \lambda$} ;
\draw (-7,0) node {$\bullet$} ;
\draw (-7,0) node[below]  {$\tau_0$} ;
\draw (-5,0) node {$\bullet$} ;
\draw (-5,0) node[below]  {$\tau_1$} ;
\draw (-2,0) node {$\bullet$} ;
\draw (-2,0) node[below]  {$\tau_2$} ;
\draw (2,0) node {$\bullet$} ;
\draw (2,0) node[below]  {$\tau_n$} ;
\draw (1.5,-3) node{${\rm Box}_{K,n}(w)$} ;
\end{tikzpicture}
\caption{ $L_{u}$ has one eigenvalue in each white disc (schematic).}
\end{figure}
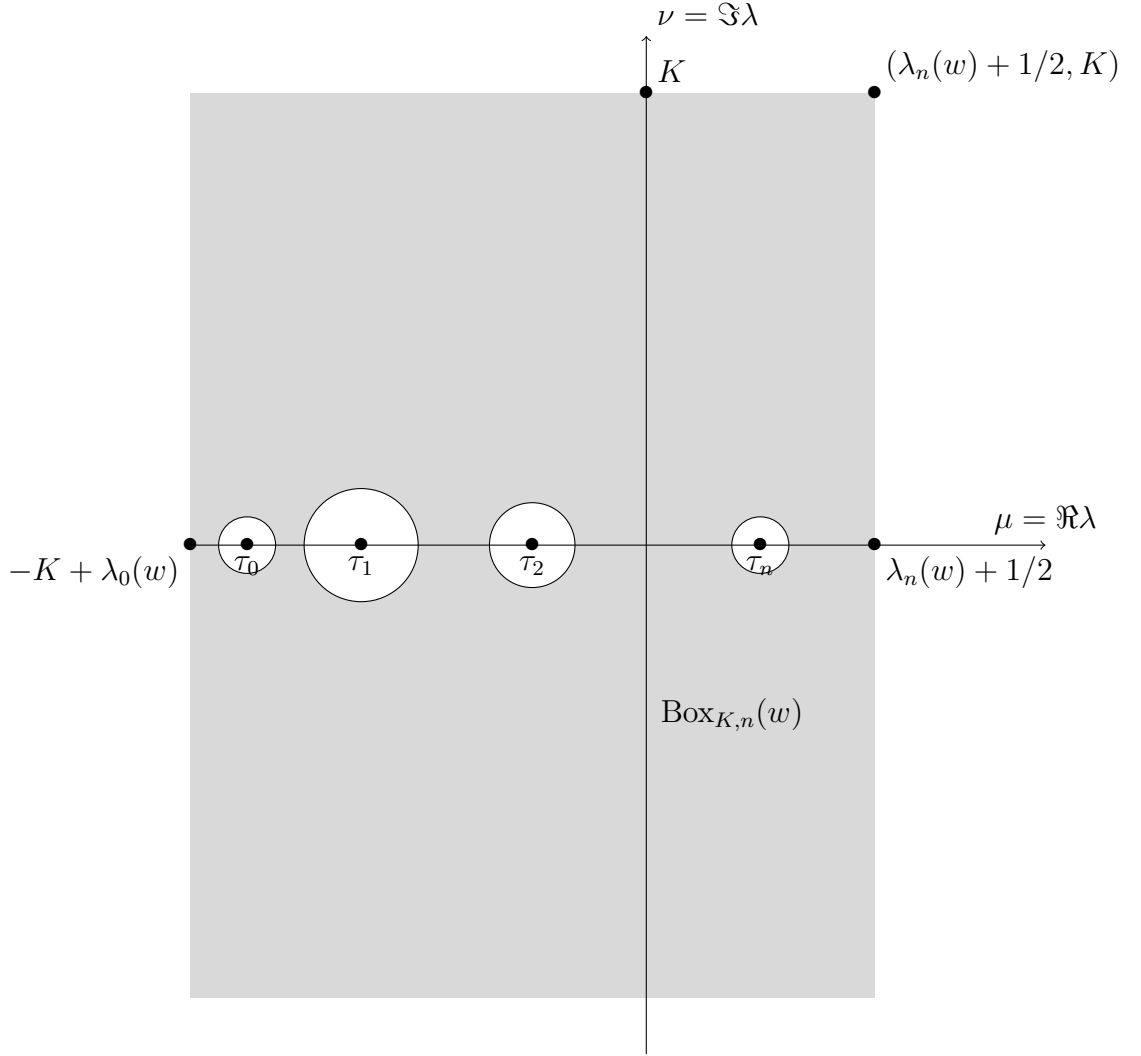 

%%%%%%%%%%%%%%%%%%%%%%%%%%%%%%%%%%%%%%%%%%%%%%%
%%%%%%%%%%%%%%%%%%%%%%%%%%%%%%%%%%%%%%%%%%%%%%%

 \section{Quasi-moment map}\label{Quasi-moment map Im u small}

In this section, we study the quasi-moment map $F(u) = (F_n(u))_{n \ge 1}$.
Recall that for any $u \in H^{-s}_{r,0}$ and $n \ge 1$, the component $F_n(u)$ of $F(u)$ 
is given by  $ - |\langle 1| f_n(\cdot, u)\rangle |^2 = - \gamma_n \kappa_n$ and hence $F(u) \in  \ell^{1, 2-2s}(\N, \R)$ 
(cf. Remark \ref{identity2 for F_n for u real}) and 
that $F$ admits an analytic extension to a neighborhood of  $0 \in H^{-s}_{c,0}$ for any $0 \le s < 1/2$ (cf.  Proposition \ref{F analytic near 0}). 
Our aim is to show that for any given $w \in H^{-s}_{r,0}$ with $0 \le s < 1/2$, $F: H^{-s}_{r,0} \to \ell^{1, 2-2s}(\N, \R)$ 
extends to an analytic map on a neighborhood of $w$ in $H^{-s}_{c,0}$. The proof relies on the approximation of $w$
by finite gap potentials, which allows us to use similar methods of proofs as in the case where $\| u \|_{-s}$ is small.
 So we will be as brief as possible and only discuss in detail the additional arguments involved.

 First note that it follows by Theorem \ref{general counting lemma} that for any given $w \in H^{-s}_{r,0}$ with $0 \le s < 1/2$ and $n \ge 1,$ 
$F_n$ admits an analytic extension to $U^{-s}_w$, 
 \begin{equation}\label{def F_n general}
 F_n : U^{-s}_w \to \mathbb C\ , \ u \mapsto \frac{1}{2\pi i}\int_{\partial D_{\tau_n}(r_n + 1/3)} \langle (L_u - \lambda)^{-1} 1 | 1 \rangle d \lambda \, ,
 \end{equation}
 where the circle $\partial D_{\tau_n}(r_n + 1/3)$ is counterclockwise oriented. 
 Here $U^{-s}_w,$ $\tau_n$, $r_n$, and $D_{\tau_n}(r_n + 1/3)$ are given by  Theorem \ref{general counting lemma}
 with $\rho = 1/4$ and $M = \|w\|_{-s}$.
 
 \begin{proposition}\label{F analytic general}
Let $w \in H^{-s}_{r,0}$ with $0 \le s < 1/2$. Then for any $u \in U^{-s}_w$, $F(u)  \in \ell^{1, 2-2s}_+$
and $F :  U^{-s}_w \to \ell^{1, 2-2s}_+$ is analytic. In particular, by shrinking $U^{-s}_w$ if needed, 
$F(U^{-s}_w)$ is bounded in $\ell^{1, 2-2s}_+$.
%Furthermore
%$$
%\sum_{n=1}^\infty  \sum_{m=0}^\infty \big| n^{2-2s} \frac{1}{2\pi i} \int_{\partial D_n(1/3)}  \langle (D- \lambda)^{-1}(T_u(D - \lambda)^{-1})^m 1 | 1 \rangle d \lambda \big|
% $$
%converges  uniformly with respect to $u$.
\end{proposition}

Before proving Proposition \ref{F analytic general}, we make some preliminary considerations.
 Let $0 \le s < 1/2$, $ \rho = 1/4$, and $M > 0$ be given. For any $w \in H^{-s}_{r,0}$ with $\|w\|_{-s} \le M$, 
 choose $w_0 \in \mathcal U_N$ with $N\ge 1$, satisfying \eqref{estimate w_0}, and $n_0 > N$ 
 so that  Lemma \ref{T u (D - lambda)-1 general case} holds. Then there exists a neighborhood
 $U^{-s}_w$ of $w$ in $H^{-s}_{c,0}$, contained in $B^{-s}_{c,0}(w, 1/8C_{M, s, 1/4})$, so that 
 Theorem \ref{general counting lemma} holds. 
Together with Corollary \ref{invertibility general case}  it then follows that
 for any $u = w_0 + v \in U^{-s}_w$ and $n \ge n_0$, 
 \begin{equation}\label{estimate 1}
\| T_v(L_{w_0} - \lambda)^{-1} \|_{H^{-s; n}_+ \to H^{-s; n}_+} \le 1/2\, , 
\quad \forall \lambda \in  {\rm Vert}_{n}(1/4).
\end{equation}
It implies that for any $\lambda \in \bigcup_{n \ge n_0} {\rm Vert}_{n}(1/4)$, 
the inverse of the operator $L_u - \lambda : H^{1-s}_+ \to H^{-s}_+$
can be written as 
$$
(L_{w_0} - \lambda)^{-1}( \text{Id} - T_v(L_{w_0} - \lambda)^{-1})^{-1} \, .
$$
 Denote by $f_n \equiv f_n(\cdot, w_0)$, $n \ge 0$, the eigenfunctions of $L_{w_0}$, corresponding to the eigenvalues $\lambda_n \equiv \lambda_n(w_0)$, $n \ge 0$,
 normalized as in \cite{GK}. Then $\lambda_n = n$ and $f_n= g_\infty e^{inx}$ for any $n \ge N$ where $g_\infty = e^{i \partial_x^{-1}w_0}$. Furthermore, let
 $$
 E_1:= \text{span}\{ f_\alpha \,  : \, 0 \le \alpha \le N \} \, , \qquad E_2 \equiv E_2^{-s}:=  \text{span}\{ f_k \, :  \, k > N \} \, .
 $$
 Here $E_1$ is a subspace of $\cap_{m \ge 0}H^{m}_+$ of dimension $N+1$, spanned by $f_\alpha,$ $0 \le \alpha \le N$, 
and endowed with the $L^2-$norm, 
 whereas $E_2$ is the closed subspace of $H^{-s}_+$, spanned by $f_k$, $k > N$, 
 endowed with the norm $\| \cdot \|_{-s}$.
 We denote by $E_j^{-s; n}$ the space $E_j,$ when endowed with the norm $\| \cdot \|_{-s; n}$. Since for any $n > N$,  $\gamma_n(w_0) = 0$ and hence 
 $\langle 1 | f_n \rangle = 0$ it follows that the constant function $1$ is in $E_1$, 
 \begin{equation}\label{expansion of 1}
 1 = \sum_{\alpha \le N} \langle 1 | f_\alpha \rangle f_\alpha\, .
 \end{equation}
 It implies that $\langle (L_u - \lambda)^{-1} 1 | 1 \rangle$ is given by
 $$
  \sum_{\alpha, \beta \le N} \langle 1 | f_\alpha \rangle  \langle f_\beta | 1 \rangle  
 \langle (L_{w_0} - \lambda)^{-1}( \text{Id} - T_v(L_{w_0} - \lambda)^{-1})^{-1} f_\alpha | f_\beta \rangle
 $$
 which equals 
 \begin{equation}\label{formula 0}
   \sum_{\alpha, \beta \le N} \frac{ \langle 1 | f_\alpha \rangle  \langle f_\beta | 1 \rangle  }{\lambda_\beta - \lambda} 
   \langle ( \text{Id} - T_v(L_{w_0} - \lambda)^{-1})^{-1} f_\alpha | f_\beta \rangle \, .
 \end{equation}
 We now introduce for $1 \le j \le 2$ the canonical $L^2-$ projection $\Pi_j$ and 
 the canonical inclusion $\iota_j$, associated to the subspace $E_j$,
 $$
 \Pi_j : H^{-s}_+ \to E_j \, , \qquad \iota_j : E_j \to H^{-s}_+ \, ,
 $$ 
 and define for any $\lambda \in \text{Vert}_n(1/4)$, $n \ge n_0$, 
 $$
A_{ij} \equiv  A(\lambda)_{ij} := \Pi_i  T_v(L_{w_0} - \lambda)^{-1}  \iota_j  : E_j \to E_i\, ,  \quad  1 \le i, j \le 2 \, .
 $$
 A straightforward formal computation shows that
 $$
 \Pi_1  \big( \text{Id} - T_v(L_{w_0} - \lambda)^{-1} \big)^{-1} \iota_1 =  (\text{Id}- A_{11})^{-1} (\text{Id} - B)^{-1}\, , 
 $$
 where $B \equiv B(\lambda) : E_1 \to E_1$ is given by
 \begin{equation}\label{38i}
 B :=  B_1 (\text{Id}  - A_{11})^{-1}\, , \qquad  B_1 := A_{12} (\text{Id}  - A_{22})^{-1} A_{21}   \, .
 \end{equation}
 These computations are justified by the estimates below for $n$ sufficiently large.
 To estimate the quasi-moment map $F(u)$, we expand
 $$
 (\text{Id}- A_{11})^{-1} = \text{Id} + (\text{Id}- A_{11})^{-1}A_{11} =   \text{Id} + A_{11} +  (\text{Id}- A_{11})^{-1}A_{11}^2
 $$
 and
 $$
 (\text{Id} - B)^{-1} =  \text{Id} + B  + \cdots  + B^{K-1} +  (\text{Id} - B)^{-1} B^K \, ,
 $$
yielding
\begin{equation}\label{formula 1}
 \Pi_1  ( \text{Id} - T_v(L_{w_0} - \lambda)^{-1})^{-1} \iota_1 =  (\text{Id}- A_{11})^{-1} + A\, ,
\end{equation}
where 
$$
A \equiv A(\lambda) := (\text{Id}- A_{11})^{-1}(\text{Id} - B)^{-1} B: E_1 \to E_1 \, .
$$ 
The operator $A$ is expanded as
\begin{align}
 A & = (\text{Id} - B)^{-1}B + (\text{Id}- A_{11})^{-1} A_{11}B + (\text{Id}- A_{11})^{-1} A_{11} (\text{Id} - B)^{-1}B^2  \nonumber\\
& = (\text{Id} - B)^{-1}B + A_{11}B + (\text{Id}- A_{11})^{-1} A_{11}^2 B + (\text{Id}- A_{11})^{-1} A_{11} (\text{Id} - B)^{-1}B^2 \, , \nonumber
\end{align}
yielding
\begin{equation}\label{39i}
A =  (\text{Id} - B)^{-1}B + A_{11}B_1 + R_1 \, ,
\end{equation}
where the remainder $R_1 \equiv R_1(\lambda):  E_1 \to E_1$ is given by
\begin{equation}\label{39ii}
R_1 := A_{11}B A_{11} +  (\text{Id}- A_{11})^{-1}A_{11}^2B +  (\text{Id}- A_{11})^{-1}A_{11} (\text{Id} - B)^{-1} B^2 \, .
\end{equation}
We then expand $(\text{Id} - B)^{-1}B$ in the above expression for $A$ further, 
\begin{equation}\label{39iii}
(\text{Id} - B)^{-1}B = \sum_{j=1}^{K-1} B^j +  (\text{Id} - B)^{-1} B^K \, .
\end{equation} 
Note that
\begin{equation}\label{39iv}
B = B_1 + BA_{11} = B_1 + B_1 A_{11} + BA_{11}^2,
\end{equation}
and by induction, for any $2 \le k \le K-1$,
\begin{equation}\label{39v}
B^k =  B_1^k + \sum_{j=0}^{k-1} B_1^jBA_{11}B^{k-j-1}.
\end{equation}
Choosing $K_s:= \min \{ K \in \N : K > (2-2s)/(1-2s) \}$ we then get
by \eqref{38i}--\eqref{39v}
$$
 \Pi_1  ( \text{Id} - T_v(L_{w_0} - \lambda)^{-1})^{-1} \iota_1 =  
 (\text{Id}- A_{11})^{-1} + A_{11}B_1  + B_1 A_{11}  + \sum_{k=1}^{K_s -1} B_1^k + R
$$
where the remainder $R \equiv R(\lambda)$ is given by
$$
R := R_1 + BA_{11}^2 +  \sum_{k=2}^{K_s -1} \sum_{j=0}^{k-1}B_1^jBA_{11}B^{k-j-1} + B^{K_s} (\text{Id} - B)^{-1} \, .
$$
By \eqref{formula 0} and \eqref{formula 1} one has for any $n \ge n_0$
\begin{equation}\label{identity 1 F_n }
 F_n(u) 
 %= \frac{1}{2\pi i}\int_{\partial D_{n}(1/3)} \langle (L_u - \lambda)^{-1} 1 | 1 \rangle d \lambda 
 = \sum_{\alpha, \beta \le N}  \langle 1 | f_\alpha \rangle  \langle f_\beta | 1 \rangle  (  F^0_{n; \alpha, \beta}(u) + F_{n; \alpha, \beta}(u) )
 \end{equation}
 where
 \begin{equation}\label{def F^0_n}
 F^0_{n; \alpha, \beta}  (u) 
 : = \frac{1}{2\pi i}\int_{\partial D_{n}(1/3)}   \frac{1}{\lambda_\beta - \lambda} \langle (\text{Id}- A_{11}(\lambda))^{-1} f_\alpha | f_\beta \rangle d \lambda
 \end{equation}
 and
 \begin{align}\label{identity 2 F_n }
F_{n; \alpha, \beta}  (u) 
 & : = \frac{1}{2\pi i}\int_{\partial D_{n}(1/3)}   \frac{1}{\lambda_\beta - \lambda} \langle A(\lambda) f_\alpha | f_\beta \rangle d \lambda\nonumber \\
& = \frac{1}{2\pi i}\int_{\partial D_{n}(1/3)}   \frac{1}{\lambda_\beta - \lambda} \langle \big(A_{11}B_1  + B_1 A_{11}  + \sum_{k=1}^{K_s -1} B_1^k \big)  f_\alpha \, | \, f_\beta \rangle d \lambda 
\nonumber\\
& \quad+ \frac{1}{2\pi i}\int_{\partial D_{n}(1/3)}   \frac{1}{\lambda_\beta - \lambda} \langle R(\lambda) f_\alpha | f_\beta \rangle d \lambda \, .
\end{align}
To estimate the remainder term, we analyze the operator norms of the various operators involved. 
For any bounded linear operator $G: E_i \to E_j$, we denote by $\|G \|$ the usual operator norm. 
For a bounded linear operator $G: E_i \to E_j$, where $E_i$ and $E_j$ are endowed with the norm $\| \cdot \|_{-s; n}$,
the corresponding operator norm is denoted by $\| G \|_{-s; n}$. 

\medskip
\noindent
{\em Estimate of $\| A_{11}(\lambda) \|.$} For any $f \in E_1$ and $\lambda \in \text{Vert}_n(1/4)\cup D_n(1/4)$, $n \ge n_0$,
one has with $f = \sum_{\alpha \le N} \langle f | f_\alpha \rangle f_\alpha$,
\begin{equation}\label{42i}
A_{11}(\lambda) f   =  \sum_{\alpha \le N} \langle f | f_\alpha \rangle \Pi_1 T_v (L_{w_0} - \lambda)^{-1} f_\alpha 
=  \sum_{\alpha, \beta \le N} \langle f | f_\alpha \rangle \frac{ \langle v f_\alpha | f_\beta \rangle}{\lambda_\alpha - \lambda}  f_\beta \, .
\end{equation}
 Since 
 $$
 \sup_{\lambda \in \text{Vert}_n(\frac{1}{4})\cup D_n(\frac{1}{4})} \frac{1}{|  \lambda -  \lambda_\alpha |}  \le \frac{1}{n - 1/2 - N } \, , \qquad \forall \, n \ge n_0,
 $$
 one has with $\|f\| = \big( \sum_{\alpha \le N} | \langle f | f_\alpha \rangle |^2 \big)^{1/2}$,
 $$
 \begin{aligned}
  \sup_{\text{Vert}_n(\frac{1}{4})\cup D_n(\frac{1}{4})}  \| A_{11}(\lambda) f  \| & \le \frac{1}{n - 1/2 - N} 
  \big( \sum_{\beta \le N} \big| \sum_{\alpha \le N} \langle f | f_\alpha \rangle \langle v f_\alpha | f_\beta \rangle \big|^2 \big)^{1/2}\\
  & \le  \frac{N + 1}{n - 1/2 - N} \max_{\alpha, \beta} \big(\| vf_\alpha \|_{-s}  \|f_\beta\|_s\big) \, \|f\| \, .
  \end{aligned}
 $$
 Using that $\| vf_\alpha \|_{-s} \le C_s \|v\|_{-s} \|f_\alpha \|_{1-s}$ (cf. Lemma \ref{multiplication u f}) one can find 
 a constant $C_{11}^{N,s} > 0$, independent of $n \ge n_0,$ so that for any $n \ge n_0$ and $u = w_0 + v \in U^{-s}_w$, 
 \begin{equation}\label{estimate 2}
  \sup_{\lambda \in  \text{Vert}_n(\frac{1}{4})\cup D_n(\frac{1}{4})}  \| A_{11}(\lambda)  \|  \le C_{11}^{N, s} \,  \frac{1}{n}\, .
 \end{equation}
 As a consequence, there exists $n_1 \ge n_0$ so that 
 \begin{equation}\label{43i}
  \sup_{\lambda \in  \text{Vert}_n(\frac{1}{4})\cup D_n(\frac{1}{4})}  \| A_{11}(\lambda)  \|  \le 1/2\, , \qquad \forall \, n \ge n_1 ,
  \end{equation}
 and hence for any $u = w_0 + v \in U^{-s}_w$,
 \begin{equation}\label{estimate 3}
   \sup_{\lambda \in  \text{Vert}_n(\frac{1}{4})\cup D_n(\frac{1}{4})}  \| (\text{Id} - A_{11}(\lambda))^{-1}  \|  \le 2 \, , \qquad \forall \, n \ge n_1 .
 \end{equation} 
In particular, since by \eqref{42i}, $\lambda \mapsto A_{11}(\lambda)$ is analytic on $D_n(1/3)$ for any $n \ge n_1$,
it follows from \eqref{43i} that
$\lambda \mapsto (\text{Id}- A_{11}(\lambda))^{-1}$ is analytic on $D_n(1/3)$ for any $n \ge n_1$ as well and  
hence by definition \eqref{def F^0_n},
\begin{equation}\label{vanishing F^0_n}
 F^0_{n; \alpha, \beta}  (u) = 0\, , \qquad \forall \, n \ge n_1\, , \  \forall \, u = w_0 + v \in U^{-s}_w\, .
 \end{equation}
 
\noindent
{\em Estimate of $\| A_{21}(\lambda) \|.$} For any $f \in E_1$ and any $\lambda \in \text{Vert}_n(1/4)$, $n \ge n_0$,
one has with $f = \sum_{\alpha \le N} \langle f | f_\alpha \rangle f_\alpha$,
$$
A_{21}(\lambda) f   =  \sum_{\alpha \le N} \langle f | f_\alpha \rangle \Pi_2 T_v (L_{w_0} - \lambda)^{-1} f_\alpha 
= \sum_{\alpha \le N, k > N}  \langle f | f_\alpha \rangle \frac{ \langle v f_\alpha | f_k \rangle}{\lambda_\alpha - \lambda}  f_k \, .
$$ 
Since by \cite[Lemma 6]{GKT}, there exists a constant $C > 0$ so that
$$
 \|g\|_{-s}  \le C \big( \sum_{k \ge 0} \frac{1}{\langle k \rangle ^{2s}} | \langle g | f_k \rangle |^2 \big)^{1/2} \, , \qquad \forall \, g \in H^{-s}_+ \, ,
$$
one has by Cauchy-Schwarz,
$$
\begin{aligned}
  \sup_{\lambda \in \text{Vert}_n(1/4)}  & \| A_{21}(\lambda) f  \|_{-s}   \le 
  \frac{C}{n - 1/2 - N}   \big( \sum_{k >N}  \frac{1}{k  ^{2s}} \big| \sum_{\alpha \le N} \langle f | f_\alpha \rangle \langle v f_\alpha | f_k \rangle \big|^2 \big)^{1/2}\\
 & \le   \frac{C \cdot (N+1) }{n - 1/2 - N} \, \max_{\alpha \le N}  \big( \sum_{k > N}  \frac{1}{k  ^{2s}}   | \langle v f_\alpha | f_k \rangle |^2 \big)^{1/2} \, \|f\| \, .
  \end{aligned}
 $$
 Using that $f_k = g_\infty e^{ikx}$ for any $k \ge N$ one sees that
 $$
  \big(\sum_{k > N}  \frac{1}{k  ^{2s}}   | \langle v f_\alpha | f_k \rangle |^2 \big)^{1/2} \le   \|  v f_\alpha \overline{g}_\infty \|_{-s}\, .
 $$
Recall that 
$g_\infty = e^{i \partial_x^{-1}w_0}$ is $C^\infty-$smooth and that $H^{1-s}_c$ is an algebra, implying that there exists a constant $C' > 0$ 
so that 
$$
 \| h_1 h_2  \|_{1-s} \le C' \|h_1\|_{1-s} \|h_2\|_{1-s} \, , \qquad \forall \, h_1, h_2 \in H^{1-s}_{c}.
 $$
It then follows from Lemma  \ref{multiplication u f}  and \eqref{Moser} that
 $$
  \|  v f_\alpha \overline{g}_\infty \|_{-s} \le  C_s \|v\|_{-s}  \cdot C' \|f_\alpha \|_{1-s} C_{s,1}(1 + \|w_0\|_{-s}) \, .
 $$
 Altogether we have shown that there exists a constant $C_{21}^{N,s} > 0 $ so that
  for any $n \ge n_0$ and $u = w_0 + v \in U^{-s}_w$, 
 \begin{equation}\label{estimate 4}
   \sup_{\lambda \in \text{Vert}_n(1/4)}   \| A_{21}(\lambda)   \|   \le C_{21}^{N,s} \frac{1}{n} \, .
 \end{equation}
 \medskip
\noindent
{\em Estimate of $\| (\text{Id} - A_{22}(\lambda ))^{-1} \|_{-s; n}$.} It follows from \eqref{estimate 1} that for any $n \ge n_0$ and  any $u = w_0 + v \in U^{-s}_w$
\begin{equation}\label{estimate 5a}
 \sup_{\lambda \in \text{Vert}_n(1/4)} \| A_{22}(\lambda)  \|_{-s; n} \le 1/2 \, .
 \end{equation}
 Hence writing $( \text{Id} - A_{22}(\lambda) )^{-1}$ as a Neumann series, one sees that
\begin{equation}\label{estimate 5}
 \sup_{\lambda \in \text{Vert}_n(1/4)} \| ( \text{Id} - A_{22}(\lambda) )^{-1} \|_{-s; n} \le 2\, .
\end{equation}
 \medskip
\noindent
{\em Estimate of $\| A_{12}(\lambda) \|_{-s; n}.$}  
Recall that $  A_{12}(\lambda) = \Pi_1  T_v(L_{w_0} - \lambda)^{-1}  \iota_2$. Note that
$\iota_2 : E_2^{-s;n} \hookrightarrow H^{s;-n}_+$ and $\Pi_1 : H^{s;-n}_+ \to E_1^{-s; n}$ satisfy
$\| \iota_2 \|_{- s; n} = 1$ and $\| \Pi_1 \|_{- s; n} = 1$. Hence it follows  
from \eqref{estimate 1} that for any $n \ge n_0$ and  any $u = w_0 + v \in U^{-s}_w$
\begin{equation}\label{estimate 6}
 \sup_{\lambda \in \text{Vert}_n(1/4)} \| A_{12}(\lambda)  \|_{-s; n} \le \frac{1}{2}\, . 
\end{equation}
 From the estimates   \eqref{estimate 2} - \eqref{estimate 6} and Lemma \ref{equivalence of norms}
 one concludes by \eqref{38i} that 
  there exists $C_B^{N,s} > 0$ so that for any $n \ge n_0$
\begin{equation}\label{estimate 7a}
\sup_{\lambda \in \text{Vert}_n(1/4)}   \| B_1(\lambda)  \|  \le C_B^{N,s} \, \frac{n^{2s}}{n}.
\end{equation}
Since $B =  B_1 (\text{Id}  - A_{11})^{-1}$ one then infers from  \eqref{estimate 3}, by increasing $n_1$, if needed, that 
\begin{equation}\label{estimate 7}
\sup_{\lambda \in \text{Vert}_n(1/4)} \| B(\lambda)  \| \le 2 C_B^{N,s} \, \frac{n^{2s}}{n} \le \frac{1}{2} \, , \qquad \forall \, n \ge n_1 \, .
\end{equation}

\medskip
\noindent 
{\em Proof of Proposition \ref{F analytic general}.}
Since $F_n$, $n \ge 1$, defined  by \eqref{def F_n general}, is analytic on $U^{-s}_w$, it remains to show that 
$\sum_{n \ge n_1} n^{2-2s} |F_n(u) | < \infty$ locally uniformly on $U^{-s}_w$ where $n_1\ge n_0$ is chosen so that 
 \eqref{estimate 3} and \eqref{estimate 7} are satisfied.
By \eqref{identity 1 F_n } - \eqref{identity 2 F_n } and \eqref{vanishing F^0_n},  it suffices to show that for any $0 \le \alpha, \beta \le N$ and 
$G \in \{ B_1 A_{11}, A_{11}B_1, B_1^k \, (1 \le k < K_s), R\},$
$$
\sup_{u \in U^{-s}_w} \,  \sum_{n \ge  n_1} n^{2-2s}  \Big| \frac{1}{2\pi i}
\int_{\partial D_{n}(1/3)}  \frac{1}{\lambda_\beta - \lambda} \langle G(\lambda) f_\alpha \, | \, f_\beta \rangle d \lambda \, \Big| < \infty .
$$

\smallskip
\noindent
{\em Estimate of $R$.}
From the estimates \eqref{estimate 2}, \eqref{estimate 3}, and \eqref{estimate 7a}, \eqref{estimate 7},
one infers that there exists a constant $C_R^{N, s} \ge 1$ so that for any $0 \le \alpha, \beta \le N$ and any $\Xi(\lambda)$ in
$$
\begin{aligned}
&\{ A_{11}B A_{11},  (\text{Id}- A_{11})^{-1}A_{11}^2B,  (\text{Id}- A_{11})^{-1}A_{11} (\text{Id} - B)^{-1} B^2, \\
& BA_{11}^2 ,  B_1^jBA_{11}B^{k-j-1} \, (0 \le j < k, 2 \le k < K_s)   \},
\end{aligned}
$$
$$
\Big| \frac{1}{\lambda_\beta - \lambda} \langle \Xi(\lambda) f_\alpha | f_\beta \rangle \Big|
\le \frac{1}{n -\frac 12 - N} \sup_{\lambda \in {\text Vert}_n(1/4)} \| \Xi(\lambda) \|
\le C_R^{N, s} \frac{1}{n} \frac{n^{4s}}{n^3} ,
$$
implying that for any $0 \le \alpha, \beta \le N$,
\begin{align}\label{estimate R}
\sup_{u \in U_w^{-s}} \sum_{n \ge n_1} n^{2-2s} 
\big| \frac{1}{2\pi i} \int_{\partial D_n(1/3)} \frac{1}{\lambda_\beta - \lambda} \langle \Xi(\lambda) f_\alpha | f_\beta \rangle d \lambda   \big|
\lesssim \sum_{n \ge n_1} C^{N,s}_R \frac{n^{2s}}{n^2} < \infty \, .
\end{align}

The operator $B^{K_s} (\text{Id} - B)^{-1}$ is estimated similarly,
$$
\begin{aligned}
& \Big| \frac{1}{\lambda_\beta - \lambda} \langle B^{K_s} (\text{Id} - B)^{-1} f_\alpha | f_\beta \rangle \Big| \\
&\le \frac{1}{n -\frac 12 - N} \sup_{\lambda \in {\text Vert}_n(1/4)} \| B^{K_s} (\text{Id} - B)^{-1} \|
\le C_R^{N, s} \frac{1}{n} \big( \frac{n^{2s}}{n} \big)^{K_s} .
\end{aligned}
$$
Since $K_s \cdot (1-2s) > 2 - 2s$ by the definition of $K_s$, it follows that
$$
\sum_{n \ge n_1} n^{2-2s} \frac{1}{n} \big( \frac{n^{2s}}{n} \big)^{K_s} < \infty
$$
and we have
\begin{align}\label{estimate Ra}
\sup_{u \in U_w^{-s}} \sum_{n \ge n_1} n^{2-2s} 
\big| \frac{1}{2\pi i} \int_{\partial D_n(1/3)} \frac{1}{\lambda_\beta - \lambda} \langle B^{K_s} (\text{Id} - B)^{-1} f_\alpha | f_\beta \rangle d \lambda   \big| < \infty .
\end{align}

\smallskip
\noindent
{\em Estimate of $G \in \{ A_{11}B_1, \,  B_1A_{11}, \, B_1^k$ ($1 \le k <  K_s$) \}.} 
%First note that
%$$
%B = B_1 + B_1A_{11} + B_1 (\text{Id} - A_{11})^{-1}  A_{11}^2\, , \qquad B_1 := A_{12} (\text{Id}  - A_{22})^{-1} A_{21}\, ,
%$$
%implying that
%$$
%A_{11}B = A_{11}B_1 +  A_{11}B_1 (\text{Id} - A_{11})^{-1}  A_{11} \, ,
%$$
%$$
%B A_{11} = B_1A_{11} +  B_1 (\text{Id} - A_{11})^{-1}  A_{11} ^2\, , \quad
%$$
%and
%$$
%B^2 = B_1^2 + B_1(\text{Id} - A_{11})^{-1}  A_{11}B_1 + BB_1(\text{Id} - A_{11})^{-1}  A_{11}\, .
%$$
%The estimates   \eqref{estimate 2} - \eqref{estimate 6} imply (cf. \eqref{estimate 7}) that there exists $C_{B_1}^{N,s} > 0$ so that for any $n \ge n_1$
%\begin{equation}\label{estimate 8}
%\sup_{\lambda \in \text{Vert}_n(1/4)} \| B_1(\lambda)  \| \le C_{B_1}^{N,s} \frac{1}{n}.
%\end{equation}
%It then follows that 
%$$
%B_1 (\text{Id} - A_{11})^{-1}  A_{11}^2, \qquad A_{11}B_1 (\text{Id} - A_{11})^{-1}  A_{11},
%$$
%$$
%B_1(\text{Id} - A_{11})^{-1}  A_{11}B_1, \qquad BB_1(\text{Id} - A_{11})^{-1}  A_{11}
%$$
%can be treated in the same way as $R$ and 
The operators in the latter list will be treated individually. 

\smallskip
\noindent
{\em Estimate of $B_1$.} We claim that for any given $0 \le \alpha, \beta \le N$,
\begin{equation}
\sup_{u \in U^{-s}_w} \sum_{n \ge  n_1} n^{2-2s}  |\frac{1}{2\pi i}
\int_{\partial D_{n}(1/3)}  \frac{1}{\lambda_\beta - \lambda} \langle B_1(\lambda) f_\alpha \, | \, f_\beta \rangle d \lambda \, |  < \infty \, .
\end{equation}
One computes, using the definition \eqref{38i} of $B_1$ and the identities $f_k = g_\infty e^{ikx}$, $\lambda_k = k$, which are valid for any $k > N$, 
$$
\begin{aligned}
\langle B_1 f_\alpha \, | \, f_\beta \rangle & = \sum_{k, \ell > N} \frac{\langle v f_\alpha | f_k \rangle}{\lambda_\alpha - \lambda} 
\frac{\overline{\langle \overline v f_\beta | f_\ell \rangle}}{\ell - \lambda} 
\langle (\text{Id}  - A_{22})^{-1} f_k | f_\ell \rangle \\
& = \sum_{k, \ell > N} \frac{\langle v f_\alpha \overline g_\infty | e^{ikx} \rangle}{\lambda_\alpha - \lambda} 
\frac{\langle v \overline f_\beta g_\infty | e^{-i\ell x} \rangle}{\ell - \lambda} 
\langle (\text{Id}  - A_{22})^{-1} f_k | f_\ell \rangle \, .
\end{aligned}
$$
By \eqref{estimate 5a} it follows that for any $\lambda \in \text{Vert}_n(1/4)$ with $n \ge n_0,$
one can write $(\text{Id}  - A_{22})^{-1}$ as a Neumann series, $\sum_{m \ge 0}  (A_{22})^m $.
We are thus lead to introduce for any $m \ge 0$ and any $k, \ell > N,$
$$
G_{n, k, \ell, m} :=  \frac{1}{2\pi i} \int_{\partial D_{n}(1/3)}  \frac{1}{(\lambda_\beta - \lambda)(\lambda_\alpha - \lambda)(\ell - \lambda)}
\langle A_{22}(\lambda)^m f_k \, | \, f_\ell \rangle d \lambda
$$
so that
$$
 \frac{1}{2\pi i} \int_{\partial D_{n}(1/3)} \frac{ \langle B_1(\lambda) f_\alpha \, | \, f_\beta \rangle }{\lambda_\beta - \lambda} d \lambda =
 \sum_{k, \ell > N, m \ge 0}  \widehat{ v f_\alpha \overline g_\infty}(k) \, \widehat{ v \overline f_\beta g_\infty}(-\ell )
 G_{n, k, \ell, m}\, .
$$
For $m = 0$, we get by Cauchy's theorem
$$
\begin{aligned}
 G_{n, k, \ell, 0} & =  \frac{1}{2\pi i} \int_{\partial D_{n}(1/3)}  \frac{1}{(\lambda_\beta - \lambda)(\lambda_\alpha - \lambda)(\ell - \lambda)} \langle f_k \, | \, f_\ell \rangle d \lambda\\
 & = - \delta_{k\ell} \delta_{n\ell} \frac{1}{(\lambda_\beta - n)(\lambda_\alpha - n)}\, ,
 \end{aligned}
$$
implying that
$$
\begin{aligned}
& \sum_{n \ge  n_1} n^{2-2s}  \big|  \sum_{k, \ell > N}  \, \widehat{ v f_\alpha \overline g_\infty}(k) \, \widehat{ v \overline f_\beta g_\infty}(-\ell ) G_{n, k, \ell, 0} | \\
& \le  \sum_{n \ge  n_1} \frac{n^2}{(n - \lambda_\alpha)(n - \lambda_\beta)}  
\frac{| \widehat{ v f_\alpha \overline g_\infty}(n)|}{n^s} \frac{|\widehat{ v \overline f_\beta g_\infty}(-n ) |}{n^s}\\
& \lesssim \|  v f_\alpha \overline g_\infty\|_{-s} \|v \overline f_\beta g_\infty\|_{-s} \, .
\end{aligned}
$$
We thus have proved that 
\begin{equation}\label{estimate for m=0}
 \sum_{n \ge  n_1} n^{2-2s}  \big|  \sum_{k, \ell > N}  \, \widehat{ v f_\alpha \overline g_\infty}(k) \, \widehat{ v \overline f_\beta g_\infty}(-\ell )  G_{n, k, \ell, 0} |  
 \lesssim  \|v\|_{-s}^2\, .
\end{equation}
Let us now consider the case $m =1$. Note that for any $k > N,$ one has
$$
 A_{22} f_k =  \Pi_2T_v(L_{w_0} - \lambda)^{-1}  f_k = \frac{\Pi_2(v f_k)}{k - \lambda}\, .
$$
Since $|g_\infty (x)| =1$ one then infers that for any $\ell > N$
$$
\langle  A_{22} f_k  | f_\ell \rangle =   \frac{\langle v g_\infty e^{ikx}  | g_\infty e^{i \ell x}\rangle }{k - \lambda} = \frac{ \widehat v(\ell - k)}{k - \lambda}\, .
$$
Hence using that $\widehat v(0) = 0,$  one sees that for $n \ge n_1$ and $(k, \ell) \ne (n, n)$, $G_{n, k, \ell, 1}$ can be computed as
$$
\begin{aligned}
 G_{n, k, \ell, 1} & =  \frac{1}{2\pi i} \int_{\partial D_{n}(1/3)}  \frac{ \widehat v(\ell - k)}{(\lambda_\beta - \lambda)(\lambda_\alpha - \lambda)(\ell - \lambda)(k - \lambda)} d \lambda\\
 & = \frac{\widehat v(n- k) \delta_{\ell n}}{(\lambda_\beta - n)(\lambda_\alpha - n)(n - k)} +  \frac{\widehat v(\ell- n) \delta_{k n}}{(\lambda_\beta - n)(\lambda_\alpha - n)(n - \ell )}  \, ,
 \end{aligned}
$$
whereas for $(k, \ell) = (n, n)$ one has $G_{n, n, n, 1} = 0$.
This implies that
$$
\begin{aligned}
& \sum_{n \ge  n_1} n^{2-2s}  |  \sum_{k, \ell > N}  \widehat{ v f_\alpha \overline g_\infty}(k) \, \widehat{ v \overline f_\beta g_\infty}(-\ell )  G_{n, k, \ell, 1} | \\
& \lesssim  \sum_{n \ge  n_1}   \frac{|\widehat{ v \overline f_\beta g_\infty}(-n ) |}{n^s} \,  
\frac{ 1}{n^s}  \sum_{k > N, k \ne n } | \widehat{ v f_\alpha \overline g_\infty}(k)|  \frac{|\widehat v(n- k)|}{|n- k|}  \\
 & +  \sum_{n \ge  n_1}   \frac{|\widehat{ v f_\alpha \overline g_\infty}(n ) |}{n^s} \,  
 \frac{ 1}{n^s}  \sum_{\ell > N, \ell \ne n} | \widehat{ v \overline f_\beta g_\infty}(- \ell)|  \frac{|\widehat v( \ell -n )|}{| \ell - n|}  \\
& \lesssim \|  v f_\alpha \overline g_\infty\|_{-s} \|v \overline f_\beta g_\infty\|_{-s} \|v\|_{-s}\, .
\end{aligned}
$$
Hence
$$
 \sum_{n \ge  n_1} n^{2-2s}  |  \sum_{k, \ell > N}  \widehat{ v f_\alpha \overline g_\infty}(k) \, \widehat{ v \overline f_\beta g_\infty}(-\ell )   G_{n, k, \ell, 1} | 
 \lesssim \|v\|_{-s}^3 \, .
$$
Now we consider the case $m \ge 2$. One computes
$$
\langle  (A_{22})^2 f_k  | f_\ell \rangle =  
\sum_{k_1 > N} \frac{\langle v f_k  | f_{k_1} \rangle \langle v f_{k_1} | f_\ell \rangle }{(k - \lambda) (k_1 - \lambda)} = 
\sum_{k_1 > N} \frac{\widehat v(k_1 - k) \widehat v(\ell - k_1)}{(k - \lambda)(k_1 - \lambda)} \, .
$$
%\textcolor{red}{NOT TO BE INCLUDED Using Cauchy's formula for derivatives of an analytic function and again the assumption $\widehat v(0) = 0$,  
%$ G_{n, k, \ell, 2}$ can be computed as
%$$
%\begin{aligned}
%G_{n, k, \ell, 2}  & = 
% \sum_{k_1 > N} \frac{1}{2\pi i} \int_{\partial D_{n}(1/3)}  
% \frac{ \widehat v(k_1 - k) \widehat v(\ell - k_1)}{(\lambda_\beta - \lambda)(\lambda_\alpha - \lambda)(\ell - \lambda)(k - \lambda)(k_1 - \lambda)} d \lambda\\
%&=  \frac{\widehat v(k_1 - k)  \widehat v(n- k_1) 1_{\{\ell = n, k \ne n, k_1\ne n\}}}{(\lambda_\beta - n)(\lambda_\alpha - n)(k - n)(n - k_1) } \\
% &\quad  +  \frac{\widehat v(k_1 - n)  \widehat v(\ell - k_1) 1_{\{\ell \ne n, k = n, k_1\ne n\}}}{(\lambda_\beta - n)(\lambda_\alpha - n)(\ell - n)(n - k_1) } \\
%&\quad + \frac{\widehat v(n - k)  \widehat v(\ell - n) 1_{\{\ell \ne n, k \ne n, k_1= n\}}}{(\lambda_\beta - n)(\lambda_\alpha - n)(\ell - n)(n - k) }\\
%&\quad + \widehat v(k_1 - n)  \widehat v(n - k_1) 1_{\{\ell = n, k = n, k_1 \ne n\}} \mu(k_1) \end{aligned}
%$$
%where
%$$
%\mu(k_1) := \frac{d}{d\lambda}|_{\lambda = n}\frac{1}{(\lambda_\beta - \lambda)(\lambda_\alpha - \lambda)(k_1 - \lambda) }\, .
%$$}
More generally, for any $m \ge 2$, the expression $\langle  (A_{22})^m f_k  | f_\ell \rangle$ can be expanded as
$$
%\langle  A_{22} (\lambda)^m f_k  | f_\ell \rangle =
\sum_{k_j > N} \frac{\widehat v(k_1 - k)\widehat v(k_2 - k_1)  \cdots  \widehat v(k_{m-1} - k_{m-2}) \widehat v(\ell - k_{m-1})}{(k - \lambda)(k_1 - \lambda) \cdots (k_{m-1} - \lambda)} \, ,
$$
leading to the following formula for $G_{n, k, \ell, m} $,
$$
 \sum_{k_j > N} \frac{1}{2\pi i} \int_{\partial D_{n}(1/3)} 
 \frac{\widehat v(k_1 - k)\widehat v(k_2 - k_1)  \cdots  \widehat v(k_{m-1} - k_{m-2}) \widehat v(\ell - k_{m-1})}
 {(\lambda_\beta - \lambda)(\lambda_\alpha - \lambda)(\ell - \lambda)(k - \lambda) (k_1 - \lambda) \cdots (k_{m-1} - \lambda)}
d \lambda\, .
$$
One then shows by the arguments of the proof of Proposition \ref{F analytic near 0} that
$$
\sup_{u \in U^{-s}_w} \sum_{n \ge  n_1} n^{2-2s}  \Big| \frac{1}{2\pi i}
\int_{\partial D_{n}(1/3)}  \frac{1}{\lambda_\beta - \lambda} \langle B_1(\lambda) f_\alpha \, | \, f_\beta \rangle d \lambda \, \Big|  < \infty \, .
$$

\smallskip
\noindent
{\em Estimate of $A_{11}(\lambda)B_1(\lambda)$.} For any $0 \le \alpha, \beta \le N$,
$$
 \frac{1}{\lambda_\beta - \lambda}\langle A_{11} B_1 f_\alpha | f_\beta \rangle 
 =   \sum_{\beta_1 \le N} \langle v \overline f_\beta | \overline f_{\beta_1} \rangle
 \frac{1}{\lambda_\beta - \lambda} \frac{1}{\lambda_{\beta_1} - \lambda} \langle B_1 f_\alpha \, | \, f_{\beta_1} \rangle 
$$
and hence $\langle A_{11} B_1 f_\alpha | f_\beta \rangle$ can be treated in the same way as
$ \langle B_1 f_\alpha \, | \, f_{\beta} \rangle $.

\medskip
\noindent
{\em Estimate of $B_1(\lambda) A_{11}(\lambda)$.} For any $0 \le \alpha, \beta \le N$,
$$
 \frac{1}{\lambda_\beta - \lambda}\langle  B_1 A_{11} f_\alpha | f_\beta \rangle = \sum_{\alpha_1 \le N} \langle v  f_\alpha |  f_{\alpha_1} \rangle
 \frac{1}{\lambda_\beta - \lambda} \frac{1}{\lambda_{\alpha} - \lambda} \langle B_1 f_{\alpha_1} \, | \, f_{\beta} \rangle 
$$
and so $\langle  B_1 A_{11} f_\alpha | f_\beta \rangle$ can also be treated in the same way as
$ \langle B_1 f_\alpha \, | \, f_{\beta} \rangle $.

\medskip
\noindent
{\em Estimate of $B_1(\lambda)^2$.} The expression $\langle B_1^2 f_\alpha \, | \, f_\beta \rangle$ can be expanded as
$$
\begin{aligned}
& \sum_{k, \ell > N} \frac{\langle v f_\alpha | f_k \rangle}{\lambda_\alpha - \lambda} 
\frac{\overline{\langle \overline v f_\beta | f_\ell \rangle}}{\ell - \lambda} 
\langle (\text{Id}  - A_{22})^{-1} A_{21}A_{12}(\text{Id}  - A_{22})^{-1}  f_k | f_\ell \rangle \\
& = \sum_{k, \ell > N} \frac{\widehat{ v f_\alpha \overline g_\infty} (k) }{\lambda_\alpha - \lambda} 
\frac{\widehat{ v \overline f_\beta g_\infty}(-\ell)}{\ell - \lambda} 
\langle (\text{Id}  - A_{22})^{-1}  A_{21}A_{12}(\text{Id}  - A_{22})^{-1}  f_k | f_\ell \rangle \, .
\end{aligned}
$$
Since $A_{12} f_{k_1} = \sum_{\alpha_1 \le N} \frac{\langle vf_{k_1} | f_{\alpha_1} \rangle}{k_1 - \lambda} f_{\alpha_1}$ one has
$$
\begin{aligned}
A_{21}A_{12} f_{k_1}
&= \sum_{\alpha_1 \le N, k_2 > N} \frac{\langle vf_{k_1} | f_{\alpha_1} \rangle}{k_1 - \lambda} \frac{ \langle v f_{\alpha_1} | f_{k_2} \rangle }{\lambda_{\alpha_1} - \lambda} f_{k_2} \\
& = \sum_{\alpha_1 \le N, k_2 > N} \frac{\widehat{ v \overline f_{\alpha_1} g_\infty}(-k_1)}{k_1 - \lambda} 
\frac{ \widehat{ v f_{\alpha_1} \overline g_\infty}( k_2) }{\lambda_{\alpha_1} - \lambda} f_{k_2}
\end{aligned}
$$
and hence $\langle B_1^2 f_\alpha \, | \, f_\beta \rangle$ can be written as 
$$
\langle B_1^2 f_\alpha \, | \, f_\beta \rangle = \sum_{\alpha_1 \le N} H_{1, \alpha_1} H_{2, \alpha_1}
$$
where $H_{j, \alpha_1} \equiv H_{j, \alpha_1}(\lambda)$, $j = 1,2$, are given by
$$
H_{1, \alpha_1} :=  \sum_{k, k_1 >N}  \frac{\widehat{ v f_\alpha \overline g_\infty} (k) }{\lambda_\alpha - \lambda} 
\frac{\widehat{ v \overline{f}_{\alpha_1}  g_\infty }(-k_1)}{k_1 - \lambda} \langle (\text{Id}  - A_{22})^{-1} f_k | f_{k_1} \rangle
$$
and
$$
H_{2, \alpha_1} :=  \sum_{\ell, k_2 >N}  \frac{\widehat{ v \overline f_\beta g_\infty}(-\ell)}{\ell - \lambda} 
\frac{\widehat{ v f_{\alpha_1} \overline g_\infty}(k_2)  }{\lambda_{\alpha_1} - \lambda}  \langle (\text{Id}  - A_{22})^{-1} f_{k_2} | f_{\ell} \rangle \, .
$$
Writing $(\text{Id}  - A_{22})^{-1}$ as a Neumann series, $\sum_{m \ge 0}  (A_{22})^m $, and using that
for any $m \ge 2$, $\langle  (A_{22})^m f_k  | f_\ell \rangle$ can be expanded as
$$
\sum_{k_j > N} \frac{\widehat v(k_1 - k)\widehat v(k_2 - k_1)  \cdots  \widehat v(k_{m-1} - k_{m-2}) \widehat v(\ell - k_{m-1})}{(k - \lambda)(k_1 - \lambda) \cdots (k_{m-1} - \lambda)},
$$
one then shows again by the arguments of the proof of Proposition \ref{F analytic near 0} that
$$
\sup_{u \in U^{-s}_w} \sum_{n \ge  n_1} n^{2-2s}  \Big| \frac{1}{2\pi i}
\int_{\partial D_{n}(1/3)}  \frac{1}{\lambda_\beta - \lambda} \langle B_1(\lambda)^2 f_\alpha \, | \, f_\beta \rangle d \lambda \, \Big|  < \infty \, .
$$

\medskip
\noindent
{\em Estimate of $B_1(\lambda)^k,$ $k \ge 3$.} 
One proceeds as in the case $k=2$.
\hfill $\square$

%%%%%%%%%%%%%%%%%%%%%%%%%%%%%%%%%%%%%%%%%%%%%%%%%%%%%%%%%%%%%%%
%%%%%%%%%%%%%%%%%%%%%%%%%%%%%%%%%%%%%%%%%%%%%%%%%%%%%%%%%%%%%%%%

 \section{Proof of Theorem \ref{main theorem}($ii$) and Corollary \ref{product representations}}\label{proof of main results}
 
 In this section we prove Theorem \ref{main theorem}(ii), saying that for any $0 \le s < 1/2$, the moment map 
 $$
 \Gamma : H^{-s}_{r,0} \to \ell^{1,1-2s}(\N, \R) \,: \, u \mapsto (\gamma_n(u))_{n \ge 1}
 $$ 
 analytically extends to a neighborhood of $H^{-s}_{r, 0}$ in $H^{-s}_{c, 0}$. 
 By Theorem \ref{main theorem near zero}, we already know that such an extension exists in a neighborhood of $0$ in $H^{-s}_{c, 0}$.
 Since by Theorem  \ref{main theorem}($i$),
 the gaps $\gamma_n(u)$, $n \ge 1$, analytically extend to a neighborhood  $H^{-s}_{r, 0}$, which is independent of $n$,
 it remains to show that for any $w \in H^{-s}_{r,0}$, there exists a neighborhood $U^{-s}_{w}$ of $w$ in $H^{-s}_{c,0}$ so that
 $\sup_{u \in U^{-s}_{w}}\sum_{n \ge n_*} n^{1-2s} |\gamma_n(u)| < \infty$ for some  $n_* \ge 1$. 
Since we only need to consider $\gamma_n(u)$ for $n$ large, we again approximate
 $u$ by a finite gap potential, which allows to use arguments of the proof of Theorem \ref{main theorem near zero}.
 
 Theorem \ref{main theorem}(ii) will then be applied to prove Corollary \ref{product representations} and to
  define for all $n \ge 1$ and $u \in U^{-s}_{w}$ the scaling factor 
 $\kappa_n(u)$ and the normalizing constant $\mu_n(u)$  by extending product formulas
 of these quantities, established in \cite{GK}, \cite{GKT} for $u \in H^{-s}_{r,0}$.
 In addition, we prove properties of $\kappa_n(u)$ and $\mu_n(u)$, $n \ge 1$, needed in subsequent work.
  
 \smallskip
 
 As a first step, we study the extension of the generating function $\mathcal H_{\lambda}(u)$.
 For $w \in H^{-s}_{r,0}$ with $0 \le s < 1/2$ and $M:= \|w\|_{-s}$, 
 choose $w_0 \in \mathcal U_N$ with $N\ge 1$, satisfying \eqref{estimate w_0}, and $n_0 > N$,
 so that  Lemma \ref{T u (D - lambda)-1 general case} holds.
 Denote by $U^{-s}_w$ the neighborhood of $w$ in $H^{-s}_{c,0}$,
 given by Theorem \ref{general counting lemma}  with $\rho = 1/4$ and $M= \|w\|_{-s}$. Then for any $u \in U^{-s}_w$, 
 $$
\lambda \mapsto  \mathcal H_{\lambda}(u) = \langle (L_u - \lambda)^{-1} 1 | 1 \rangle 
 $$
 defines a meromorphic function on $\C$ with the set of poles contained in $\text{spec}(L_u)$.
 By the definition \eqref{def F_n general}, it follows that for any $n \ge 0$ and $u \in U^{-s}_w$, the residue
 of  $ \mathcal H_{\lambda}(u)$ at $\lambda = \lambda_n(u)$ is given by $F_n(u)$
 and arguing as in the proof of Lemma \ref{mathcal H as a sum}, one sees that 
 %$F_n(u) =  - \langle P_n(u) 1 | 1 \rangle$,  \quad $\forall \, n \ge 0$;
  $$
  \mathcal H_\lambda(u) = \sum_{n \ge 0} \frac{F_n(u)}{\lambda - \lambda_n(u)} \, .
  $$
  Furthermore, the arguments in the proof of Lemma \ref{zeros of mathcal H} show that for any  $u \in U^{-s}_w$ and $n \ge 1$ with $\gamma_n(u) \ne 0, $
  \begin{equation}\label{zeroes of mathcal H for u general}
   \mathcal H_{\lambda_{n-1}(u) +1}(u) = 0\, .
  \end{equation}
Denote by $f_n \equiv f_n(\cdot, w_0)$, $n \ge 0$, the eigenfunctions of $L_{w_0}$, corresponding to the eigenvalues $\lambda_n \equiv \lambda_n(w_0)$, $n \ge 0$,
 normalized as in \cite{GK}. Then $\lambda_n = n$ and $f_n= g_\infty e^{inx}$ for any $n \ge N$ where $g_\infty = e^{i \partial_x^{-1}w_0}$.
 By  \eqref{expansion of 1}, one has $ 1 = \sum_{\alpha \le N} \langle 1 | f_\alpha \rangle f_\alpha$
 and by \eqref{formula 0}
 $$
  \mathcal H_\lambda(u) =
   \sum_{\alpha, \beta \le N} \frac{ \langle 1 | f_\alpha \rangle  \langle f_\beta | 1 \rangle  }{\lambda_\beta - \lambda} 
   \langle ( \text{Id} - T_v(L_{w_0} - \lambda)^{-1})^{-1} f_\alpha | f_\beta \rangle \, .
$$
 We need the following estimates for $\mathcal H_\lambda(u)$
  \begin{lemma}\label{estimate of mathcal H for general u}
  There exist $n_2 \ge n_1$ with $n_1$ given by  \eqref{estimate 3},
  %and a constant $C_{\mathcal H} \ge 1$
  so that for any $n \ge n_2$ and $u \in U_w^{-s}$, 
 $$
 \frac{1}{2| \lambda |} \le | \mathcal H_\lambda(u) | \le   \frac{2}{| \lambda |} \, ,
\quad \qquad \forall  \, \lambda \in  \rm{Vert}_n(1/4) \, .
$$
 \end{lemma}
 \begin{proof}
 We adapt the proof of Lemma \ref{estimate of mathcal H} to the more general situation at hand.
 For any $u = w_0 + v \in U^{-s}_w$ and $n \ge n_0$, we use
 $$
 ( \text{Id} - T_v(L_{w_0} - \lambda)^{-1})^{-1} =  \text{Id} + ( \text{Id} - T_v(L_{w_0} - \lambda)^{-1})^{-1}T_v(L_{w_0} - \lambda)^{-1}  \, ,
 $$
 to expand $ \mathcal H_\lambda(u)$,
 \begin{equation}\label{expansion mathcal H(u)}
 \mathcal H_\lambda(u) =
  \mathcal H_{\lambda, 0}  + \frac{1}{\lambda^2}\sum_{\alpha, \beta \le N} 
   \frac{\langle 1 | f_\alpha \rangle  \langle f_\beta | 1 \rangle}{(1 - \lambda_\alpha / \lambda)(1 -\lambda_\beta /\lambda)} R_{\alpha, \beta} \, ,
\end{equation}
where the principal term $ \mathcal H_{\lambda, 0}$ and the remainders $R_{\alpha, \beta}$ are given by 
$$
 \mathcal H_{\lambda, 0} : = \sum_{\alpha \le N} \frac{ |\langle 1 | f_\alpha \rangle |^2 }{\lambda_\alpha - \lambda}  \, ,
 \qquad
R_{\alpha, \beta} := \langle ( \text{Id} - T_v(L_{w_0} - \lambda)^{-1})^{-1} \Pi(v f_\alpha) | f_\beta \rangle \, .
$$
The term $\mathcal H_{\lambda, 0}$ is expanded further,
$$
\mathcal H_{\lambda, 0} 
= - \frac{1}{\lambda} \sum_{\alpha \le N} \frac{ |\langle 1 | f_\alpha \rangle |^2 }{ 1 - \lambda_\alpha / \lambda} 
= - \frac{1}{\lambda} \big( 1 +  \sum_{\alpha \le N} \frac{ |\langle 1 | f_\alpha \rangle |^2 }{ 1 - \lambda_\alpha / \lambda} \, \frac{\lambda_\alpha}{\lambda} \big) \, ,
$$ 
where we used that  $ \sum_{\alpha \le N}  |\langle 1 | f_\alpha \rangle |^2 =1$.
Choose $n_2 \ge n_1$ so that for any $n \ge n_2$
\begin{equation}\label{lower and upper bounds}
 \sum_{\alpha \le N} \big| \frac{ |\langle 1 | f_\alpha \rangle |^2 }{ 1 - \lambda_\alpha / \lambda} \,  \frac{\lambda_\alpha}{\lambda} \big|  \le \frac 14 \, ,
 \qquad \forall  \, \lambda \in  \rm{Vert}_n(1/4) \, .
\end{equation}
 Now let us estimate the remainder terms $R_{\alpha, \beta}$.  By Corollary \ref{invertibility general case} with 
 $M= \|w\|_{-s}$, $\rho = 1/4$, one has for any $n \ge n_0$ and any $u = w + v \in U_{w}^{-s}$
$$
\| ( \text{Id} - T_v(L_{w_0} - \lambda)^{-1})^{-1} \|_{H^{-s; n}_+ \to H^{-s; n}_+}  \le 2 \, , \qquad \forall \, \lambda \in \text{Vert}_n(1/4) \, .
$$
One then infers from Lemma \ref{equivalence of norms} and Lemma \ref{multiplication u f} that
$$
\begin{aligned}
\| ( \text{Id} - & T_v(L_{w_0} - \lambda)^{-1})^{-1} \Pi(v f_\alpha) \|_{-s}  \\
&\le  2^s n^s \| ( \text{Id} - T_v(L_{w_0} - \lambda)^{-1})^{-1} \Pi(v f_\alpha) \|_{-s; n}\\
& \le 2^s n^s  2 \| \Pi(v f_\alpha) \|_{-s; n} \le  2^{1+s} n^s  2^s n^s  \| v f_\alpha) \|_{-s}\\
& \le  2^{1 + 2s} n^{2s}  C_s \|v\|_{-s} \|f_\alpha \|_{1-s}
\end{aligned}
$$
with $C_s$ given as in Lemma \ref{multiplication u f}.
It then follows that for any $0 \le \alpha, \beta \le N$, 
\begin{equation}\label{estimate error R alpha, beta}
| R_{\alpha, \beta} |  
 \le 2^{1 +2s}C_s n^{2s}  \|v \|_{-s} \max_{\alpha \le N} \|f_\alpha \|^2_{1-s} \, , \quad \forall \, \lambda  \in \text{Vert}_n(1/4)\, .
\end{equation}
Since $0 \le 2s < 1$, one then concludes by increasing $n_2$, if needed, that for any $n \ge n_2$
$$
\frac{1}{|\lambda|}  \sum_{\alpha, \beta \le N} 
 \big|  \frac{\langle 1 | f_\alpha \rangle  \langle f_\beta | 1 \rangle}{(1 - \lambda_\alpha / \lambda)(1 -\lambda_\beta /\lambda)} R_{\alpha, \beta} \big|
  \le \frac 14 \, , \quad \forall \, \lambda  \in \text{Vert}_n(1/4) \, .
$$
Combining this with \eqref{expansion mathcal H(u)} - \eqref{lower and upper bounds} then yields the claimed estimates.
 \end{proof}
 Arguing is in the proof of Lemma \ref{argument principle}
 one sees that Lemma \ref{estimate of mathcal H for general u} yields the following lemma and its corollary.
 \begin{lemma}\label{argument principle for u general}
 For any $u \in U^{-s}_w$ and $n \ge n_2$ (with $n_2$ as in Lemma \ref{estimate of mathcal H for general u}), 
 the difference $ZP_{u;n}$ of the number of zeroes of $\mathcal H_\lambda(u)$ in $D_n(1/3)$
 and the number of its poles  in $D_n(1/3)$ vanishes.  As a consequence, for any $n \ge n_2$, 
 $\mathcal H_\lambda(u)$  has at most one zero in $D_n(1/3)$.
 \end{lemma}
\begin{corollary}\label{residue and gamma for u general}
 Assume that $u \in U^{-s}_w$ with $0 \le s < 1/2$. For any $n \ge n_2,$  $F_n(u) = 0$ if and only if $\gamma_n(u)=0$.
 \end{corollary}
 
 For any $u \in U^{-s}_w$ and $n \ge n_2$, define (cf. \eqref{def eta near zero})
  $$
  \eta_n(\lambda, u) := - 
  \frac{\lambda -  \lambda_n(u)}{\lambda -  \lambda_{n-1} (u) - 1 } (\lambda - \lambda_0(u))\mathcal H_\lambda(u)\, ,
  \quad \forall \lambda \in \rm{Vert}_n(1/4)\, .
 $$
For any $n \ge n_2$, $\eta_n(\cdot, u)$ is analytic on $\rm{Vert}_n(1/4)$ and 
by \eqref{zeroes of mathcal H for u general},  Lemma \ref{argument principle for u general}, 
and Corollary \ref{residue and gamma for u general}
analytically extends to $D_n(1/4)$. Hence for any $\lambda \in D_n(1/4)$,
one has by Cauchy's formula
\begin{equation}\label{Cauchy eta_n}
\eta_n(\lambda, u) = \frac{1}{2\pi i} \int_{\partial D_n(1/3)} \frac{\eta_n(\mu, u)}{\mu - \lambda} d \mu .
\end{equation}
As a consequence, for any $n \ge n_2$,
$$
\eta_n : \big( D_n(1/4) \cup \rm{Vert}_n(1/4) \big) \times U^{-s}_w \to \C
$$
is analytic and by Lemma \ref{argument principle} vanishes nowhere.
Arguing as in the proof of Lemma \ref{estimate of eta_n}
and using that $ | \lambda_0(u) - \lambda_0(w)| \le 1/4$ for any $u \in U^{-s}_w$ (cf. Theorem \ref{general counting lemma})
and $\lambda_0(w) \le 0$ (\cite{GK}($s=0$), \cite{GKT} ($0 < s < 1/2$))
one obtains the following
\begin{lemma}\label{estimate of eta_n for u general}
For any $u \in U^{-s}_w$ and $n \ge n_2$ (with $n_2$ as in Lemma \ref{estimate of mathcal H for general u}),
$$
\frac{1}{C_w} \le | \eta_n(\lambda, u) | \le C_w\, , \qquad \forall \lambda \in D_n(1/4) \, ,
$$
where $C_w \ge 1$ is a constant, only depending on $\lambda_0(w)$.
\end{lemma}

 \smallbreak

For any $n \ge n_2$, define (cf. \eqref{def kappa_n near zero})
$$
\kappa_n : U^{-s}_w \to \C, \ u \mapsto  \frac{1}{\lambda_n(u) - \lambda_0(u)} \eta_n(\lambda_n(u), u) \, .
$$
Being a composition of analytic functions, $\kappa_n$ is analytic as well.
Arguing as in the proof of Proposition \ref{proposition kappa_n} one obtains the following
\begin{proposition}\label{proposition kappa_n for u general}
For any $n \ge n_2$ (with $n_2$ as in Lemma \ref{estimate of mathcal H for general u}), 
the map $\kappa_n :  U^{-s}_w \to \C$ is analytic.
For any $u \in U^{-s}_w$, one has
$$
| \kappa_n(u) | \le 2 C'_w \frac{1}{n} \, , \qquad
\frac{1} {| \kappa_n(u) |} \le 2C'_w  n \, ,  \qquad  \forall \, n \ge n_2 \, ,
$$
where $C'_w \ge 1$ is a constant, only depending on $\lambda_0(w)$.
In particular, the map 
$$
U^{-s}_w \to \ell^{\infty}(\Z_{\ge n_2}, \C), u \mapsto (1 / n \kappa_n(u))_{n \ge n_2}
$$ 
is analytic.
\end{proposition}

Arguing as in the proof of Corollary \ref{identity residues}
one sees, 
using our results on $F_n(u)$ and the gap lengths $\gamma_n(u)$, 
that Proposition \ref{proposition kappa_n for u general} yields the following 
\begin{corollary}\label{identity residues for u general}
Let $0 \le s < 1/2.$
  For any $u$ in $U^{-s}_w$,
  \begin{equation}\label{identity quasi moments}
  F_n(u) = - \kappa_n(u) \gamma_n(u)\, , \qquad \forall n \ge n_2\,  .
  \end{equation}
\end{corollary} 

\medskip 

With all these preliminary results established, we now can prove Theorem \ref{main theorem}(ii) and Corollary \ref{product representations}.
We reformulate Theorem \ref{main theorem}(ii) as follows:
 \begin{theorem}\label{main theorem (ii) reformulated}
Let $0 \le s < 1/2$, $w \in H^{-s}_{r,0}$, and let $ U^{-s}_w$ be the neighborhood of $w$ in $H^{-s}_{c,0}$ of Theorem \ref{general counting lemma}.
Then the moment map
$$
\Gamma :  U^{-s}_w \to \ell^{1, 1- 2s}_+, \ u \mapsto (\gamma_n(u))_{n \ge 1}
$$
is analytic. In addition, $\Gamma( U^{-s}_w)$ is bounded in $\ell^{1, 1- 2s}_+$.
\end{theorem}
\begin{proof}
Let $0 \le s < 1/2$ and $w \in H^{-s}_{r,0}$
and denote by $U^{-s}_w$ the neighborhood of $w$ in $H^{-s}_{c,0}$ of Theorem \ref{general counting lemma}.
Since by Proposition \ref{F analytic general},
 $$
 U^{-s}_w \to \ell^{1, 2-2s}(\Z_{\ge n_2}, \C), u \mapsto (F_n(u))_{n \ge n_2}
 $$ 
and by Proposition \ref{proposition kappa_n for u general},
$$
U^{-s}_w \to \ell^{\infty}(\Z_{\ge n_2}, \C), u \mapsto (1/n \kappa_n(u))_{n \ge n_2}
$$ 
 are analytic, it follows from  \eqref{identity quasi moments} that
 $$
 U^{-s}_w  \to \ell^{1, 1-2s}(\Z_{\ge n_2}, \C)\, , \, u \mapsto (\gamma_n(u))_{n \ge n_2}
 $$ 
 is analytic as well. 
 Since by Theorem \ref{main theorem}(i), for any $n \ge 1$, $\gamma_n$ is analytic on $U^{-s}_w $
 it then follows  that $\Gamma  :  U^{-s}_w \to \ell^{1, 1- 2s}_+$ is analytic.
 By Proposition \ref{F analytic general} and 
 Proposition \ref{proposition kappa_n for u general}, $\Gamma( U^{-s}_w)$ is bounded in $\ell^{1, 1- 2s}_+$.
 \end{proof}
 
 \medskip
 \noindent
{\em Proof of Corollary \ref{product representations}.} 
Arguing as in the proofs of Corollary \ref{product representations near zero} (i) - (iii),  the claimed identities follow.
\hfill $\square$

\medskip
Using Theorem \ref{main theorem (ii) reformulated} and the arguments of the proof of Corollary \ref{product representations near zero} (iv), 
it follows that for any $u \in U^{-s}_w$,  $\kappa_n(u)$, $n \ge n_2$, admit  product representations, 
\begin{equation}\label{kappa for u general and n large}
\kappa_n(u) = \frac{1}{\lambda_n(u) - \lambda_0(u)}  \prod_{p \ne n} \Big(1 - \frac{\gamma_p(u)}{\lambda_p(u) - \lambda_n(u)} \Big)
\, , \qquad \forall \, n \ge n_2\, ,
\end{equation}
where the infinite products are absolutely convergent. Furthermore,
Theorem \ref{main theorem (ii) reformulated} allows to define the scaling factors $\kappa_n$ also for the integers $0 \le n < n_2$.
 For any $1 \le n < n_2$, $\kappa_n : U^{-s}_w \to \C$ is {\em defined} by the right hand side of \eqref{kappa for u general and n large},
whereas for $n=0$,
$$
\kappa_0(u) :=  \prod_{p \ge 1} \Big(1 - \frac{\gamma_p(u)}{\lambda_p(u) - \lambda_0(u)} \Big) \, .
$$
All these infinite products are  bounded on $U^{-s}_w$ and hence analytic on $U^{-s}_w$.
We point out, that  in the case $u$ is real valued, these definitions coincide with the product representations obtained in \cite{GK}($s=0$) and \cite{GKT}($0 < s < 1/2$).

Finally we consider the spectral invariants $\mu_n \equiv \mu_n(u)$, $n \ge 1$, discussed at the end of Section \ref{proof of main results near zero}.
For $u \in L^2_{r,0},$ they admit the product representations \eqref{formula m_n for u real}, 
   $$
    \mu_n = \big( 1 - \frac{\gamma_n}{\lambda_n - \lambda_0} \big) 
  \prod_{p \ne n} \big(1 - \frac{\gamma_n \gamma_p}{(\lambda_{p-1} - \lambda_{n-1})(\lambda_{p} - \lambda_{n})} \big) \, .
  $$
  It follows from Theorem \ref{main theorem (ii) reformulated} that $\mu_n$, $n \ge 1$, analytically extend to $U^{-s}_w$
  and are bounded on $U^{-s}_w$.
  
  Theorem \ref{main theorem (ii) reformulated} leads to estimates of the maps $\kappa_n$, $n \ge 1$, and  $\mu_n$, $n \ge 1,$
  which will be used in subsequent work to show that for any $0 \le s < 1/2$, the Birkhoff map analytically extends to a neighborhood of $H^{-s}_{r,0}$  in $H^{-s}_{c,0}$.
  \begin{proposition}\label{proposition kappa_n mu_n}
Let $0 \le s < 1/2$,  $w \in H^{-s}_{r,0}$, and let $U^{-s}_w$ be
the neighborhood of $w$ in $H^{-s}_{c,0}$ of Theorem \ref{general counting lemma}.
% (and Proposition \ref{F analytic general}). 
Then there exists $n_3 \ge n_2$ with $n_2$ as in Lemma \ref{estimate of mathcal H for general u}
so that the following holds.\\
(i) For any $u \in U^{-s}_w$ and $n \ge n_3$,
$$
| n \kappa_n(u) - 1 |  \le \frac{7}{12} e^{1/3} < 1 \, , \qquad 
| \mu_n(u) - 1 | \le \frac{7}{6} | \gamma_n(u) | e^{1/3}   \le \frac{7}{12} e^{1/3} < 1     \, .
 $$ 
 (ii) There exists a neighborhood $\tilde U^{-s}_w \subset  U^{-s}_w$ of $w$ in $H^{-s}_{c,0}$ so that for any $1 \le n < n_3$
 and any $u \in \tilde U^{-s}_w$
 $$
 | \kappa_n(u) - \kappa_n(w) | < \frac{1}{2}  \kappa_n(w) \, , \qquad 
  | \mu_n(u) - \mu_n(w) | < \frac{1}{2}  \mu_n(w) \, .
 $$
 (iii) For any $n \ge 1$, the principal  branch of the square root of $n \kappa_n$ and the one of $\mu_n$ are well defined
  on $\tilde U^{-s}_w$ and
  $$
 \big(\sqrt{n \kappa_n} \big)_{n \ge 1} :  \tilde U^{-s}_w \to \ell^\infty_+ , \
 u \mapsto   \big(\sqrt{n \kappa_n(u)} \big)_{n \ge 1}
  $$
 and
 $$
 \big(\sqrt{\mu_n} \big)_{n \ge 1} :  \tilde U^{-s}_{w} \to \ell^\infty_+ , \ 
 u \mapsto   \big(\sqrt{\mu_n(u)} \big)_{n \ge 1}
 $$
are analytic.
\end{proposition}  
\begin{proof} 
By Theorem \ref{main theorem (ii) reformulated}, there exists $n_3 \ge \max(n_2, 3)$ so that
\begin{equation}\label{choice n_3}
\sum_{p  \ne n} \frac{|\gamma_p(u)|}{\big| \lambda_p(u) - \lambda_n(u) \big|} \le \frac{1}{4}\, , \qquad \forall \, n \ge n_3 \, , \ \ \forall \, u \in U^{-s}_w\, .
\end{equation}
(i)  To simplify notation, whenever possible, we do not indicate the dependence of the quantities involved on $u$. 

First we estimate $ | \kappa_n(u) - 1 |$ for $n \ge n_3$.
We argue similarly as in the proof of Proposition \ref{estimate for kappa_n}.
For any $n \ge n_3$ and $u \in U^{-s}_w$, write $ n \kappa_n - 1 =  I_n + II_n$ where
  $$
 I_n := \big (\frac{n}{\lambda_n - \lambda_0} -1 \big)\prod_{p \ne n} \big(1 - \frac{\gamma_p}{\lambda_p - \lambda_n} \big) \, , \qquad
II_n := \prod_{p \ne n} \big(1 - \frac{\gamma_p}{\lambda_p - \lambda_n} \big) - 1 \, .
 $$
 We begin by estimating the term $I_n$.
 Since for any $n \ge n_3$,   $|\lambda_n - n| \le 1/4$ (cf. Theorem \ref{general counting lemma})
 and  $\lambda_0(w) \le 0$ ($w$ real valued) it follows that
 \begin{equation}\label{estimate term 1}
 |\lambda_n - \lambda_0| \ge  n-1/4 - \lambda_0(w) - 1/4 \ge n - 1/2 \, .
  \end{equation}
  Since by assumption  $n_3 \ge 3$, it follows that
$$
 | \frac{n}{\lambda_n - \lambda_0} - 1 |  \le \frac{1}{4}.
$$
 Furthermore, since $\log(1 +  |x|) \le |x|$ one has 
 $$
 \begin{aligned}
&| \prod_{p \ne n} \big(1 - \frac{\gamma_p}{\lambda_p - \lambda_n} \big)|   \le 
 \prod_{p \ne n} \big(1 + \frac{|\gamma_p|}{|\lambda_p - \lambda_n|} \big) \\
& \le \exp\big(  \sum_{p \ne n} \log(1 + \frac{|\gamma_p|}{|\lambda_p - \lambda_n|})\big)
\le \exp\big(  \sum_{p \ne n}  \frac{|\gamma_p|}{|\lambda_p - \lambda_n|}\big) \, .
\end{aligned}
 $$
 From \eqref{choice n_3} one then infers that altogether one has
$$
| I_n(u)  | \le \frac{1}{4} e^{1/4}\, , \qquad \forall \, n \ge n_3 \, , \quad \forall \, u \in U^{-s}_w\, .
$$
To estimate $II_n$, first note that  by \eqref{choice n_3},
 the principal branch of the logarithm of $ 1 -  \frac{\gamma_p}{\lambda_p - \lambda_n}$
 is well defined for any $p \ne n$ and $n \ge n_3$. Thus one has
 $$
  \prod_{p \ne n} \big(1 - \frac{\gamma_p}{\lambda_p - \lambda_n} \big) = 
 \exp\big(  \sum_{p \ne n} \log\big(1 - \frac{\gamma_p}{\lambda_p - \lambda_n}\big) \big) \, .
 $$
 Arguing as in the proof of Proposition \ref{estimate for kappa_n} one then infers from \eqref{choice n_3} that
 $$
 | II_n(u) | \le \frac{1}{3} e^{1/3} \, , \qquad \forall \, n \ge n_3 \, , \quad  \forall \, u \in U^{-s}_w \, .
 $$
 Combining all the estimates derived yields the desired bound,
 $$
 | n \kappa_n(u) - 1  | \le  \frac{1}{4} e^{1/4} +  \frac{1}{3} e^{1/3} = \frac{7}{12} e^{1/3} \, ,
 \qquad \forall \, n \ge n_3 \, , \quad  \forall \, u \in U^{-s}_w \, .
 $$
  Now let us turn to the estimate for $|\mu_n(u) - 1|$ for $n \ge n_3$.
 We argue as in the proof of Proposition \ref{proposition mu_n}. By the definition of $\mu_n$
 one has $\mu_n -1 = I_n + II_n$ where
 $$
 I_n:=   \big( \big( 1 - \frac{\gamma_n}{\lambda_n - \lambda_0} \big)  - 1 \big)
  \prod_{p \ne n} \big(1 - \frac{\gamma_n \gamma_p}{(\lambda_{p-1} - \lambda_{n-1})(\lambda_{p} - \lambda_{n})} \big),
  $$
  $$
  II_n:=   \prod_{p \ne n} \big(1 - \frac{\gamma_n \gamma_p}{(\lambda_{p-1} - \lambda_{n-1})(\lambda_{p} - \lambda_{n})} \big) - 1 \, . \qquad \qquad \qquad \quad
 $$
 We begin with estimating the term $I_n$.
 By \eqref{estimate term 1} one has for any $n \ge n_3,$
 $$
| \big( 1 - \frac{\gamma_n}{\lambda_n - \lambda_0} \big)  - 1|
 \le \frac{1}{n - 1/2} | \gamma_n | \le \frac{1}{2}  | \gamma_n | \, .
 $$ 
 Using that for any $p \ne n$ and $n \ge n_3$, $| \lambda_{p-1}- \lambda_{n-1} | \ge 1/2$ 
 and $|\gamma_n| \le 1/2$ (cf. Theorem \ref{general counting lemma}), 
 it then follows by \eqref{choice n_3} that
  $$
 | \prod_{p \ne n} \big(1 - \frac{\gamma_n \gamma_p}{ (\lambda_{p-1} - \lambda_{n-1})(\lambda_p - \lambda_n)} \big)|   \le
\exp\big( 2  |\gamma_n| \frac{1}{4}  \big)  \le e^{1/4} \, 
$$ 
and hence
$$
| I_n(u) | \le \frac{1}{2}\,  |\gamma_n(u)| \, e^{1/4}\, , \qquad \forall \, n \ge n_3 \, , \ \ \forall \, u \in U^{-s}_w\, .
$$
 It remains to estimate $II_n =  \prod_{p \ne n}  \big(1 - \frac{\gamma_n \gamma_p}{(\lambda_{p-1} - \lambda_{n-1})(\lambda_{p} - \lambda_{n})} \big)  - 1$.
 First note that since for any $p \ne n,$
 $$
 | \frac{\gamma_n \gamma_p}{(\lambda_{p-1} - \lambda_{n-1})(\lambda_p - \lambda_n)} | \le 2|\gamma_n|  \frac{|\gamma_p|}{|\lambda_p - \lambda_n|}  \le 1/4\, , 
 $$
 the principal branch of the logarithm of $\big(1 - \frac{\gamma_n \gamma_p}{(\lambda_{p-1} - \lambda_{n-1})(\lambda_{p} - \lambda_{n})} \big)$
 is well defined and hence one has
 $$
 \begin{aligned}
  & \prod_{p \ne n} \big( 1 - \frac{\gamma_n \gamma_p}{(\lambda_{p-1} - \lambda_{n-1})(\lambda_{p} - \lambda_{n})}  \big) \\
  & = 
 \exp\big(  \sum_{p \ne n} \log\big(1 - \frac{\gamma_n \gamma_p}{(\lambda_{p-1} - \lambda_{n-1})(\lambda_{p} - \lambda_{n})} \big) \big) \, .
 \end{aligned}
 $$
 One then concludes as in the proof of Proposition \ref{proposition mu_n} that
$$
|  II_n(u) | \le \frac{2 |\gamma_n|}{3}  \exp(\frac{2 |\gamma_n|}{3})
 \le \frac{2 |\gamma_n|}{3}  e^{1/3}\, ,  \qquad \forall \, n \ge n_3 \, , \ \ \forall \, u \in U^{-s}_w\, .
$$
 Combining the estimates derived yields
 $$
 |\mu_n(u) - 1  | \le \frac{7}{6} |\gamma_n|  e^{1/3} \le \frac{7}{12}  e^{1/3} \,  , \qquad \forall \, n \ge n_3 \, , \quad  \forall \, u \in U^{-s}_w \, .
 $$

\smallskip
\noindent
(ii) Note that for any $n \ge 1$, $\kappa_n : U^{-s}_w \to \C$ is  continuous and $\kappa_n(w) > 0$ since $w$ is assumed to be real (cf. \cite{GK}, \cite{GKT}).
Hence there exists a neighborhood $\tilde U^{-s}_w$ of $w$ so that  
$$
| \kappa_n(u) - \kappa_n(w) | < \frac{1}{2}  \kappa_n(w)\, , \qquad \forall \, 1 \le n < n_3 \, , \ \  \forall u \in \tilde U^{-s}_w \, .
$$
By shrinking $\tilde U^{-s}_w$ if needed, the same arguments as in the case of $ \kappa_n$   show that 
$$ 
| \mu_n(u) - \mu_n(w) | < \frac{1}{2}  \mu_n(w)  \, ,  \qquad \forall \, 1 \le n < n_3 \, , \ \  \forall u \in \tilde U^{-s}_w \, .
$$

\smallskip
\noindent
(iii) By item (i) and (ii), for any $n \ge 1$, the principal branch of the square roots $\sqrt{n \kappa_n(u)}$ and $\sqrt{\mu_n(u)}$
are well defined on $\tilde U^{-s}_w$. They are analytic and uniformly in $n$ bounded on $\tilde U^{-s}_w$. 
The claimed analyticity of the sequences $ \big(\sqrt{n \kappa_n} \big)_{n \ge 1}$
$\big(\sqrt{\mu_n} \big)_{n \ge 1}$  then follows from  \cite[Theorem A.3]{KP}.
\end{proof}

 %%%%%%%%%%%%%%%%%%%%%%%%%%%%%%%%%%%%%%%%%%%%%%%%%%%%%%
 %%%%%%%%%%%%%%%%%%%%%%%%%%%%%%%%%%%%%%%%%%%%%%%%%%%%%%
 
  \appendix
  
  \section{Asymptotics of eigenfunctions of $L_u$}\label{asymptotics eigenfunctions}
  
 Let  $u \in H^{-s}_{r, 0}$ with $0 \le s < 1/2$. 
 %Since $u$ is real valued, $L_u$ is self-adjoint.
 Denote by $f_n \equiv f_n(\cdot, u)$, $n \ge 0$,
 the eigenfunctions of the Lax operator $L_u$, corresponding to the eigenvalues $\lambda_n \equiv \lambda_n(u)$,
  introduced in \cite{GK} $(s = 0)$ and \cite{GKT} $(0 < s < 1/2)$. Recall that for any $n \ge 0$, $f_n$ is in $H^{1-s}_+$
  and satisfies $\|f_n \| =1$, whereas the phase of $f_n$ is determined recursively by
  \begin{equation}\label{normalization of eigenfunctions}
  \langle 1 | f_0 \rangle > 0\, , \qquad   \langle f_n | e^{ix} f_{n-1} \rangle > 0\, , \quad \forall \, n \ge 1 \,  .
  \end{equation}

  The aim of this appendix is to prove asymptotic estimates of  
  \begin{equation}\label{def g_n}
  g_n(x, u) := e^{-inx} f_n(x, u) \in H^{1-s}_c 
  \end{equation}
  as $n \to \infty$. 
  
  \begin{proposition}\label{asymptotics g_n}
  For any $0 \le s < 1/2$ the following holds:\\
  (i) For any $n \ge 0$, the map 
  $$
  g_n: H^{-s}_{r, 0} \to H^{1-s}_c, \ u \mapsto g_n(\cdot, u)
  $$
   is weakly sequentially continuous.\\
  (ii) For any $u \in H^{-s}_{r,0}$, the sequence $(g_n(\cdot, u))_{n \ge 0}$ converges in $H^{1-s}_c$ 
  and the limit $ g_{\infty} \equiv g_{\infty}(\cdot, u)$ is given by 
  $$
  g_{\infty} = e^{D^{-1}u} \, , \qquad D^{-1}u(x) =  i \partial_x^{-1} u(x) = \sum_{k \ne 0} \frac{\widehat u(k)}{k} e^{ikx} \,  .
  $$
  In particular, $| g_{\infty}(x) | = 1$ for any $x \in \T$ and $\| g_\infty \| = 1$.\\
  (iii) For any $u \in H^{-s}_{r,0}$,  $\sum_{k \ge 1} \| g_k(\cdot, u) - g_{k-1}(\cdot, u) \| < \infty $ and
  $$
 \sum_{k=n}^\infty \| g_k(\cdot, u) - g_{k-1}(\cdot, u) \| 
 \le \sqrt{2} \sum_{k = n}^\infty \sqrt[+]{1 - \mu_k}
 \le C \frac{1}{\langle n \rangle ^{\tau}}\, , \quad \forall \, n \ge 1\,  ,
  $$
  where $\tau:= (1/2 - s)/ 2$ and where the constant $C > 0$ can be chosen locally uniformly with respect to $u \in H^{-s}_{r,0}$.
  \end{proposition}
  \begin{remark}
  For $u = 0,$ one has $f_n(x) = e^{inx}$ and hence $g_n \equiv 1$ for any $n \ge 0$. 
  \end{remark}
  
  The proof of Proposition \ref{asymptotics g_n} is based on the following lemmas.
  First we compute $\lim_{n \to \infty} g_n$ for a finite gap potential. Recall
  from \cite{GK} that for any integer $N \ge 1$, we denote by $\mathcal O_N$ the subset of functions $u \in H^0_{r,0}$ with
  $$
  \gamma_n(u) > 0\, , \quad \forall 1 \le n \le N, \qquad \gamma_n(u) = 0 ,  \quad  \forall \, n \ge N+1 \, .
  $$
  We remark that since $\Phi: L^2_{r,0} \to h^{1/2 }_+$ is homeomorphism (cf. \cite{GK}),
  $\bigcup_{N \ge 1} \mathcal O_N$ is dense in $L^2_{r,0}$ and hence in $H^{-s}_{r,0}$ for any $-1/2 < s \le 0$.
  \begin{lemma}\label{g_infty finite gap}
  For any integer $N \ge 1$ and any $u \in \mathcal O_N$, one has in $H^1_c$
  $$
  \lim_{n \to \infty} g_n(\cdot, u) = e^{D^{-1}u} \, .
  $$
  \end{lemma}
  \begin{proof}
  We use the notation and the results established in \cite{GK}. For any $u \in \mathcal O_N$ and $n \ge N,$
  the eigenfunction $f_n \equiv f_n(\cdot, u)$ of $L_u$, corresponding to the eigenvalue $\lambda_n \equiv \lambda_n(u)$,
  is given by $f_n(x) = e^{i(n-N)x} f_N(x)$ and therefore, $g_n \equiv g_n(\cdot, u)$ satisfies
  $$
  g_n(x) = e^{-inx}f_n(x) = g_N(x) , \quad \forall \, n \ge N .
  $$
 As a consequence, the sequence $(g_n)_{n \ge 1}$ converges and its limit $g_\infty$ is given by
 \begin{equation}\label{formula 1 limit}
 g_{\infty}(x) = e^{-iNx} f_N(x) \, .
 \end{equation}
 By \cite[Theorem 7.1]{GK} and its proof, $\Pi u(x)$ is of the form 
 \begin{equation}\label{formula 1 u}
 \Pi u(x) = - e^{ix}\frac{Q'(e^{ix})}{Q(e^{ix})}
 \end{equation}
 where $Q(z)$ is a polynomial in $z \in \C$ of degree $N$ of the form
 $$
 Q(z) = \prod_{j=0}^{N-1}(1 - q_j z) , \qquad q_j \in \C, \quad 0 < |q_j| < 1 \, ,
 $$
 and $Q'(z) = \partial_z Q(z)$. In addition, $f_N$ is given by
 \begin{equation}\label{formula 1 f_N}
 f_N(x) = a_N e^{iNx} \overline{Q(e^{ix})} / Q(e^{ix})
 \end{equation}
 where $a_N \in \C\setminus {0}$ is constant. Since 
 $$
 D \log Q(e^{ix}) = \frac{1}{Q(e^{ix})} Q'(e^{ix}) e^{ix} = - \Pi u(x)\, ,
 $$
 one concludes that $D^{-1} \Pi u(x) = \log(1/Q(e^{ix}))$ and hence
 $$
 e^{D^{-1} \Pi u(x) } = \frac{1}{Q(e^{ix})} \, .
 $$
 Taking the complex conjugate of the latter identity one obtains 
 $$
 e^{- D^{-1} \overline{\Pi u(x)} } = \frac{1}{ \overline{Q(e^{ix})}} \, ,  \qquad \text{or} \qquad
 e^{D^{-1} \overline{\Pi u(x)} } = \overline{Q(e^{ix})} \, .
 $$
 Combining the latter two formulas one concludes that
 $$
 e^{D^{-1} u(x) } = \frac{ \overline{Q(e^{ix})}}{Q(e^{ix})} \, .
 $$
 The identity \eqref{formula 1 f_N} then becomes
 $$
 f_N(x) = a_N e^{iNx} e^{D^{-1} u(x) } \, .
 $$
  Since $u$ is assumed to be real valued it then follows by the normalization condition $\| f_N \| = 1$
  that $ |a_N| = 1$.  Hence \eqref{formula 1 limit} becomes
  \begin{equation}\label{formula 2 limit}
 g_{\infty}(x) = a_N e^{D^{-1} u(x) } \ , \qquad   | a_N | = 1 \, .
  \end{equation}
 It remains to show that $a_N = 1$. To this end we compute $\langle 1 | f_N \rangle$. Since by \eqref{formula 1 f_N}
  $$
  f_N(x) = a_N \prod_{j=0}^{N-1} \frac{e^{ix} - \overline{q_j}}{1 - q_je^{ix}}
  =  a_N \prod_{j=0}^{N-1} \big( (e^{ix} - \overline{q_j}) \sum_{n=0}^{\infty} (q_je^{ix})^n \big) \, ,
  $$
  one concludes that $ f_N(x) = a_N(-1)^N \prod_{j=0}^{N-1} \overline{q_j}  +  e^{ix}\tilde f_N(x)$ and hence
  \begin{equation}\label{formula 2 for 1 f_N}
  \langle 1 | f_N \rangle = \overline{a_N} (-1)^N \prod_{j=0}^{N-1} q_j \, ,  \qquad \text{or}  \qquad
  a_N \langle 1 | f_N \rangle =  (-1)^N \prod_{j=0}^{N-1} q_j \, .
  \end{equation}
  On the other hand, as $z \to \infty,$
  $$
  \frac{Q(z)}{z^N} = \prod_{j=0}^{N-1} \frac{1 - q_j z}{z} = (-1)^N  \prod_{j=0}^{N-1} q_j + O(\frac{1}{z}) \, .
  $$
 Alternatively, the asymptotics of $Q(z)/z^N$ can be computed using \cite[(7.6), (7.7)]{GK},
  $$
    \frac{Q(z)}{z^N} =  \frac{\det(\text{Id} - z M_{N-1})}{z^N} = (-1)^N \det(M_{N-1}) + O(\frac{1}{z})
   $$
   and hence
     \begin{equation}\label{asymptotics 1 Q/ z^N}
      (-1)^N \prod_{j=0}^{N-1} q_j = (-1)^N \det(M_{N-1}) \, .
     \end{equation}   
  Here $M_{N-1}$ is the $N \times N$ matrix $(M_{np})_{0 \le n, p \le N-1}$ with (cf. \cite[(7.8)]{GK})
  $$
  M_{np} = \sqrt[+]{\mu_{n+1}} \gamma_{n+1} \frac{ \langle f_p | 1 \rangle } {(\lambda_p - \lambda_n -1)  \langle f_{n+1} | 1 \rangle } 
  $$
  where we used that $u \in \mathcal O_N$ and hence $ \langle f_{n+1} | 1 \rangle \ne 0$ for any $0 \le n \le N-1$. Clearly,
  $$
  \det(M_{N-1}) = \big( \prod_{n=1}^N  \sqrt[+]{\mu_{n}} \gamma_{n} \big)  \frac{ \langle f_0 | 1 \rangle }{ \langle f_N | 1 \rangle }
  \det\big(  (\frac{1} {\lambda_p - \lambda_n -1} )_{0 \le n, p \le N-1} \big) \, .
  $$
  Since by the formula for the Cauchy determinant,
  $$
   \det\big(  (\frac{1} {\lambda_p - (\lambda_n +1)} )_{0 \le n, p \le N-1} \big) 
   = \frac{(-1)^{N(N-1)/2}\prod_{0 \le p < q \le N-1} (\lambda_p - \lambda_q)^2}{\prod_{0 \le n, p \le N-1} (\lambda_p - (\lambda_n +1))}
  $$
  and since 
  $$
  \prod_{0 \le n, p \le N-1} (\lambda_p - (\lambda_n +1)) = (-1)^{N(N+1)/2} \prod_{0 \le n, p \le N-1} |\lambda_p - (\lambda_n +1)|
  $$
  it follows together with $\overline{\langle f_N | 1 \rangle} = \langle 1 | f_N  \rangle$ that 
  $ (-1)^N \det M_{N-1} = c_N  \langle 1 | f_N  \rangle$ where
  \begin{equation}\label{formula det M_(N-1)}
  c_N = \big( \prod_{n=1}^N  \sqrt[+]{\mu_{n}} \gamma_{n} \big)  \frac{ \langle f_0 | 1 \rangle }{ |\langle f_N | 1 \rangle |^2 }
  \frac{\prod_{0 \le p < q \le N-1} (\lambda_p - \lambda_q)^2}{\prod_{0 \le n, p \le N-1} |\lambda_p - (\lambda_n +1)|} > 0 \, .
  \end{equation}
 Combining \eqref{formula 2 for 1 f_N} - \eqref{formula det M_(N-1)}  one concludes $a_N = c_N > 0$. Since $| a_N | =1$
 this shows that $a_N = 1$.
   \end{proof} 
  For estimating $\| g_k - g_{k-1} \|$  we first need to study the normalizing constants $\mu_n \equiv \mu_n(u)$.
  By \cite{GK} $(s=0)$ and \cite{GKT} $(0 < s < 1/2)$, $\mu_n$, $n \ge 1,$ is given by
  $$
  \mu_n = (1 - \frac{\gamma_n}{\lambda_n - \lambda_0}) \prod_{p \ne n} \frac{1 -  \frac{\gamma_p}{\lambda_p - \lambda_n}}{ 1 -  \frac{\gamma_p}{\lambda_p - \lambda_{n-1} - 1}} \, .
  $$
 Using that by the definition of $\gamma_{n+1}$, $\lambda_{n+1} =  \lambda_n + 1 + \gamma_{n+1}$,  one obtains by algebraic transformations, 
 $ \mu_n = (1 - \frac{\gamma_n}{\lambda_n - \lambda_0})  \prod_{p \ne n} (1- b_{np})$ where 
 $$
   b_{np} := \gamma_n \frac{\gamma_p}{(\lambda_p - \lambda_n) (\lambda_{p-1} - \lambda_{n-1})} \, , \quad \forall \, p \ne n \, .
 $$
 \begin{lemma}\label{estimate 1 for mu_n}
 Let $0 \le s < 1/2$.
 For any $u \in H^{-s}_{r,0}$, $0 < \mu_n \le 1$ and
 $$
 0 \le 1 - \mu_n \le \frac{\gamma_n}{n} + \gamma_n\sum_{p \ne n} \frac{\gamma_p}{(p-n)^2} (1+ \gamma_n) (1+ \gamma_p) \, , \quad \forall \, n \ge 1 \, .
 $$
 \end{lemma}  
  \begin{proof}
  Since $\lambda_n = n - \sum_{k = n+1}^\infty \gamma_k,$ $n \ge 0$, one has 
  \begin{equation}\label{estimate 2 for mu_n}
  0 \le \frac{\gamma_n}{\lambda_n - \lambda_0} = \frac{\gamma_n}{n + \sum_{k=1}^n \gamma_k}  
  \le \frac{\gamma_n}{1 + \gamma_n } < 1 \, .
  \end{equation}
Furthermore, for any $p \ge n+1,$
 $$
 \begin{aligned}
 (\lambda_p - \lambda_n) (\lambda_{p-1} - \lambda_{n-1}) &=  \big( (p-n) + \sum_{k=n+1}^{p}\gamma_k \big)  \big( (p-n) + \sum_{k=n}^{p-1}\gamma_k \big) \\
 & \ge  ((p-n) + \gamma_p) ((p-n) + \gamma_n) \, .
 \end{aligned}
 $$
 Similarly, for $1 \le p \le n-1$, exchanging the role of $n$ and $p$,
$$
 (\lambda_p - \lambda_n) (\lambda_{p-1} - \lambda_{n-1}) \ge  ((n-p) + \gamma_n) ((n-p) + \gamma_p)\, .
$$
Hence for any $p \ne n$
\begin{equation}\label{estimate 3 for mu_n}
b_{np} \le \frac{\gamma_n \gamma_p}{(|n-p| + \gamma_p)(|n-p| + \gamma_n)}
\le  
 \frac{\gamma_n \gamma_p}{(1 + \gamma_p)(1 + \gamma_n)} < 1 \, .
\end{equation}
Altogether this proves that
\begin{equation}\label{estimate 4 for mu_n}
0 < \prod_{p \ne n}(1 - b_{np}) \le 1 \, .
\end{equation}
 To estimate $1 - \mu_n$, write it as a sum $I_n + II_n$ where
 $$
 I_n := \big( 1 - (1 - \frac{\gamma_n}{\lambda_n - \lambda_0}) \big) \prod_{p \ne n} (1 - b_{np})
% = a_n  \prod_{p \ne n} (1 - b_{np}) 
 \ge 0 \, , 
 $$  
  $$
  II_n :=  1 -  \prod_{p \ne n} (1 - b_{np}) \ge 0\, .
  $$ 
  By \eqref{estimate 2 for mu_n} and \eqref{estimate 4 for mu_n} it then follows that
  $$
  0 \le I_n \le \frac{\gamma_n}{\lambda_n - \lambda_0} \le \frac{\gamma_n}{n} \, , \qquad \forall \, n \ge 1 \,  .
  $$
  To estimate $II_n$ note that 
  $$
  0 \le II_n = 1 - \exp\big( - \sum_{p \ne n} - \log(1 - b_{np}) \big) .
  $$
  Since by the mean value theorem, for any $x \ge 0$,
  $1 - e^{-x } = \int_0^x e^{-t} d t \le x $ one infers 
  $$
 0 \le  II_n  \le \sum_{p \ne n} - \log(1 - b_{np}) \, .
  $$
  Similarly, for $0 \le y \le 1$, one has $- \log(1-y) = \int_0^y \frac{1}{1-t} d t \le y/(1-y)$,
  implying together with \eqref{estimate 3 for mu_n} that
  $$
   0 \le  II_n  \le \sum_{p \ne n} \frac{b_{np}}{1 - b_{np} } \le \gamma_n  \sum_{p \ne n} \frac{\gamma_{p}}{(p-n)^2 } \frac{1}{1 - b_{np} }
  $$ 
  and
  $$
   \frac{1}{1 - b_{np} } \le \frac{(1 + \gamma_n)(1 + \gamma_p)}{1 + \gamma_n + \gamma_p}
   \le (1 + \gamma_n)(1 + \gamma_p) \, .
  $$
  Therefore,
  $$
 0 \le  II_n  \le \gamma_n \sum_{p \ne n} \frac{\gamma_{p}}{(p-n)^2 } (1 + \gamma_n)(1 + \gamma_p) \, . 
  $$
  \end{proof}  
  Lemma \ref{estimate 1 for mu_n} is used to prove the following
  \begin{lemma}\label{estimate 1 of g_k - g_(k-1)}
  Let $0 \le s < 1/2.$ For any $u \in H^{-s}_{r,0},$  $(g_k)_{k \ge 0}$ is a Cauchy sequence in $ L^2_c.$
  More precisely, one has for any $n \ge 1,$
  \begin{equation}\label{estimate 2 of g_k - g_(k-1)}
  \sum_{k = n}^\infty \| g_k - g_{k-1} \|  \le   \sum_{k = n}^\infty \sqrt[+]{1 - \mu_k} \le C \frac{1}{n^\tau}
  \end{equation}
  where $\tau = (\frac{1}{2} - s)/2$ and where the constant $C > 0$ can be chosen locally uniformly with respect to $u \in H^{-s}_{r,0}$.
  As  a consequence, the limit $g_\infty$ of $(g_k)_{k \ge 0}$ in $L^2_c$ is given by the telescoping sum,
  \begin{equation}\label{estimate 3 of g_k - g_(k-1)}
  g_\infty = g_0 + \sum_{k=1}^\infty (g_k - g_{k-1})
  \end{equation}
  which is absolutely convergent in $L^2_c$.  Furthermore,  $\|g_\infty \| =1$.
  \end{lemma}
  \begin{proof}
 We write $g_n$ as a telescoping sum, $g_n = g_0 + \sum_{k=1}^n (g_k - g_{k-1})$ 
 and note that 
 $$
 \langle g_k | g_{k-1} \rangle = \langle e^{ikx} g_k \ | \ e^{ix} e^{i(k-1)x}g_{k-1} \rangle
 = \langle f_k | e^{ix} f_{k-1} \rangle = \sqrt[+]{\mu_k} > 0\, .
 $$
  Since $\| g_k \| = \| e^{ikx}g_k \| = \| f_k \| = 1$, it then follows that
  $$
  \| g_k - g_{k-1} \|^2 =  \| g_k \| ^2 - 2  \langle g_k | g_{k-1} \rangle    + \| g_{k-1} \|^2 
  = 2(1 - \sqrt[+]{\mu_k}) \le 2(1 - \mu_k) \, .
  $$
  By Lemma \ref{estimate 1 for mu_n}, $ \| g_k - g_{k-1} \| \le \sqrt[+]{2} \sqrt[+]{1 - \mu_k}$
  can be estimated as
  \begin{equation}\label{estimate 4 of g_k - g_(k-1)} 
 \| g_k - g_{k-1} \| 
 \le  \sqrt[+]{2 \gamma_k / k} +  \
 \sqrt[+]{2 \gamma_k } \Big( \sum_{p \ne k} \frac{\gamma_p}{(p-k)^2}(1 + \gamma_k)(1 + \gamma_p) \Big)^{1/2} \, . 
 \end{equation}
  By \cite{GK} $(s = 0)$ and \cite{GKT} ($0 < s < 1/2$), $(\sqrt[+]{\gamma_k})_{k \ge 1}$ is in $h_+^{\frac{1}{2}- s}$. Hence 
  with $1-s = \frac{1}{2} + 2\tau$
  \begin{equation}\label{estimate 5 of g_k - g_(k-1)} 
   \sqrt[+]{\gamma_k / k} = k^{\frac{1}{2} - s} \sqrt[+]{ \gamma_k } \frac{1}{k^{1-s}} 
   = k^{\frac{1}{2} - s} \sqrt[+]{ \gamma_k } \frac{1}{k^{\frac{1}{2} +\tau}}  \frac{1}{k^\tau}
   =  \frac{1}{k^\tau} \ell_k^1 \, .
  \end{equation}
  Similarly, since
   \begin{equation}\label{estimate 6 of g_k - g_(k-1)} 
   \sqrt[+]{ \gamma_k } \sum_{p \ne k} \frac{\gamma_p}{(p-k)^2} 
   \le   \frac{1}{k^{2\tau}}  (k^{\frac{1}{2} - s} \sqrt[+]{ \gamma_k })  \big(\sum_{p \ne k} \frac{\gamma_p}{(p-k)^2}\big)^{1/2}
   =   \frac{1}{k^{2\tau}} \ell_k^1 \, .
  \end{equation}
Combining  \eqref{estimate 4 of g_k - g_(k-1)}  -- \eqref{estimate 6 of g_k - g_(k-1)}  yields the claimed estimate \eqref{estimate 2 of g_k - g_(k-1)}.
Since $\| g_k \| =1 $ for any $k \ge 0$, the limit $g_\infty$ also satisfies $\|g_\infty \| = 1$.
   \end{proof}
   We further investigate the sequence $(g_n)_{n \ge 0}$.
   \begin{lemma}\label{regularity 1 of g_n}
   For any $u \in H^{-s}_{r, 0}$ with $0 \le s < 1/2,$ $(g_n)_{n \ge 0}$ is a bounded sequence in $H^{1-s}_c$.
   \end{lemma}
   \begin{proof}
   Let  $u$ be an element in $H^{-s}_{r, 0}$ with $0 \le s < 1/2$.
   Recall that $L_u = D - T_u$ is a bounded operator $H^{1-s}_+ \to H^{-s}_+$
   (cf. \cite{GK} ($s=0$) and \cite{GKT} ($0 < s < 1/2$)). One has for any $n \ge 0$,
   $$
   D e^{inx}g_n = n e^{inx} g_n + e^{inx} Dg_n
   $$
   and 
   $$
   T_u(e^{inx}g_n) = \Pi(u e^{inx}g_n) = u e^{inx}g_n - (\text{Id} -\Pi)(u e^{inx}g_n)
   $$
   where $u e^{inx}g_n$ denotes the element in $H^{-s}_c$, given by the bounded linear functional
   $$
   H^s_c \to \C\ , \ v \mapsto \langle u \, , \ v e^{inx}g_n \rangle
   $$
   and $\langle \cdot, \cdot \rangle$ denotes the dual pairing between $H^{-s}_c$ and $H^s_c$. Here we used that
   $$
   H^{s}_c \times H^{1-s}_c \to H^s_c, \, (v, f) \mapsto vf
   $$
   is a bounded bilinear map. Multiplying both sides of the identity $L_u f_n = \lambda_n f_n$ by $e^{-inx}$,
   the above computations lead to the following identity in $H^{-s}_c$,
   \begin{equation}\label{regularity 2 of g_n}
   Dg_n = (\lambda_n - n) g_n + ug_n - e^{-inx}(\text{Id} -  \Pi)(u e^{inx}g_n) \, .
   \end{equation}
   For any given $m \in \Z$, denote by $\Pi_{\ge m}$ the $L^2-$orthogonal projection on $H^s_c$, $s \in \R$, defined by
   $$
   \Pi_{\ge m} :  \sum_{k \in \Z} \widehat f(k) e^{ikx} \mapsto   \sum_{k \ge m} \widehat f(k) e^{ikx}\,  , 
   $$
   and let $ \Pi_{\le m} := \text{Id} -  \Pi_{\ge m + 1} $. One then has
   $$
    \Pi_{\ge -n }  \ g_n = g_n\ , \qquad   \Pi_{\ge -n }  \ D g_n = D g_n \, ,
   $$
   and
   $$
    \Pi_{\ge -n }  \big( e^{-inx} (\text{Id} - \Pi) (u e^{inx} g_n)  \big)  = 0 \, .
   $$
   Applying $\Pi_{\ge -n}$ to both sides of \eqref{regularity 2 of g_n} thus yields
   \begin{equation}\label{regularity 3 of g_n}
    Dg_n = (\lambda_n - n) g_n + \Pi_{\ge -n} (ug_n) \, ,
   \end{equation}
   implying the estimate
   $$
    \| Dg_n  \|_{-s} \le  |\lambda_n - n| \|g_n\|_{-s} + \| \Pi_{\ge -n} (ug_n) \|_{-s} \, .
   $$
   Using that $\|g_n\|_{-s} \le \|g_n\| = 1$, $n - \lambda_n = \sum_{k = n+1}^\infty \gamma_k,$  
   and that by the definition of the norm $\| \cdot \|_{-s},$
   $$
   \| g_n \|_{1-s} \le \| Dg_n \|_{-s} + | \langle g_n | 1 \rangle | \le \| Dg_n \|_{-s} + 1\, ,
   $$
   one then obtains 
   \begin{equation}\label{regularity 4 of g_n}
    \| g_n \|_{1-s} \le 1 +  \sum_{k = 1}^\infty \gamma_k + \|  \Pi_{\ge -n} (ug_n) \|_{-s}  \, .
   \end{equation}
   Since $H^{-s}_{r, 0} \to \ell_+^{1, 1-2s} , \  u \mapsto (\gamma_n(u))_{n \ge 1}$ is a continuous map and
   $$
     \|  \Pi_{\ge -n} (ug_n) \|_{-s}  \le \| ug_n \|_{-s} \,  , 
  $$
   it remains to estimate $ \| ug_n \|_{-s}$.
   Recall that $\tau = (\frac{1}{2} - s)/2$. Then there exists a constant $C_1 \ge 1$ so that
   $$
    \| ug_n \|_{-s} \le C_1  \| u \|_{-s}  \| g_n \|_{1 - s - \tau}\ , \quad \forall n \ge 0\,  .
   $$
   Using that $\|g_n\| = 1$ it follows by interpolation between $L^2_c$ and $H^{1-s}_c$ that for some constant $C_2 \ge 1$,
   $$
   \| g_n \|_{1 - s - \tau} \le C_2 \| g_n \|_{1 - s }^{1 - \frac{\tau}{1-s}} \,  .
   $$
   By Young's inequality one then infers
   $$
   C_1 \| u \|_{-s}  C_2 \| g_n \|_{1 - s}^{1 - \frac{\tau}{1-s}} \le \frac{1}{q} \big(C_1   \| u \|_{-s}  C_2 \big)^q + (1 - \frac{1}{q} ) \| g_n \|_{1 - s }
   $$
   where $\frac{1}{q} =  \frac{\tau}{1-s}$ . Hence \eqref{regularity 4 of g_n} yields
   $$
   \frac{1}{q}   \| g_n \|_{1-s} \le 1 + \sum_{k=1}^\infty \gamma_k + \frac{1}{q} \big(  C_1   \| u \|_{-s}  C_2   \big)^{q} 
   $$
   or
   $$
    \| g_n \|_{1-s} \le \frac{1-s}{\tau} + \frac{1-s}{\tau}\sum_{k=1}^\infty \gamma_k +     \big(  C_1   \| u \|_{-s}  C_2   \big)^{\frac{1-s}{\tau} } \, , 
    $$
    showing that $(g_n)_{n \ge 0}$ is bounded in $H^{1-s}_c$.
   \end{proof}
   Lemma \ref{estimate 1 of g_k - g_(k-1)} and Lemma \ref{regularity 1 of g_n} lead to the following
   \begin{corollary}\label{Lemma limit 1 of g_n}
   For any $u \in H^{-s}_{r,0}$ with $0 \le s < 1/2$, the sequence $(g_n)_{n \ge 0}$ converges strongly in $H^{1-s}_c$ to $g_\infty$.
   In particular, $g_\infty$ is in $H^{1-s}_c$. It is given by
   $$
   g_\infty = a e^{D^{-1}u}\,  , \qquad a \in \C , \, \, |a| = 1 \,  .
   $$
   \end{corollary}
   \begin{proof}
   By Lemma \ref{regularity 1 of g_n}, $(g_n)_{n \ge 0}$ is bounded in $H^{1-s}_c$.
   Since by Lemma \ref{estimate 1 of g_k - g_(k-1)}, $g_n \to g_\infty$ in $L^2_c,$ it follows that $g_\infty$ is in $H^{1-s}_c$
   and that $g_n \rightharpoonup g_\infty$ weakly in $H^{1-s}_c$.  Therefore  $g_n \to g_\infty$ strongly in
   $H^{1-s - \tau}_c$ and hence $ug_n \to ug_\infty$ strongly in $H^{-s}_c$. Since
   $$
   \begin{aligned}
      \|  \Pi_{\le -n -1} (ug_n) \|_{-s} &  \le  \|  \Pi_{\le -n -1} (ug_n - ug_\infty) \|_{-s} + \|  \Pi_{\le -n -1} (ug_\infty) \|_{-s}  \\
      & \le \| ug_n - ug_\infty \|_{-s} +  \|  \Pi_{\le -n -1} (ug_\infty) \|_{-s}
   \end{aligned}
   $$
   and $\|  \Pi_{\le -n -1} (ug_\infty) \|_{-s} \to 0$ as $n \to \infty$, one then concludes that
   $$
   \lim_{n \to \infty}    \Pi_{\ge -n } (ug_n) = ug_\infty
   $$
   strongly in $H^{-s}_c$. Since $\lim_{n \to \infty} (\lambda_n - n) = 0$ it then follows from \eqref{regularity 3 of g_n} that
   $(Dg_n)_{n \ge 0}$ converges strongly to $Dg_\infty$ in $H^{-s}_c$  and that
   \begin{equation}\label{limit 1 of g_n}
   Dg_\infty = u g_\infty \, .
   \end{equation}
   It implies that $g_\infty = a e^{D^{-1}u}$ where $a \in \C$. Since $\|g_\infty\| = 1$ by Lemma \ref{estimate 1 of g_k - g_(k-1)}
   one has $|a| =1$.
   \end{proof}
   We have now all the ingredients to prove Proposition \ref{asymptotics g_n}.
   
   \smallskip
   \noindent
   {\em Proof of Proposition \ref{asymptotics g_n}.}
   Let $0 \le s < 1/2.$  First we prove item (i).
   By \cite[Lemma 8]{GKT} it follows that for any $N \ge 1$, 
   $\sup_{0 \le n \le N}  \| f_n(\cdot, u) \|_{1-s}$ is bounded on bounded subsets of potentials $u \in H^{-s}_{r,0}$. Choose any sequence
   $(u_k)_{k \ge 1}$ in $H^{-s}_{r, 0}$ which converges weakly to an element $u$ in $H^{-s}_{r,0}$. Then
   $$
   \sup_{k \ge 1, 0 \le n \le N}  \| f_n(\cdot, u_k) \|_{1-s} < \infty \ . 
   $$
   Hence there exists a subsequence $(u_{k_j})_{j \ge 1}$ of  $(u_{k})_{k \ge 1}$ so that
   for any $0 \le n \le N,$ $f_n(\cdot, u_{k_j}) \rightharpoonup h_n$ in $H_+^{1-s}$.
   By \cite[Lemma 7]{GKT}, $\lim_{j \to \infty} \lambda_n(u_{k_j}) = \lambda_n(u)$ for any $0 \le n \le N$.
   Since $f_n(\cdot, u_{k_j}) \to h_n$ strongly in $H_+^{1-s - \tau}$, where $\tau (\frac{1}{2} - s)/2$,
   it follows that $u_{k_j} f_n(\cdot, u_{k_j}) \rightharpoonup uh_n$ weakly in $H^{-s}_c$, implying that
   $$
   L_u h_n = \lambda_n(u) h_n \, , \qquad \forall \, 0 \le n \le N \, .   
   $$
   Since $f_n(\cdot, u_{k_j}) \to h_n$ strongly in $L^2_+,$ it follows from the normalization conditions of $f_n$ (cf. \eqref{normalization of eigenfunctions})
   that $\|h_n\| = 1$ for any $0 \le n \le N$ and
   $$
  \langle 1 | h_0 \rangle > 0\, , \qquad   \langle h_n | e^{ix} h_{n-1} \rangle > 0\,  , \quad \forall  \, 1 \le n \le N \,  .
  $$
  Hence $h_n = f_n(\cdot, u)$ for any $0 \le n \le N$. Since the subsequence  
  $(u_{k_j})_{j \ge 1}$ with the property that for any $0 \le n \le N,$ $f_n(\cdot, u_{k_j}) \rightharpoonup h_n$ in $H_+^{1-s}$
  was chosen arbitrarily, it follows that in $H_+^{1-s}$,
  $$
  f_n(\cdot, u_{k}) \rightharpoonup f_n(\cdot, u)\,  , \qquad  \forall \, 0 \le n \le N \,  ,
  $$
  and in turn that
  $$
   g_n(\cdot, u_{k}) \rightharpoonup g_n(\cdot, u)\,  , \qquad  \forall \, 0 \le n \le N \,  .
  $$
  This establishes claim (i). Item (iii) follows from Lemma \ref{estimate 1 of g_k - g_(k-1)}  and item (ii) from Corollary \ref{Lemma limit 1 of g_n}
  except for the phase of $a$. It remains to prove that $a$ is a real, positive number. By Lemma \ref{g_infty finite gap},
  $a = 1$ for any potential $u \in \bigcup_{N \ge 1} \mathcal O_N$. Since $ \bigcup_{N \ge 1} \mathcal O_N$ is dense in $H^{-s}_{r,0}$,
  it suffices to show that $a \equiv a(u)$ depends continuously on $u \in H^{-s}_{r,0}$, or equivalently that the map
  $$
  H^{-s}_{r,0} \to L^2_c\, , \, u \mapsto g_\infty(\cdot, u)
  $$
  is continuous.  This can be seen as follows. Let $(u_k)_{k \ge 1}$ be any sequence in $H^{-s}_{r, 0}$ which converges in $ H^{-s}_{r,0}$.
  Denote its limit by $u$. Then  by the triangle inequality, $\| g_\infty(\cdot, u_k ) - g_\infty(\cdot, u) \|$ is bounded by
  $$
 \| g_\infty(\cdot, u_k ) - g_n(\cdot, u_k) \| +
   \| g_n(\cdot, u_k ) - g_n(\cdot, u) \| +  \| g_n(\cdot, u ) - g_\infty(\cdot, u) \| \,  .
  $$
  By \eqref{estimate 2 of g_k - g_(k-1)} in Lemma \ref{estimate 1 of g_k - g_(k-1)},
  for any given $\varepsilon > 0$, there exists  $n_0 \ge 1$ so that 
  $$
   \| g_\infty(\cdot, u_k ) - g_n(\cdot, u_k) \| \le \e \,  , \qquad  \, \forall \, n \ge n_0, \quad \forall \, k \ge 1  \
  $$ 
  and
  $$
   \| g_n(\cdot, u ) - g_\infty(\cdot, u) \| \le \e \ , \qquad \forall \, n \ge n_0 \, .
  $$
  Since by item (i), $\lim_{k \to \infty} \| g_{n_0}(\cdot, u_k ) - g_{n_0}(\cdot, u) \| = 0$ 
  there exists $k_\e \ge 1$ so that
  $$
   \| g_{n_0}(\cdot, u_k ) - g_{n_0}(\cdot, u) \|  \le  \e \,  , \qquad \forall \, k \ge k_{\e} \, .
  $$
  Altogether we have proved that
  $$
  \| g_\infty(\cdot, u_k ) - g_\infty(\cdot, u) \| \le 3\e \, , \qquad \forall \, k \ge k_{\e} \, .
  $$
  Since $\e > 0$ is arbitrary, one has $\lim_{k \to \infty} g_\infty( \cdot, u_k) = g(\cdot, u)$ in $L^2_c$.
   \hfill $\square$
   
  %%%%%%%%%%%%%%%%%%%%%%%%%%%%%%%%%%%%%%%%%%%%%%%%%%%%%%%%%%%%%%%%%%%%%%%%%%%%%%%%%%%
  %%%%%%%%%%%%%%%%%%%%%%%%%%%%%%%%%%%%%%%%%%%%%%%%%%%%%%%%%%%%%%%%%%%%%%%%%%%%%%%%%%%
  
 \section{The notion of normally analytic maps}\label{normally analytic}
In this appendix, we briefly review the notion of normally analytic maps. We restrict our attention
to the setup encountered in Section \ref{Quasi-moment map}. Without any further reference, we use
the notation established in the paper.
Let $0 \le s < 1/2$, $r > 0$ and assume that $f : B^{-s}_{c,0}(r) \to \C$ is an analytic function given by its Taylor series
$f(u) = \sum_{k = 0}^\infty f^k(u)$ where
$$
f^k(u) = \sum_{|\alpha | = k} f^k_\alpha \widehat u^\alpha , \qquad f^k_\alpha \in \C, \ \ \alpha \in \N_0^\Z.
$$
Here we used the multi-index notation $\widehat u^\alpha = \prod_{n \in \Z} \widehat u(n)^{\alpha_n}$ with 
$\alpha = (\alpha_n)_{n \in \Z} \in \N_0^\Z$ and $|\alpha | = \sum_{n \in \Z} \alpha_n$.
\begin{definition}
The map $f: B^{-s}_{c,0}(r) \to \C$ is said to be normally analytic if $\underline f : B^{-s}_{c,0}(r) \to \C$ is well defined and analytic where
$$
\underline f(u) := \sum_{k=0}^\infty \underline f^k(u)\, , \qquad 
 \underline f^k(u) := \sum_{|\alpha| = k} |f^k_\alpha | \widehat u^\alpha\, .
$$
%locally uniformly converges on $B^{-s}_{c,0}(r)$.\\
Similarly, we say that a map $F= (F_k)_{k \ge 1} : B^{-s}_{c,0}(r) \to \ell^{1, 2 - 2s}(\N, \C)$
is normally analytic if for any $k \ge 1$, $F_k : B^{-s}_{c,0}(r) \to \C$ is normally analytic and 
$\underline F := (\underline{F_k})_{k \ge 1}  : B^{-s}_{c,0}(r) \to \ell^{1, 2 - 2s}(\N, \C)$ is analytic.
\end{definition}
For any $u = \sum_{n \in \Z} \widehat u(n) e^{inx} \in B^{-s}_{c,0}(r)$, let $ u_* := \sum_{n \in \Z} |\widehat u(n)| e^{inx}$. It then follows in a straightforward way
from the definition of $f : B^{-s}_{c,0}(r) \to \C$ being normally analytic that for any $u \in B^{-s}_{c,0}(r)$,
$$
| \underline f( u) | \le \underline f( u_*) 
%=  \sum_{k=0}^{\infty} \sum_{|\alpha| = k} |f^k_\alpha| |\widehat u|^\alpha 
=  \sum_{k=0}^{\infty} \sum_{|\alpha| = k} |f^k_\alpha|  \prod_{n \in \Z} |\widehat u(n)|^{\alpha_n} < \infty \, .
$$

%%%%%%%%%%%%%%%%%%%%%%%%%%%%%%%%%%%%%%%%%%%%%%%%%%%%%%
%%%%%%%%%%%%%%%%%%%%%%%%%%%%%%%%%%%%%%%%%%%%%%%%%%%%%%

\end{document}